\pdfoutput=1
\documentclass[11pt]{article}
\usepackage{palatino}

\usepackage{etex}

\usepackage[english]{babel}
\usepackage[utf8]{inputenc}
\usepackage{fancyhdr}

\usepackage{amsmath,dsfont}
\usepackage{graphicx} 
\usepackage{caption} 
\usepackage{subcaption} 
\usepackage{amssymb}
\usepackage{amsthm}
\usepackage{authblk}
\usepackage{dcpic}
\usepackage[hang,flushmargin]{footmisc}
\usepackage[left=1in,right=1in,top=1in,bottom=1in]{geometry}
\usepackage[ocgcolorlinks,linkcolor=linkblue,citecolor=linkred,urlcolor=linkred]{hyperref}
\usepackage{url}

\usepackage{cleveref}

\usepackage{MnSymbol}
\usepackage{pictexwd}
\usepackage{tikz-cd}
\usepackage{verbatim}
\usepackage{xcolor}
\usepackage{bibentry}
\usepackage[normalem]{ulem}

\usepackage{mdframed}

\usepackage{todo}
\makeatletter
\renewcommand*\@tododisplay[0]
\makeatother
\let\TodoOriginal\Todo
\renewcommand\Todo[2][To~do]{%
  {\color{red} \TodoOriginal[#1]{\normalsize #2}}%
  \@todotrue%
}

\usepackage{JH-abbrev}

\usepackage{enumitem}
\usepackage{empheq}

\usepackage{pgf}

\makeatletter
\newcommand\pgfinvisible{\pgfsys@begininvisible}
\newcommand\pgfshown{\pgfsys@endinvisible}
\makeatother

\def\CIpidisplay{hide}

\usepackage[final]{showkeys} 
\usepackage{seqsplit} 
\usepackage{etoolbox}

\usepackage{calc} 

\newcommand{\fullrlap}[1]{%
  \rlap{\kern\dimexpr-\@totalleftmargin+\textwidth+\marginparsep\relax#1}}

\newcommand{\restatableeq}[2]{
  #2
  \gdef#1{#2}
}

\newcommand{\subtag}[1]{
  \makeatletter
  \def\@currentlabel{#1}
  \makeatother
  \renewcommand{\theequation}{#1\arabic{equation}}
}

\newcommand{\CI}{
  \mathrm{CI}
}

\usepackage{environ}
\NewEnviron{killcontents}{}

\numberwithin{equation}{section}

\definecolor{linkred}{rgb}{0.75,0,0}
\definecolor{linkblue}{rgb}{0,0,0.75}

\theoremstyle{plain}

\theoremstyle{plain}
\newtheorem{theorem}{Theorem}[section]
\newtheorem{lemma}[theorem]{Lemma}
\newtheorem{proposition}[theorem]{Proposition}
\newtheorem{corollary}[theorem]{Corollary}

\theoremstyle{definition}
\newtheorem{definition}[theorem]{Definition}
\newtheorem{rem}[theorem]{Remark}

\definecolor{kellygreen}{rgb}{0.3, 0.73, 0.09}

\def\Email#1{{
  \textit{E-mail}: #1
}}

\begin{document}

\title{\vspace{-8ex} Contour Integral Formulas for PushASEP on the Ring}

\author{Jhih-Huang Li\footnote{Department of mathematics, National Taiwan University, Taipei, Taiwan. \Email{lijhih@ntu.edu.tw}} \ and Axel Saenz\footnote{Department of mathematics, Oregon State University, Corvallis, United States. \Email{saenzroa@oregonstate.edu}}}

\date{}

\maketitle

\begin{abstract}
  We give contour integral formulas for the generating function of the joint distribution of the PushASEP on a ring.
  We obtained these formulas through a rigorous treatment of Bethe Ansatz.
  The approach relies on residue computations and controlling the location of the Bethe roots, which we achieve by partially decoupling the Bethe equations through extending the system of equations.
  Moreover, we are able to use our formulas to compute the asymptotic fluctuations for the flat and step initial conditions at the relaxation time scale.  
\end{abstract}

\section{Introduction}

\subsection{Description of the model}

The \emph{push asymmetric simple exclusion process (PushASEP)} on $\bbZ$ was introduced by Borodin-Ferrari in \cite{BorodinFerrari2008}.
Recently, in \cite{Matetski2023}, exact formulas for the PushASEP on the integer line obtained via the biorthogonalization method, generalizing the formulas from the paper \emph{KPZ fixed point} \cite{matetski2021}.
The PushASEP is a prototypical example of a determinantal interactive particle process, with an exclusion rule, in the \emph{Kardar-Parisi-Zhang (KPZ) universality class}. 
We consider the natural extension of the process on a periodic, one-dimensional lattice that we identify with $\bbZ / L \bbZ$, and call this model the \emph{PushASEP on a ring}.
The \emph{length}, or \emph{period}, of the ring is $L \in \bbZ_{\geq 1}$, and let $N \in \bbZ_{\geq 1}$, with $N \leq L$, be the number of particles on the ring.
In the results below, we take any deterministic initial configuration at time $t=0$.
The system evolves according to the following rules in time:

\begin{itemize}
  \item Each particle jumps independently to the neighboring right site at an exponential rate $p$ with \emph{exclusion} dynamics, i.e. if the target site is occupied, then the jump does not happen.
  \item Each particle jumps independently to the neighboring left site at an exponential rate $q$ with \emph{push} dynamics, i.e. if the target site is occupied, the particle at the occupied site also performs a jump to their neighboring left site with push dynamics at the same time, triggering a series of jumps of all the neighboring left particles by one step until reaching an unoccupied site.
\end{itemize}
We assume that both rates $p$ and $q$ are non-negative and normalized so that $p + q =1$. Moreover, the jumps to the right and to the left are also independent.

\subsection{Background and literature}

\label{s:historical_rmks}

The PushASEP specializes to the \emph{totally asymmetric simple exclusion process (TASEP)} by setting $(p,q)=(1,0)$.
The TASEP on the ring has been previously studied by several different authors, including Baik, Liu and Prolhac.
In a series of papers, Baik and Liu \cite{baikliu2016, baikliu2018, BaikLiu19} and Prolhac \cite{Prolhac2016, prolhac2020TASEP, prolhac2020KPZ} obtained asymptotic results for the one-point function in the relaxation time scale (i.e., $t \sim L^{3/2}$).
While both Baik--Liu and Prolhac are working with TASEP with periodic boundary conditions, the striking difference lies in the fact that Baik--Liu have an infinite state space that distinguishes the winding number of a particle \footnote{For instance, in the case of Baik--Liu, the state of a particle that has traveled around the ring exactly once is not the same as a particle that has stayed in the same location. Whereas, in the case of Prolhac, a particle that has traveled around the ring exactly one is indistinguishable from a particle that has stayed in the same location.}, and Prolhac has a finite-state space that does not distinguish that winding number of the particles.
This subtle difference leads to a different formulation of their results for the one-point function, which are only checked to be the same numerically \footnote{The claim that the results are only checked numerically is found in \cite{baikliu2018} on p.3. This claim may be outdated and needs to be checked.}.
Additionally, it is worth noting that the result obtained by Prolhac are not rigorous as they are based on two assumptions: (a) the completeness of the Bethe Ansatz, (b) the main contribution in the relaxation time scale is due to the eigenstates obtained from finite excitations of the stationary eigenstate.
In this paper, we aim to unify the work of Baik--Liu and Prolhac by looking at a joint process taking into account the state variable of the particles and the current, which is equivalent to counting the winding number.
Moreover, we generalize the result to a bidirectional model, called PushASEP model. More details will be explained in \Cref{s:preliminaries}.

\subsection{New results}\label{s:intro_results}

Here we briefly describe the PushASEP model and the new results of this paper.
A more detailed and formal definition of the model will be given in \Cref{s:PushASEP}. The proofs of the new results will be given in later sections.

The PushASEP model is a Markov process and may be described by a master equation \eqref{e:master_joint}.
However, diagonalizing the infinitesimal generator $\calH$ in the master equation is in general highly non-trivial.
In Prolhac's work \cite{Prolhac2016}, with additional assumptions on
(a) the completeness of the eigenfunctions and
(b) the decay of the eigenvalues,
the author was able to perform the one-point asymptotic analysis of the periodic TASEP model (i.e.~the periodic PushASEP model with $p=1$ and $q=0$) for flat, step and stationary initial condition with $L = 2N$.
In Baik--Liu \cite{baikliu2018}, the authors considered the periodic lattice \emph{with a winding number} to obtain a determinantal contour integral formula and performed the asymptotic analysis for the flat and step initial conditions for any fixed ratio $\rho = \tfrac{N}{L}$.

In this paper, we introduce a new formula for the transition probability of the system using a $(N+1)$ contour integral; see \Cref{t:intro_transition_prob} (with $\zeta = 1$) or Equation \eqref{e:transition_prob_N+1} in \Cref{t:transition_prob}.
Additionally, through residue computations, we give the transition probability function as a linear combination of eigenfunctions for the infenitesimal generator of the master equation mentioned above;
see \Cref{t:transition_prob_eigenbasis} (with $\zeta \rightarrow 1$) or Equation \eqref{e:transition_prob_eigenfunctions} in \Cref{t:transition_prob_simple}.
Our starting point for these results is the Bethe Asatz and the corresponding Bethe equations. Then, to make the analysis tractable, we introduce an auxiliary variable and rewrite the system of equations so that all but one of the resulting equations are decoupled. Our results, which make a precise connection between contour integral formulas and eigenfunctions arising from the Bethe Ansatz, bridge the gap between the work of Baik--Liu and Prolhac and, in general, give a rigorous treatment of the Bethe Ansatz. 

A configuration $X = (x_k)_{1 \leq k \leq N}$ encodes the positions of the $N$ particles on the ring $\bbZ / L \bbZ$.
In particular, we encode the positions of the particles so that the coordinates are in increasing order, i.e., $0 \leq x_1 < \cdots < x_N \leq L-1$.
We write $\calX_{L, N}$ for the configuration space of our system.

The \emph{global (total) current} of the system plays a crucial role in our analysis of the Bethe Ansatz for the PushASEP on the ring.
See \Cref{s:definitions} for a precise definition.
In short, we denote the global current at time $t\in \bbR_{\geq0}$ by $Q(t)$ and it measures the total net current, with particles moving locally to the right (resp., left) contributing positive (resp., negative) current, for all the particles on the ring from time zero to time $t$.
In the work of Prolhac \cite{Prolhac2016}, this observable appears explicitly and, in the work of Baik--Liu \cite{baikliu2018}, this observable appears implicitly through other auxiliary observables such as the position variables.
This is another instance showing that the results below bridge the gap between the work of Prolhac and Baik--Liu.
Moreover, as the global current is often left out in the Bethe Ansatz formulation, we emphasize that this observable was key in our analysis to make the Bethe Ansatz rigorous.

We state our precise results below, but first we introduce some notation.
Let $v + C_r$ be the positively oriented circle of radius $r > 0$ and centered at the point $v \in \bbC$.
Then, throughout the paper, we use the following notation for contour integrals,
\begin{equation}
  \label{e:intro_diffcontours}
  \diffcontoursz := \frac{1}{2 \pi \icomp} \Big( \oint_{C_{R'}} - \oint_{C_{\epsilon'}}  \Big),
  \quad \diffcontoursw := \frac{1}{2 \pi \icomp}\Big( \oint_{C_{R}} - \oint_{C_{\epsilon_1}} - \oint_{1 + C_{\epsilon_2}}  \Big),
\end{equation}
for $R, R', \epsilon', \epsilon_1, \epsilon_2 > 0$ so that $\epsilon' < R'$, $\epsilon_1 < R$ and $\epsilon_2+1 < R$.
This notation greatly simplifies many of our formulas bellow.

\begin{theorem}
  \label{t:intro_transition_prob}
  Let $X(t) \in \calX_{L,N}$ be the random configuration for the PushASEP, with $N$ particles on a ring of length $L$, at time $t \in \bbR_{\geq 0}$ and initial configuration $Y =(y_k)_{1 \leq k \leq N}\in \calX_{L,N}$. Then, the following $(N+1)$-fold contour integral formula gives the partial Fourier transform for the joint transition probability of the process,
  \begin{align}
    \label{e:intro_Fourier}
    \sum_{q\in \bbZ}\zeta^q\, \bbP_{Y}\left(X(t) = X, Q(t)=q \right) = \bigcontourintRzetaCpct ,
  \end{align}
  for $X \in \calX_{L, N}$ with $|\zeta| =1$, the functions in the integrand defined by \eqref{e:intro_eigenfunctions} and \eqref{e:intro_bethe_eqns}, and the notation for the contours defined above in \eqref{e:intro_diffcontours}.
  The radii $R, R', \epsilon', \epsilon_1, \epsilon_2 > 0$ are chosen as in \eqref{e:conditions}, see \Cref{s:delta_basis}, so that the poles of the integrand are appropriately positioned. 
\end{theorem}

Let us describe the functions in the integrand of the contour integral from the previous theorem. For any given vector $\vec{w}\in \bbC^N$, $X, Y \in \calX_{L, N}$ and $\zeta \in \bbC$, we have
\begin{equation}\label{e:intro_eigenfunctions}
  \begin{split}
    \calG_{\vec{w}}(Y;\zeta) & = \frac{h_{\vec{w}}(Y; \zeta)}{p_z( \vec{w} ; \zeta ) \prod_{i=1}^N q_z(w_i)}, \\
    u_{\vec{w}} (X;\zeta) & = \textsum_{\sigma \in S_N} A_{\sigma}(\vec{w}) \prod_{i=1}^N (\zeta^{-1} w_{\sigma (i)})^{-x_i},\\
    E(\vec{w}) & = \textsum_{i=1}^N (p w_i + q w_i^{-1} - 1)
  \end{split}
\end{equation}
where $S_N$ is the symmetric group on $N$ elements. The functions $p_z(\vec{w})$ and $q_z(w_i)$, for $1 \leq i \leq N $ arise naturally from what we call the \emph{deformed Bethe equations}, introduced later in \Cref{s:Bethe_Ansatz_equations}.
\begin{equation}\label{e:intro_bethe_eqns}
  \begin{split}
    p_z(\vec{w}; \zeta) &  = 1 + (-1)^N \zeta^L z^{-L} \textprod_{i=1}^N (1 - w_i^{-1}), \\
    q_z(w_i) & = 1 - z^L w_i^{-L} (1 - w_i^{-1})^{-N}, \qquad \forall i = 1, \cdots, N.
  \end{split}
\end{equation}
The other terms in \eqref{e:intro_eigenfunctions} are given as follows
\begin{equation}
  \begin{split}
    h_{\vec{w}}(Y; \zeta) = \prod_{i=1}^N (\zeta^{-1} w_i)^{y_i}, \quad
    A_{\sigma}(\vec{w})  =(-1)^{\sigma} \prod_{j=1}^N (1 - w_{\sigma(j)}^{-1} )^{\sigma(j)-j},
  \end{split}
\end{equation}
where $\sigma \in S_N$ and $(-1)^{\sigma}$ is the sign of the permuation.
Additionally, let us introduce the set of \emph{deformed Bethe roots},
\begin{equation}\label{e:intro_bethe_roots}
  \calR(\zeta) = \{ \vec{w} \in \bbC^N \mid p_z(\vec{w}; \zeta) = 0 \mbox{ and } q_z(w_i) = 0 \mbox{ for all } i = 1, \cdots, N \, \text{and for some } z \in \bbC^* \}.
\end{equation}
Note that the set of deformed Bethe roots is a finite set since it is given as system of equations of $N+1$ unknowns and $N+1$ equations -- it is not too difficult to show that system of equations are independent since $\zeta\in \bbC$ with $|\zeta|=1$.
Also, note that we exclude the case $z=0$ since the deformed Bethe equations \eqref{e:intro_bethe_eqns} are not well-defined in that case.
Then, through careful residue computations, we write \eqref{e:intro_Fourier}, the partial Fourier transform for the joint probability distribution of the PushASEP on a ring, as a sum over the deformed Bethe roots $\calR(\zeta)$.

\begin{theorem}\label{t:transition_prob_eigenbasis}
  The functions $u_{\vec{\lambda}}(X;\zeta)$ with $\vec{\lambda} \in \calR(\zeta)$, given by \eqref{e:intro_eigenfunctions} and \eqref{e:intro_bethe_roots}, are eigenfunctions of the infinitesimal generator for the PushASEP on the ring. The corresponding eigenvalues are $E(\vec{\lambda})$, given by \eqref{e:intro_eigenfunctions}.
  Moreover, the partial Fourier transform of the joint probability distribution \eqref{e:intro_Fourier} has the following spectral decomposition
  \begin{equation}
    \label{e:intro_finite_sum}
    \sum_{q \in \bbZ} \zeta^q \, \bbP_Y (X(t) = X, Q(t) = q)
    = {L \choose N}^{-1} \mathds{1} (\zeta^L=1)
    + \sum_{\vec{\lambda} \in \calR(\zeta) }  \calG_{\vec{\lambda}}(Y; \zeta) u_{\vec{\lambda}}(X;\zeta) e^{ t E(\vec{\lambda}) },
  \end{equation}
  for all but finite $\zeta \in \bbC$ with $|\zeta|=1$ and 
  \begin{equation}
    \calG_{\vec{\lambda}}(Y; \zeta) = h_{\vec{\lambda}}(Y; \zeta) \left(\prod_{i=1}^N\frac{\lambda_i -1} { L \lambda_i - (L-N)} \right)\left(L - \sum_{i=1}^{N}\frac{L}{ L \lambda_i - (L-N)} \right)^{-1}.
  \end{equation}
\end{theorem}

\begin{rem}
  Note that we treat the case $\zeta^L =1$ separately in \Cref{t:transition_prob_eigenbasis}.
  When $\zeta^L =1$, there is an \emph{exceptional} solution to the deformed Bethe equations \eqref{e:intro_bethe_eqns} corresponding to $z = 0$, but we exclude this solution from the deformed Bethe roots \eqref{e:intro_bethe_roots} since the Bethe functions are not well-behaved for the exceptional solution.
  One may think of the exceptional solution to the deformed Bethe equation, when $\zeta^L = 1$, as the limit of a particular solution of the deformed Bethe equations, when $\zeta^L \neq 1$ and $\zeta^L \rightarrow 1$.
  In the computations below, the exceptional solution will actually appear naturally from residue computations.
  Additionally, the exceptional solution, when $\zeta^L=1$, corresponds an eigenfunction given by the stationary distribution.
  Lastly, we note that exceptional solution described above is the only exceptional solution and that there are no exceptional solutions when $\zeta^L \neq 1$.
  Thus, the first term on the right side of \eqref{e:intro_finite_sum} is due to the exceptional solution, when $\zeta^L =1$, and it correspond to the stationary eigenfunction in the spectral decomposition.
\end{rem}

Let us now consider the asymptotic limit of the \emph{current observable}. Note that the current, as opposed to particle positions, is more suitable for asymptotic analysis due to the periodic nature of our system.
The current of the edge $(x, x+1)$ at time $ t \geq 0$, for $x \in \bbZ / L \bbZ$, measures the number of particle jumps up to time $t$ from $x$ to $x+1$ (increased by 1) and from $x+1$ to $x$ (decreased by 1).
We denote the current of the edge $(x,x+1)$ at time $t\geq 0$ as $Q_{x}(t) \in \mathbb{Z}$.

In \cite{BorodinFerrari2008}, Borodin and Ferrari studied the fluctuations of the PushASEP model on $\bbZ$ for a tagged particle at large times. They showed that the fluctuations for the position of a tagged particle are of order $t^{1/3}$ and that the Airy processes appear in the limit for both the flat and step initial conditions when the particle density is $\rho = \frac12$.
This leads one to expect similar results irrespectively of particle density and/or the initial conditions of the system.
In the periodic case, one should expect similar results under a specific scaling.
In particular, note that the PushASEP on a ring is an ergodic Markov process, which means that the system will eventually converge in distribution to its stationary measure and that the fluctuations are Gaussian.
Then, in order to uncover KPZ-like statistics from the PushASEP on a ring, we must scale the length of the ring, while keeping the density (roughly) fixed, as we take time to infinity.
Moreover, in the KPZ universality class, the correlation length typically scales on the order of $t^{2/3}$ for time $t \rightarrow \infty$; this is called the \emph{relaxation time scale}.
Then, one should expect this to be the scale when particles on opposite sides of the ring start to interact and the system starts transitioning to the stationary distribution with Gaussian fluctuations.
Similar observations have been made for the TASEP model in the works of Baik--Liu \cite{baikliu2018} and Prohlac \cite{Prolhac2016}.
Thus, we take the relaxation time scale, i.e.~$t \propto L^{3/2}$, and obtain a transition regime between KPZ and Gaussian for the fluctuation statistics of the PushASEP on the ring for flat and step intial conditions, extending the results of Baik--Liu and Prohlac.

\begin{theorem}
  \label{t:intro_asymptotic_flat_current}
  Fix an integer $d \geq 2$. Consider the PushASEP model on a ring of length $L$ and the number of particles is $N$ satisfying $L = Nd$. Additionally, take the flat initial condition
  \[
    (y_1, \dots, y_N) = (0, d, \dots, (N-1)d) \in \calX_{L, N}.
  \]
  Fix $\tau > 0$ and set
  \[
    t = \frac{\tau}{ \sqrt{ \rho (1-\rho) } } \, L^{3/2}.
  \]
  Then, the fluctuations for the current at the edge $(L-1, 0)$ are given as follows
  \[
    \lim_{L \to \infty} \bbP^{\icflat}_{L,d} \Big( \frac{Q_{L-1}(t) - vt}{ \rho^{2/3} (1-\rho)^{2/3} t^{1/3} } \geq -x \Big) = F_1(\tau^{1/3} x; r \tau),
    \quad x \in \bbR,
  \]
  where $v = v(p, q) := \rho \big[ p (1-\rho) - \frac{q}{1-\rho} \big]$ and $r = r(p, q) := p + \frac{q}{(1-\rho)^3}$.
\end{theorem}

\begin{rem}\label{r:asymptotic_flat_current}
  Let us make a few comments about the above result.
  \begin{enumerate}
    \item The limiting distribution $F_1$ in \Cref{t:intro_asymptotic_flat_current} will be defined explicitly in \eqref{e:F1_defn}.
    This distribution is exactly the same as the distribution that appears in the TASEP model in the relaxation time scale \cite{baikliu2018}, except that we have an additional multiplicative parameter, $r = r(p, q)$,  on the second argument of the distribution.
    \item The parameter $v = v(p, q)$ matches a similar parameter in the results of Baik--Liu in \cite{baikliu2018} when we specialize our result to the TASEP, i.e.~$(p,q) = (1,0)$.
    The parameter $v$ is the speed of the local current, whereas the parameter $r$ takes care of the time dilatation at the scale $L^{3/2}$. Note that both the signs in front of $p$ and $q$ are positive, which may indicate that $r$ plays a role similar to the variance.
    \item Our formulas for $v$ and $r$ also generalize similar parameters in \cite{BorodinFerrari2008}, where only $\rho = \frac12$ case is treated.
    \item For shifted flat initial conditions, the same result holds, and can be found in \Cref{t:asymptotic_flat_current}.
  \end{enumerate}
\end{rem}

We obtain a similar asymptotic result for step initial condition; see \Cref{t:asymptotic_step_current}.
We defer stating the result in order to have a more streamlined introduction since more notation is required before we state the corresponding statements and the results are a bit more complicated due to the nature of the step initial condition.
The step initial condition means that the particles are packed together at time 0, and this phenomenon persists at the relaxation time-scale, which can be seen in the corresponding statements.
More detailed discussions on this shock phenomenon will follow in \Cref{r:step_shock}.
We note that the results include the same speed parameter $v$ and the same time dilatation parameter $r$.
However, the limiting distribution is a bit different, which we call $F_2$ and also appeared in \cite{baikliu2018}.
We will introduce the distribution $F_2$ explicitly in \eqref{e:F2_defn}.

\subsection{Outline of the paper}

We introduce definitions and notations in \Cref{s:preliminaries}, for instance we give a precise definition of the PushASEP model and an introduction to the Bethe Ansatz.
In \Cref{s:results}, we state all the main results of this paper in detail, and to be more concise, we postpone the technical proofs and results to later sections.
\Cref{s:contour_deformations} contains the key technical results and proofs using contour deformation method.
\Cref{s:current} deals with the asymptotic fluctuations of the current observable at the relaxation time scale.
In \Cref{s:prob_fun_expansion}, we connect our formulas to those from \cite{baikliu2018} via series expansion arguments.
Lastly, from \Cref{a:delta_basis} to \Cref{a:last_appendix}, we establish miscellaneous results needed throughout the paper.

\section{Preliminaries}
\label{s:preliminaries}

In this section, we give more details on the PushASEP.
We begin by carefully defining the state space for the PushASEP; see \Cref{s:definitions}.
We give two equivalent and common descriptions of the state space: the occupation state space and the particle configuration state space.
Then, we give the infinitesimal generator, i.e.~the Kolmogorov backwards generator, for the PushASEP on the ring; see \Cref{s:PushASEP}.
The infinitesimal generator is more readily given using the occupation configuration notation.
We also establish the master equation for the PushASEP on the ring, giving a difference/differential equation for the joint probability distribution for the particle configuration and the current of the PushASEP.
The master equation is more readily given using the particle configuration notation.
Lastly, we introduce the Bethe Ansatz along with some relevant results; see \Cref{s:Bethe_Ansatz_equations}.
The Bethe Ansatz gives eigenfunction of the master equation for the PushASEP depending on solutions of a system of algebraic equations called the Bethe equations.

\subsection{Definitions and notations}

\label{s:definitions}

The \emph{(occupation) state space} for PushASEP on a ring of length $L\in \bbZ_{\geq1}$ is $\calX_L := \{ 0, 1 \}^{\bbZ/L\bbZ}$.
A \emph{configuration} of the PushASEP on a ring of length $L \in \bbZ_{\geq1}$ is given by a function,
\begin{align}
  \begin{array}{lccc}
    \eta : & \bbZ / L \bbZ & \rightarrow & \{ 0, 1 \} \\
           & x & \mapsto & \eta_x
  \end{array},
\end{align}
called an \emph{occupation variable}.
We say that position $x \in \bbZ / L \bbZ$ is \emph{occupied} if $\eta_x =1$ and, otherwise, the position is \emph{vacant}, i.e.~there is a particle at position $x \in \bbZ / L \bbZ$ or there is no particle, respectively.

Equivalently, we fix the number of particles, $N \leq L$, in a configuration so that we only need to consider a specific connected component of the configuration space.
Let
\begin{equation}\label{e:configuration}
  \calX_{L, N} = \{ (x_1, \cdots, x_N) \in \bbZ^N  \mid 0 \leq x_1 < x_2 < \cdots < x_N < L \},
\end{equation}
be the \emph{(position) state space} for the PushASEP on a ring of length $L \in \bbZ_{\geq 1}$ and $N \in \bbZ_{\geq 1}$ particles \footnote{The position state space $\calX_{L,N}$, given by all configurations with exactly $N$ particles, is an irreducible component of the full occupation state space $\calX_L$ due to conservation of particles in the dynamics. It is not hard to show that $X_{L} \cong \calX_{L, 0} \cup \calX_{L, 1} \cup \cdots \cup \calX_{L, N}$; the precise bijection is given by $\eta[X]$ in \eqref{e:eta_X}.}. For a position configuration $X \in \calX_{L,N}$, the corresponding occupation variable is given by 
\begin{equation}
  \label{e:eta_X}
  \eta[X]_{x} :=
  \begin{cases}
    1, \quad x \in X, \\
    0, \quad x \notin X.
  \end{cases}
\end{equation}

We introduce an additional quantity called the (oriented) \emph{current} to keep track of the evolution of the system. We distinguish two notions of current: instant current and (net) current. The former denotes the particle flow at a given time and the latter denotes the cumulative instant current up to a given time.
For $0 \leq x \leq L-1$, the \emph{instant local current} at position $x$ and time $t$ is given by
\[
  j_x(t) =
  \left\{
    \begin{array}{ll}
      1 & \mbox{ if there is a particle jumping from } x \mbox{ to } x+1, \\
      -1 & \mbox{ if there is a particle jumping from }  x+1 \mbox{ to } x, \\
      0 & \mbox{ otherwise}.
    \end{array}
  \right.
\]
Then,  \emph{instant global current}, which is the sum of $j_x(t)$ over all the positions, is given by
\[
  j(t) = \sum_{x \in \bbZ / L \bbZ} j_x(t).
\]
Similarly, for $0 \leq x \leq L-1$, the \emph{(net) local current} at position $x$ and time $t$ is given by
\begin{equation}\label{e:local_current_intro}
  Q_x(t) = \int_0^t \delta(j_x(s) = 1) - \delta(j_x(s) = -1) \dd s,
\end{equation}
which counts the number of particles jumping from $x$ to $x+1$ \emph{minus} the number of particles jumping from $x+1$ to $x$ up to time $t$.
Then, the \emph{(net) global current} is given by
\begin{equation}\label{e:global_current}
  Q(t) = \sum_{x \in \bbZ / L \bbZ} Q_{x}(t).
\end{equation}

For the PushASEP on the ring, we consider the joint evolution of the state variable and the current variable of the system, $(\eta(t), Q(t))_{t \geq 0}$, or equivalently, $(X(t), Q(t))_{t \geq 0}$, which takes values in $\calX_{L} \times \bbZ$ or $\calX_{L, N} \times \bbZ$, respectively.
This is a Markov process, which is not hard to see.
In the following, we give the corresponding Kolmogorov backwards generator and write down the master equation for the process, see \Cref{d:generator} and \Cref{e:master_joint}.

\begin{rem}
  In \cite{baikliu2018}, the authors take a different configuration space for the periodic TASEP.
  They take the particle configuration $(x_i)_{i=1}^N \in \mathbb{Z}^N$ so that $x_i < x_{i+1}$, for $i=1, \dots, N-1$, and $x_N < x_1+L$.
  See \eqref{e:configuration} for comparison with the state space we use in this work. Note that the state space in \cite{baikliu2018} is infinite, whereas our state space is finite.
  Moreover, the state space in \cite{baikliu2018} keeps track of the current implicitly, whereas we introduce the state variable $Q(t)$ to keep track of the current.
  As a consequence, in \cite{baikliu2018}, the authors first establish their results for a tagged particle and then translate their results in terms of the current through the following simple relation
  \begin{equation}
    \bbP(x_k(t) \geq \ell L +1) = \bbP(Q_0(t) \geq \ell N - k+1 )
  \end{equation}
  for all $1 \leq k \leq N$ and $\ell \in \mathbb{Z}_{\geq 0}$.
  See \cite[Equation (3.8)]{baikliu2018}.
\end{rem}

\subsection{Master equation for PushASEP}

\label{s:PushASEP}

The joint process $(\eta(t), Q(t))_{t\geq0}$ from the PushASEP is a Feller process and is uniquely determined by an initial value partial differential equation, called \emph{the master equation}, see \eqref{e:master_joint}.
The master equation is obtained from the \emph{Kolmogorov backwards generator}, see \Cref{d:generator}, with the initial conditions determined by the initial configuration of the process.
In particular, the generator gives a precise description of the dynamics for the PushASEP on a ring. More details on the general theory of interacting particle processes, including details on the Kolmogorov backwards generator and the master equation, may be found in \cite{KipnisLandim13}. Below, we give the master equation and introduce the necessary notation.

Let us start by describing the (infinitesimal) generator of the Feller process $(\eta(t), Q(t))_{t\geq 0}$. The generator is characterized by its action on the space of bounded functions on the joint state space, which we denote by $\calB(\calX_L \times \bbZ)$. The generator acts on the entries of the functions, i.e.~the configurations, by switching the position of the particles. More precisely, for a given occupation variable $\eta \in \calX_L$, denote $\eta^{x, \pm k} \in \calX_L$ to be the occupation variable with the values $\eta_x$ and $\eta_{x \pm k}$ interchanged.
That is,
\begin{equation}
  \eta^{x, \pm k}_z =
  \begin{cases}
    \eta_z, \quad z \neq x, x \pm k \\
    \eta_{x \pm k}, \quad z= x \\
    \eta_x, \quad z = x \pm k.
  \end{cases}
\end{equation}
For instance, if $\eta_x = 1$ and $\eta_{x \pm k} = 0$, the occupation variable $\eta^{x \pm k}$ corresponds to moving a particle from position $x \in \bbZ / L \bbZ$ to position $x \pm k \in \bbZ / L \bbZ$. Then, the generator of the joint process is given as follows.

\begin{definition}
  \label{d:generator}
  The joint process $(\eta(t), Q(t))$ for the PushASEP on $\bbZ / L \bbZ$ with $L \in \bbZ_{\geq1}$ at time $t \in \bbR_{\geq 0}$ with exclusion jumps to the right at rate $p \in [0,1]$, push jumps to the left at rate $q = 1-p$, and initial configuration $(\eta(0), 0) \in \calX_L \times \bbZ$ is given by the \emph{Kolmogorov backwards generator}
  \begin{align}
    \label{e:kbe}
    \calL f (\eta, Q)
    & = p \sum_{x=0}^{L-1} (1 - \eta_{x+1}) \eta_{x} [f(\eta^{x, 1}, Q+1) - f(\eta, Q)] \nonumber \\
    & \quad + q \sum_{x=0}^{L-1} \sum_{k=1}^{L-1} (1- \eta_{x-k}) \prod_{j=0}^{k-1} \eta_{x-j} [f(\eta^{x, -k}, Q-k) - f(\eta, Q)],
  \end{align}
  so that 
  \begin{equation}
    \frac{\dd}{\dd t} \bbE_{\eta(0)}[f(\eta(t), Q(t)) ] = \calL\, \bbE_{\eta(0)}[f(\eta(t), Q(t)) ],
  \end{equation}
  for any initial condition $\eta(0) \in \mathcal{X}_L$ and $f \in \calB(\calX_{L} \times \bbZ)$.
\end{definition}

Let us now establish the master equation based on \Cref{d:generator}.
Given $0 \leq N \leq L$, $X \in \calX_{L, N}$ and $\zeta \in \bbU = \{ z \in \bbC \mid |z| = 1 \}$, define $f_X(\cdot; \zeta) \in \calB(\calX_L \times \bbZ)$ as follows,
\begin{align}
  \label{e:f_X}
  f_X (\eta, Q; \zeta) := \sum_{q \in \bbZ} \zeta^q \cdot \mathds{1}( \eta = \eta[X], Q = q) = \zeta^Q \mathds{1} (\eta = \eta[X]),
\end{align}
with $\eta[X]$ given by \eqref{e:eta_X}. Note that it is important to take $\zeta \in \bbU$ so that $f_X(\cdot; \zeta)$ is a bounded function. Indeed, for a given configuration $\eta \in \calX_L$, the current is well defined up to an integer multiple of $L$ and the term infront of the indicator function may be arbitrarily large if $|\zeta| \neq 1$. Also, define the following generating series for the joint distribution,
\begin{align}
  \label{e:joint_gf}
  \QXgf := \bbE_{\eta[Y]}[ f_X(\eta(t), Q(t); \zeta) ] = \sum_{q \in \bbZ} \zeta^q \cdot \bbP_Y( X(t) = X, Q(t) = q).
\end{align}
Note that we may obtain the generating series of the current as follows,
\begin{equation}
  \label{e:current_gf}
  \Qgf := \sum_{q \in \bbZ} \zeta^q \cdot \bbP_Y( Q(t) = q) = \sum_{X \in \calX_{L, N}} \QXgf,
\end{equation}
by summing the generating series for the joint distribution over the state space. Similarly, if we take $\zeta = 1$, we recover the transition probability function,
\begin{equation}
  \label{e:prob_transition}
  g(Y, X; 1; t) = \sum_{q \in \bbZ} \bbP_Y( X(t) = X, Q(t) = q) = \bbP_Y( X(t) = X ).
\end{equation}

The \emph{master equation} for the joint process of the PushASEP is obtained by the action of the Kolmogorov backwards generator, given in \Cref{d:generator}, on the generating series $g(Y, X;\zeta;t)$, given by \eqref{e:joint_gf}. In particular, we have
\begin{equation}\label{e:master_joint}
  \begin{cases}
    \partial_t \QXgf = \Hzeta \QXgf, \\
    g(Y, X; \zeta; 0) = \mathds{1}(X=Y),
  \end{cases}
\end{equation}
so that $\zeta$-deformed linear operator $\Hzeta := p \Hzeta^T + q \Hzeta^P$ is given by
\begin{equation}
  \label{e:master_joint_operators}
  \begin{split}
    \Hzeta^T \QXgf &:= \sum_{i = 0}^{L-1} \left( \zeta \QXgf[i, -] - \QXgf \right) \mathds{1} (X^{i,-} \in \calX_{L,N}),\\
    \Hzeta^P \QXgf &:= \sum_{i=0}^{L-1} \sum_{k=1}^{N} \left( \zeta^{-k} \QXgf[i, +k] - \QXgf \right) \mathds{1} (d(x_{i+k-1}, x_i) = k-1, X^{i, +k} \in \calX_{L,N}),
  \end{split}
\end{equation}
so that the configuration $X^{i, -}$ is equal to the configuration $X$ except that the particle with label $i$ has moved one position to the left on the ring, the configuration $X^{i, +k}$ is equal to the configuration $X$ except that the particles with labels $i$ through $(i + k) (\textrm{mod}\; N)$ have moved one position to the right on the ring, and the function $d(x_{i+k-1}, x_i)$ is equal to the distance of the particles on the ring where the particle number is taken modulo $N$. More precisely, for a configuration $X = (x_1 < \cdots < x_N) \in \calX_{L,N}$, we have
\begin{equation}
  X^{i, -} =(x_1^{i,-} < x_2^{i, -} < \cdots < x_N^{i, -}) = \begin{cases}
    (x_2 , \cdots , x_N , L-1), &\quad \text{if } x_1 = 0\text{ and } i =1, \\
    (\tilde{x}_1^{i,-} ,  \tilde{x}_2^{i,-} , \cdots , \tilde{x}_N^{i,-}), &\quad \text{otherwise}, 
  \end{cases}
\end{equation}
with $\tilde{x}_j^{i,-} := x_j - \mathds{1}(j=i)$; and 
\begin{equation}
  X^{i, +k} = (x_1^{i,+k},  \cdots, x_{N-1}^{i, +k} , x_N^{i, +k}) =
  \begin{cases}
    (\tilde{x}^{i, +k}_N , \tilde{x}_1^{i, +k}, \cdots, \tilde{x}_{N-1}^{i, +k}), &\quad x_N = L-1 \text{ and } i \leq N \leq i+k-1, \\
    (\tilde{x}_1^{i,+k},  \cdots, \tilde{x}_{N-1}^{i, +k} , \tilde{x}_N^{i, +k}), &\quad \text{otherwise},
  \end{cases}
\end{equation}
with $\tilde{x}_j^{i, +k} := \left[x_j + \mathds{1}(i \leq j \leq \min\{i + k-1, N\}) + \mathds{1}(1 \leq j \leq i +k -1-N)\right] (\textrm{mod}\; L)$; and
\begin{equation}
  d(x_{i+k}, x_{i}) :=
  \begin{cases}
    L + x_{i+ k -N} - x_i, \quad i+k >N, \\
    x_{i+k} - x_i, \quad i+k \leq N.
  \end{cases}
\end{equation}
Note that the configurations $X^{i, -}$ and $X^{i, +k}$ do not necessarily lie in the configuration space $\calX_{L, N}$.
In that case, if $X^{i, -}$ and $X^{i, +k}$ do not lie in the configuration space $\calX_{L, N}$, the indicator functions are equal to zero and we treat the corresponding term equal to zero, despite the fact that generating series for the joint distribution is not defined for those configurations.
In the following section, we will extend the domain of the generating series for the joint distribution so that it is defined for $X^{i, -}$ and $X^{i, +k}$, if $X \in \calX_{L,N}$, via a some assumptions.

\subsection{Bethe Ansatz and Bethe equations}
\label{s:Bethe_Ansatz_equations}

The Bethe Ansatz was first introduced by \cite{Bethe} for the 1-dimensional Heisenberg-Ising spin-1/2 chain (aka XXZ spin-1/2 chain) from quantum mechanics, and it has since found applications in several other 1-dimensional models.
See \cite{Gaudin2014, Baxter} for a comprehensive introduction. 

The Bethe Ansatz, in short, assumes that the solution of the master equation is made up of a linear combination of solutions for the same master equation without any interactions.
This assumption comes from the infinite volume case, e.g.~the integer line $\bbZ$, when the particles are not expected to interact as time goes to infinity.
Note that, by removing the interactions, the solution for the master equation is much simpler; it is simply a product of wave functions for free particles.
In the case of a ring, there is an additional periodicity condition for the Bethe Ansatz.
Based on these assumptions, one may construct eigenfunctions for the master equation of the PushASEP given a solution to a system of algebraic equations.
These eigenfunctions are called \emph{Bethe functions} and the system of algebraic equations is called the \emph{Bethe equations}.

For the PushASEP, we first extend the domain of the function $X \mapsto \QXgf$, which is given by $\calX_{L,N}$ in \eqref{e:configuration}, to the domain,
\begin{equation}
  \overline{\calX}_{L,N} := \{ (x_1, x_2 ,\cdots, x_{N}) \in \bbZ^N \mid x_1 \leq x_2 \leq \cdots \leq x_N \leq x_1 + L \text{ with at most one equality} \},
\end{equation}
so that at most two particles may be in the same location and the ring can be shifted on the integer lattice $\bbZ$.
Then, the generating series $g(Y, X; \zeta;t)$ is extended, with respect to the $X$ configuration, from the original domain $\mathcal{X}_{L,N}$ to the domain $\overline{\calX}_{L,N}$ by the following assumptions:
\begin{itemize}
  \item (Periodicitiy rule) For all $x_1 < \cdots < x_{N} < x_1 + L$, the function remains invariant after a periodic shift of the particles,
  \begin{equation}
    \begin{split}
      \label{e:periodicity}
      g(Y, (x_1, \cdots, x_{N}); \zeta; t) & = g(Y, (x_2, \cdots, x_{N}, x_1 + L); \zeta; t).
    \end{split}
  \end{equation}  
  \item (Exclusion/push rule) For all $1 \leq k \leq N$, fix the positions $x_1< x_2< \cdots < x_{k-1} < x_{k+1} < \cdots < x_{N} < x_1 + L$ for all the particles except for the $k^{th}$ particle, where $x_0$ is identified as $x_N$. Also, denote by $X[x_k =x]$ to be the configuration given by the fixed positions and the position of the $k^{th}$ particle given by $x \in \mathbb{Z}$. Then, the function is defined for configurations where two particles are on the same site as follows,
  \begin{equation}
    \label{e:exclusion}
    g(Y,X[x_k = x_{k-1}]; \zeta;t )  = g(Y,X[x_k = x_{k-1} + 1]; \zeta;t ).
  \end{equation}
\end{itemize}

In particular, if we assume the exclusion/push rule and the periodicity rule, the linear operators \eqref{e:master_joint_operators} for the $\zeta$-deformed master equation may be rewritten in a more compact way,
\begin{equation}
  \label{e:master_operators_compact}
  \begin{split}
    \overline{\calH}_\zeta^T \QXgf & = \sum_{i = 0}^{L-1} \left( \zeta g(X^{i, -}; \zeta; t) - \QXgf \right),\\
    \overline{\calH}_\zeta^P \QXgf & = \sum_{i=0}^{L-1} \left( \zeta^{-1} g(X^{i, +}; \zeta; t) - \QXgf \right),
  \end{split}
\end{equation}
for $X = (x_k)_{1 \leq k \leq N} \in \calX_{L,N}$ and $X^{i, \pm} = (x_k \pm \mathds{1}(k=i))_{1 \leq k \leq N} \in \overline{\calX}_{L,N}$.
Note that we do not apply the operators to functions $\QXgf$ with $X \in \overline{\calX}_{L,N} \backslash \calX_{L,N}$ since the value of the function is completely determined by the periodicity rule \eqref{e:periodicity} and the exclusion/push rule \eqref{e:exclusion} for those configurations.
As a consequence, if $g(Y, X;\zeta; t)$ satisfies \eqref{e:periodicity} and \eqref{e:exclusion}, the master equation \eqref{e:master_joint} can be simplified into what we call the \emph{free equation},
\begin{equation}
  \label{e:master_joint_compact}
  \begin{cases}
    \partial_t \QXgf = \overline{\calH}_\zeta \QXgf, \\
    g(Y, X; \zeta; 0) = \mathds{1}(X=Y),
  \end{cases}
\end{equation}
where $\overline{\calH}_\zeta = p \overline{\calH}^T_\zeta + q \overline{\calH}^P_\zeta$.
Note that the free equation has the advantage that there is no particle interaction, since we do not have indicators in the formula for the operators $\overline{\calH}_\zeta^T$ and $\overline{\calH}_\zeta^P$.

We obtain eigenfunctions, called \emph{Bethe functions}, for the operator $\Hzeta$ using this Ansatz, i.e.~solutions to \eqref{e:master_joint_compact} that satsify the boundary conditions \eqref{e:exclusion} and \eqref{e:periodicity}.
The Bethe eigenfunctions are given by,
\begin{equation}
  \label{e:bethe_function_zeta}
  u_{\vec{w}} (X; \zeta)
  = \sum_{\sigma \in S_N} A_{\sigma}(\vec{w}; \zeta) \prod_{i=1}^N w_{\sigma (i)}^{-x_i},
  \quad {X \in \calX_{L, N}},
\end{equation}
with the $\zeta$-deformed \emph{amplitude coefficients} $A_{\sigma}(\vec{w}; \zeta)$ given by,
\begin{equation}
  \label{e:bethe_coeffs_zeta}
  A_{\sigma}(\vec{w}; \zeta)
  = (-1)^{\sigma} \prod_{i<j} \frac{ 1 - (\zeta w _{j})^{-1}}{ 1 - (\zeta w_{\sigma(j)})^{-1}}
  = (-1)^{\sigma} \prod_{j} \left( 1 - (\zeta w_{\sigma(j)})^{-1} \right)^{\sigma(j) - j},
\end{equation}
so that the parameters $\vec{w} \in \bbC^N$ satisfy the $\zeta$-deformed \emph{Bethe equations},
\begin{equation}
  \label{e:bethe_equations_zeta}
  w_j^L = (-1)^{N-1} \prod_{i=1}^N \frac{ 1 - (\zeta w_i)^{-1}}{ 1 - (\zeta w_j)^{-1}}, \quad j =1, \cdots, N.
\end{equation}
The form of the $\zeta$-deformed Bethe function \eqref{e:bethe_function} is due to the simpler expression for the operators $\Hzeta^T$ and $\Hzeta^P$, given by \eqref{e:master_operators_compact}, since each summand of the Bethe function is an eigenfunction of the free operators.
Still, we must satisfy the exclusion/push rule and the periodicity rule, or boundary conditions, given by \eqref{e:exclusion} and \eqref{e:periodicity}.
Thus, the $\zeta$-deformed amplitude coefficients $A_{\sigma}$, given by \eqref{e:bethe_coeffs}, are determined by the exclusion/push rule, and the Bethe equations, given by \eqref{e:bethe_equations}, are determined the periodicity rule.
Moreover, we also have a detrminantal formula for the $\zeta$-deformed Bethe function,
\begin{equation}\label{e:bethe_determinantal}
  u_{\vec{w}} (X; \zeta) = \det \left[ (1 - (\zeta w_j)^{-1})^{j-i} w_j^{-x_i} \right]_{1 \leq i, j \leq N}, \quad {X \in \calX_{L, N}}.
\end{equation}

We may set $\zeta = 1$ in the definitions given in \Cref{e:master_joint_operators,e:exclusion,e:periodicity,e:master_operators_compact,e:bethe_function,e:bethe_coeffs,e:bethe_equations}, in which case we can simply drop the notation $\zeta$ from the formula and remove the prefix ``$\zeta$-deformed'' in the corresponding name.

Consider the change of variables
\begin{equation}
  \label{e:zeta_CoVs}
  w_{i, \zeta} := \zeta w_i, \quad i = 1, \dots, N,
\end{equation}
and $\vec{w}_\zeta = (w_{i, \zeta})_{1 \leq i \leq N}$.
We note that the $\zeta$-deformed amplitude coefficients can be rewritten in terms of the non-deformed amplitude coefficients, corresponding to $\zeta = 1$,
\begin{equation}
  \label{e:bethe_coeffs}
  A_{\sigma}(\vec{w}; \zeta)
  = (-1)^{\sigma} \prod_{i<j} \frac{ 1 -  w_{j, \zeta}^{-1}}{ 1 - w_{\sigma(j), \zeta}^{-1}}
  = A_{\sigma}(\vec{w_{\zeta}}; 1)
  =: A_{\sigma}(\vec{w_{\zeta}}).
\end{equation}
Similarly, the $\zeta$-deformed Bethe functions can be rewritten as,
\begin{equation}
  \label{e:bethe_function}
  u_{\vec{w}} (X; \zeta) = \sum_{\sigma \in S_N} A_{\sigma} (\vec{w}_\zeta) \prod_{i=1}^N \zeta^{x_i} w_{\sigma (i), \zeta}^{-x_i}
  = \zeta^{\sum x_i} \cdot u_{\vec{w}_\zeta} (X), \quad {X \in \calX_{L, N}},
\end{equation}
where $u_{\vec{w}_{\zeta}}(X) := u_{\vec{w}_\zeta} (X; 1)$, $X \in \calX_{L, N}$. Finally, after the change of variables, we write the $\zeta$-deformed Bethe equations as follows
\begin{equation}
  \label{e:bethe_equations}
  \zeta^{-L} w_{j, \zeta}^L = (-1)^{N-1} \prod_{i=1}^N \frac{ 1 - w_{i, \zeta}^{-1}}{ 1 - w_{j, \zeta}^{-1}}, \quad j =1, \cdots, N.
\end{equation}
Then, instead of the $\zeta$-deformed Bethe equations in \eqref{e:bethe_equations_zeta}, we refer to \eqref{e:bethe_equations} when we talk about the $\zeta$-deformed Bethe equations, since it has a simpler form after the change of variables. In the following, we will also drop the $\zeta$ subscript when it is clear from the context.

\begin{proposition}\label{p:Bethe_Ansatz}
  Fix $\zeta \in \bbU$.
  Let $\vec{\lambda} = (\lambda_1, \cdots, \lambda_N) \in \bbC^N$ be such that $\vec{\lambda}_\zeta$ is a solution to the $\zeta$-deformed Bethe equations \eqref{e:bethe_equations}.
  Then, we have 
  \begin{equation}
    \Hzeta u_{\vec{\lambda}}(X; \zeta) = E(\vec{\lambda}; \zeta) \,  u_{\vec{\lambda}}(X; \zeta),
    \quad
    E(\vec{\lambda}; \zeta)
    : = \sum_{i=1}^N (p \zeta \lambda_i + q \zeta^{-1} \lambda_i^{-1} - 1)
    =: E(\vec{\lambda}_\zeta).
  \end{equation}
  with the Bethe function $u_{\vec{\lambda}}(\cdot; \zeta)$ given by \eqref{e:bethe_function_zeta} (or equivalently, $u_{\vec{\lambda}_\zeta}$ given by \eqref{e:bethe_function}) and the operator of the master equation of the PushASEP given in \eqref{e:master_joint_operators}.
\end{proposition}

\begin{rem}
  Note that the solutions $\vec{\lambda} \in \bbC^N$ to the Bethe equations \eqref{e:bethe_equations} \emph{depend on} $\zeta$.
  Also, note that some of these solutions lead to trivial Bethe functions \eqref{e:bethe_function}, i.e., $u_{\vec{\lambda}} = 0$.
  In fact, one can show that any solution with $\lambda_i = \lambda_j \neq 1$, for $i \neq j$, leads to a trivial Bethe function using the determinantal structure \eqref{e:bethe_determinantal}.
  In particular, the statement of \Cref{p:Bethe_Ansatz} is trivial for the trivial Bethe functions.
  As a result, we refrain from saying that the Bethe functions are eigenfunctions of the operator $\Hzeta$ in the statement of \Cref{p:Bethe_Ansatz} since some of the Bethe functions may be identically equal to zero.
\end{rem}

\begin{proof}
  To prove the above \Cref{p:Bethe_Ansatz}, we want to check that the exclusion/push rule \eqref{e:exclusion} and the periodicity condition \eqref{e:periodicity} are satisfied for $u_{\vec{\lambda}}$ if $\vec{\lambda}_{\zeta} \in \bbC^N$ is a solution to the Bethe equations \eqref{e:bethe_equations}.
  This is checked in \Cref{l:exclusion_check} and \Cref{l:periodicity_check} from \Cref{app:Bethe_properties}.
  Then, due to the exclusion/push rule and the periodicity condition, we have
  \begin{equation}
    \calH_\zeta^T\, u_{\vec{\lambda}} = \overline{\calH}_\zeta^T\, u_{\vec{\lambda}}, \quad \calH_\zeta^P\, u_{\vec{\lambda}} = \overline{\calH}_\zeta^P\, u_{\vec{\lambda}},
  \end{equation}
  with $\overline{\calH}_\zeta^T$ and $\overline{\calH}_\zeta^T$ given by \eqref{e:master_operators_compact}. Therefore, \Cref{p:Bethe_Ansatz} follows from the straightforward computation below,
  \begin{align*}
    \Hzeta u_{\vec{\lambda}}(X; \zeta)
    & = \overline{\calH}_\zeta u_{\vec{\lambda}}(X; \zeta)
      = p \overline{\calH}_\zeta^T u_{\vec{\lambda}}(X; \zeta) + q \overline{\calH}_\zeta^P u_{\vec{\lambda}}(X; \zeta) \\
    & = \sum_{i=0}^{L-1} \sum_{\sigma \in S_N} \left(
      p (\zeta \lambda_{\sigma(i)} - 1)
      + q (\zeta^{-1} \lambda_{\sigma(i)}^{-1} - 1) 
      \right)
      A_\sigma(\vec{\lambda}; \zeta) \prod_{i=j}^N \lambda_{\sigma(j)}^{-x_j} \\
    & = \sum_{\sigma \in S_N} \sum_{i=0}^{L-1} \left(
      p (\zeta \lambda_{\sigma(i)} - 1)
      + q (\zeta^{-1} \lambda_{\sigma(i)}^{-1} - 1) 
      \right)
      A_\sigma(\vec{\lambda}; \zeta) \prod_{j=1}^N \lambda_{\sigma(j)}^{-x_j} \\
    & = E(\vec{\lambda_\zeta}) u_{\vec{\lambda}}(X; \zeta),
  \end{align*}
  where in the second line, we directly apply the definitions in \eqref{e:master_operators_compact}.
\end{proof}

\subsection{Decoupled Bethe equations}

We decouple the $\zeta$-deformed Bethe equation \eqref{e:bethe_equations} by introducing an auxiliary variable and a corresponding auxiliary equation.
The coupling among the variables in the Bethe equations will shift completely to the auxiliary equation.
This will give us better control of the solutions to the Bethe equations in future formulas.

The decoupled $\zeta$-deformed \emph{Bethe equations} are given by,
\begingroup
\renewcommand{\theequation}{dBE}
\begin{subequations}
  \subtag{dBE}
  \label{e:dBE}
  \begin{align}
    q_z(w_i) & := 1 - z^L w_i^{-L} (1 - w_i^{-1})^{-N} = 0 , \quad i= 1, \cdots, N, \label{e:dBE1} \\
    p_z(\vec{w}; \zeta) &  := 1 + (-1)^N \zeta^L z^{-L} \textprod_{i=1}^N (1 - w_i^{-1})= 0. \label{e:dBE2} 
  \end{align}
\end{subequations}
\endgroup
Note that $w_i$ variables are decoupled in the first $N$ equations, given by \eqref{e:dBE1}, which are also independent of $\zeta$.
Moreover, the coupling of the $w_i$ variables and the dependency on $\zeta$ are only due to the last equation \eqref{e:dBE2}.
Note that what we call $w_i$ here, and in the rest of the paper, are actually $w_{i, \zeta}$ in \eqref{e:bethe_equations}. Here, we see again how the change of variables introduced in \eqref{e:zeta_CoVs} is useful, since it allows us to remove the $\zeta$ dependency in \eqref{e:dBE1}.

The system of equations given by the decoupled Bethe equations \eqref{e:dBE1} and \eqref{e:dBE2} is equivalent to the system of equations given by the Bethe equations \eqref{e:bethe_equations}, except for the case when $z=0$ in the decoupled Bethe equations. The exceptional point $z=0$ will be important for our analysis and will give rise to the probability of the stationary distribution. The reason for the disparity in the case $z=0$ is technical. When $z=0$, the solutions to the decoupled Bethe equations include coordinates with $w_i =0$ or $w_i=1$, where the denominator of the Bethe equations vanishes. The solutions of the decoupled Bethe equations with $z=0$ should be thought as solutions of the Bethe equations when the singularities of the system of equations are resolved. 

Let us denote the set of solution to the Bethe equation \eqref{e:bethe_equations} and the decoupled Bethe equations \eqref{e:dBE1} and \eqref{e:dBE2} as follows,
\begin{align}
  \calR(\zeta) & := \left\{ \vec{w} \in \bbC^N \mid \vec{w} \text{ is a solution of \eqref{e:bethe_equations}} \right\}, \label{e:bethe_roots}\\
  \calQ (z) & := \left\{ w \in \bbC \mid q_z(w) = 0 \right\}, \quad z \in \bbC, \label{e:bethe_roots_deformed}\\
  \calP(z; \zeta) & := \left\{ \vec{w} =(w_1, \cdots, w_N) \in \bbC^N \mid p_z( \vec{w}; \zeta ) = 0 \right\}, \quad z \in \bbC, \zeta \in \bbU \label{e:coupled_roots}. 
\end{align}
We say that $\calR(\zeta)$ is the set of $\zeta$-deformed \emph{Bethe roots} and $\calQ(z)$ is the set of \emph{decoupled Bethe roots}.
Additionally, we denote $\calZ(\zeta)$ to be the set of solutions in $z\in \bbC$ to the system of decoupled $\zeta$-deformed Bethe equations \eqref{e:dBE1} and \eqref{e:dBE2},
\begin{equation}
  \label{e:calZ_def}
  \calZ(\zeta) = \{ z \in \bbC \mid \calQ^N(z) \cap \calP(z; \zeta) \neq \emptyset \}.
\end{equation}

We have a decomposition of the Bethe roots,
\begin{equation}\label{e:bethe_roots_decomp}
  \calR(\zeta)   = \bigcup_{\substack{z \in \calZ(\zeta)\\ z \neq 0}}\left(  \calQ^N(z) \cap \calP (z; \zeta) \right).
\end{equation}
Note that we have excluded the value $z = 0$ on the right side. This value, and only this value, leads to solutions $\vec{w}$ with $w_i = 1$ for some $i \in \{1, \cdots, N\}$ when $\zeta$ is an $L$-th root of unity.
On the left side, any vector $\vec{w}$ with $w_i = 1$ for some $i \in \{1, \cdots, N\}$ is, technically speaking, excluded since the denominator in \eqref{e:bethe_equations} vanishes.
Thus, we do not compare the solution for the value $z=0$ and simply exclude the $z=0$ case, dealing with it separately when necessary.
Nonetheless, the set of solutions $\calQ(0) \cap \calP(0; \zeta)$, when $\zeta$ is an $L$-th root of unity, will be included in our analysis as these solutions act as solutions of the Bethe equations \eqref{e:bethe_equations} when the singularities from the denominators have been resolved.
Some more details will be given in \Cref{l:residue_computations} when we need to compute the residue at this particular solution.

\subsection{Delta basis}
\label{s:delta_basis}

We represent a collection of delta functions as nested contour integrals, see \Cref{p:delta_contour_formula}.
The delta functions we consider are the delta functions arising in the master equation \eqref{e:master_joint}, as the initial condition.
We will then use the nested contour integral representation in the solution we give to the master equation.  

The \emph{delta basis} for the PushASEP with $N$ particles on a ring of length $L$ is given by the set of indicator functions $\mathds{1}_{Y}$ so that,
\begin{equation}
  \mathds{1}_{Y}(X) = \mathds{1}(Y=X) =
  \begin{cases}
    1, \quad Y = X, \\
    0, \quad Y \neq X, \end{cases}
\end{equation}
for any $X, Y \in \calX_{L,N}$. Note that the set $\{\mathds{1}_X \mid X \in \calX_{L,N}\}$ of delta basis has cardinality ${L \choose N}$, equal to the cardinality of the state space $\calX_{L,N}$ and the dimension of the space of functions $\calB (\calX_{L,N})$ on the state space. We will write the delta functions as linear combination of Bethe functions, see \Cref{t:bethe_expansion}.

Let us introduce some notation before we give the representation of the delta functions as nested contour integral. First, we introduce the function
\begin{equation}\label{e:integrand_fun}
  h_{\vec{w}}(Y) := \prod_{i=1}^N w_i^{y_i},
\end{equation}
for any $\vec{w} = (w_1, \cdots, w_N) \in \bbC^N$ and $Y = (y_1, \cdots, y_n) \in \bbZ^N$.
Also, we use the following notation, already mentioned in \eqref{e:intro_diffcontours}, for the contour intergrals
\begin{align}
  \label{e:diffcontours}
  \diffcontoursz & = \frac{1}{2 \pi \icomp} \Big( \oint_{C_{R'}} - \oint_{C_{\epsilon'}}  \Big),
  &   \diffcontoursw & = \frac{1}{2 \pi \icomp} \Big( \oint_{C_{R}} - \oint_{C_{\epsilon_1}} - \oint_{1 + C_{\epsilon_2}}  \Big),
\end{align}
where, for $r > 0$, the contour $C_r$ denotes the positively oriented circle of radius $r$ around the origin, and $1+C_{r}$ denotes its shift by 1.
The radii $R', \epsilon'$, $R, \epsilon_1, \epsilon_2 > 0$ for the different contours are chosen to satisfy the following conditions,
\begingroup
\renewcommand{\theequation}{Cond}
\begin{subequations}
  \subtag{Cond}
  \label{e:conditions}
  \begin{empheq}[left = \empheqlbrace]{align}
    & \epsilon' \ll 1 \text{ is arbitrarily small and } R' = 1/\epsilon', \label{e:z_contour} \\
    & R = (R')^{\beta}, \epsilon_1 = (\epsilon')^{\beta_1}, \epsilon_2 = (\epsilon')^{\beta_2}, \label{e:powers}  \\
    & \beta > 1,  \label{e:beta} \\
    & \tfrac{1}{1-\rho} = \tfrac{L}{L-N}< \beta_1 <d = \tfrac{L}{N}, \label{e:beta1_bounds}\\
    &  d  < \beta_2 < \tfrac{L}{N-1}. \label{e:beta2_bounds}
  \end{empheq}
\end{subequations}
\endgroup

\begin{rem}
  \label{r:2nl}
  The conditions \eqref{e:conditions} are due to technical issues in the proofs below, mostly to control the location of the roots of the decoupled Bethe equations.
  The reason for the exact conditions will become apparent throughout the proofs; see \Cref{s:contour_deformations}, and in particular, the proof of \Cref{l:contour_deformation_induction}.
  Also, note that we must have $2 N <L$ due to \eqref{e:beta1_bounds} since $(1-\rho)^{-1} < L/N = \rho^{-1}$. This is technical condition that may be relaxed so that all the results also hold for $2N = L$ by altering the radius of the contours. The alternate conditions are given by \eqref{e:conditions_p}, in \Cref{r:conditions_p}, for the radius of the contours for the case $2N = L$. In particular, \Cref{l:contour_deformation_induction} and \Cref{l:CI_zero} are the key lemmas where the contours need to be altered for the case $2N = L$. Then, all the rest of the results also follow for $2N=L$. Moreover, via the so-called particle-hole duality, one can treat all the cased for general $0 \leq N \leq L$.
\end{rem}

We have the following representation of the delta functions, whose proof will be given in \Cref{a:delta_basis}.

\begin{proposition}\label{p:delta_contour_formula} 
  Fix some an initial configurations $Y= (y_1, \cdots, y_N) \in \calX_{L,N}$ for the PushASEP with $N$ particles on a ring of length $L$, and take $X = (x_1, \cdots, x_N)\in \calX_{L,N}$.
  Then, we have
  \begin{equation}
    \label{e:initial_condition}
    \mathds{1}(X=Y) = \bigcontourintRiczeta
  \end{equation}
  with the functions $p_z, q_z$ given by \eqref{e:dBE}, the function $h_{\vec{w}}$ given by \eqref{e:integrand_fun}, the function $u_{\vec{w}}$ given by \eqref{e:bethe_function}, and the radii of contours $R', \epsilon', R, \epsilon_1, \epsilon_2$ satisfying \eqref{e:conditions}.
\end{proposition}

\section{Results}
\label{s:results}

We present the more technical results in this section now that we have introduced some background in \Cref{s:preliminaries}. We begin with \Cref{s:zeta_deformed} where we give three alternate formulas for the generating series for the joint distribution $\QXgf$ in \Cref{t:main}. In the next section, \Cref{s:transition_prob}, we specialize the generating series for the joint distribution to obtain the transition probability of the PushASEP on a ring, see \Cref{t:transition_prob} and \Cref{t:transition_prob_simple}. Next, in \Cref{s:BL_transition_prob}, we give a decomposition of the generating series for the joint distribution, in \Cref{p:method_images}, as a sum of terms with shifted initial conditions. It follows, by \Cref{c:method_images}, that our formulas for the generating series for the joint distribution are compatible with the transition probability formulas given by Baik--Liu in \cite{baikliu2018}. Then, in \Cref{s:current_results}, we establish a result connecting the local and total current, \Cref{l:global_local}, and we use this result, along with previous results, to establish a contour integral formula for the cumulative distribution function of the local current at site $L-1$ in \Cref{p:current_cdf}. In the last section, \Cref{s:asymptotic_results}, we give asymptotic results for the local current under flat and step initial conditions; see \Cref{t:asymptotic_flat_current} and \Cref{t:asymptotic_step_current}. In \Cref{s:results}, we only present the results and leave the proof of the results to the following sections so that we present all of the major results in a unified manner and so that the work in the following sections is better motivated.

\subsection{\texorpdfstring{$\zeta$}{Zeta}-deformed formula}
\label{s:zeta_deformed}

We recall that the generating series for the joint distribution introduced in \eqref{e:joint_gf} is given by
\[
  \QXgf := \bbE_{\eta[Y]}[ f_X(\eta(t); \zeta) ] = \sum_{q \in \bbZ} \zeta^q \cdot \bbP_Y( X(t) = X, Q(t) = q).
\]
It specializes to the transition probability, as explained in \eqref{e:prob_transition}, when we take $\zeta \rightarrow 1$.
Moreover, the current introduced in \Cref{s:definitions} can be recovered from the same generating series, after a summation over the state space and an inverse Fourier transform,
\begin{align}
  \label{e:current_gf2}
  \Qgf & := \sum_{q \in \bbZ} \zeta^q \cdot \bbP_Y( Q(t) = q) = \sum_{X \in \calX_{L, N}} \QXgf, \\
  \label{e:current_proba}
  \bbP_Y(Q(t) = q) & = \oint_{C_1} \frac{\dd \zeta}{\zeta} \frac{\Qgf}{\zeta^q}, \quad q \in \bbZ,
\end{align}
so that the contour $C_1$ is the unit circle with positive orientation.
The following main theorem gives different equivalent formulas of the generating series $\QXgf$. Later, we see that these formulas lead to precise asymptotic results; see \Cref{t:asymptotic_flat_current} for flat initial conditions and \Cref{t:asymptotic_step_current} for step initial conditions.

\begin{theorem}
  \label{t:main}
  Fix an initial configuration $Y = (y_1, \cdots, y_N) \in \calX_{L,N}$ for the PushASEP with $N$ particles on a ring of length $L$, and let $X= (x_1, \cdots, x_N)  \in \calX_{L,N}$ be a random configuration at time $t \in \bbR_{\geq 0}$.
  Then, for all but finite values of $\zeta \in \mathbb{U}$, the generating series for the joint distribution, defined by \eqref{e:joint_gf}, is given by the following three equivalent formulas: an $(N+1)$-fold contour integral, a 1-fold contour integral with an additional continuous function $u_0 : \bbU \longrightarrow \bbC$, and the expansion on Bethe functions of the operator $\Hzeta$ in the master equation \eqref{e:master_joint}.
  In particular, we have
  \begin{subequations}
    \begin{align}
      \label{e:main_N+1}
      \QXgf
      & = \bigcontourintRzeta\\
      \label{e:main_1}
      & = u_0(\zeta) + \CIpi{1} \diffcontoursZ \sumNbetherootszeta, \\  
      \label{e:main_eigenfunctions}
      &  = u_0(\zeta) + \eigenexpansionzeta,
    \end{align}
  \end{subequations}
  where the last equality \eqref{e:main_eigenfunctions} only holds for generic $\zeta \in \bbU$ (i.e., all but finitely many $\zeta \in \bbU$), we take the convention given by \eqref{e:diffcontours} for the contour integrals, the radii $R', \epsilon', R, \epsilon_1, \epsilon_2$ satisfy \eqref{e:conditions}, $u_0$ is a continuous function on $\bbU$ that is given later in \Cref{l:u0_residue}, the functions $p_z$ and $q_z$ are given by \eqref{e:dBE1} and \eqref{e:dBE2}, the set $\calQ(z)$ is given by \eqref{e:bethe_roots_deformed}, the set $\calR(\zeta)$ is given by \eqref{e:bethe_roots}, the function $u_{\vec{w}}$ is given by \eqref{e:bethe_function}, the function $E$ and the coefficient $\alpha_{\vec{\lambda}}$ are given by
  \begin{align*}
    & & & E(\vec{\lambda}) = \sum_{i=1}^N (p \lambda_i + q \lambda_i^{-1} - 1), 
    & & \alpha_{\vec{\lambda}}^{-1}
        = r(\vec{\lambda})
        \prod_{i=1}^N \lambda_i q'(\lambda_i), \\
    \mbox{with } 
    & & & r(\vec{\lambda}) = L - \sum_{i=1}^{N}\frac{L}{ L \lambda_i - (L-N)},
    & & \lambda_i q'_z(\lambda_i) = \frac{ L \lambda_i - (L-N) }{\lambda_i -1}.
  \end{align*}
\end{theorem}

\begin{rem}
  There are three formulas in \Cref{t:main} which \emph{all} describe the same quantity, i.e. the generating series.
  The first formula \eqref{e:main_N+1} is given by a $(N+1)$-fold contour integral formula, where variables of contour integrals arise from the decoupled $\zeta$-deformed Bethe equations \eqref{e:dBE1} and \eqref{e:dBE2}.
  In \Cref{p:w_residues}, we obtain the 1-fold contour integral formula \eqref{e:main_1} by integrating the $w_i$ variables out from the $N+1$-fold contour integral in \eqref{e:main_N+1}.
  Next, in \Cref{t:bethe_expansion}, we obtain an finite sum expansion over eigenfunctions of the operator $\Hzeta$ by integrating the remaining $z$ variable out from the 1-fold contour integral.
  In this case, we need a generic condition on the $\zeta$ variable, as described in the statement of \Cref{t:main}, to make sure that the poles in the integrand are simple, making the residue computations correct.
  Finally, we also check that these three different formulas describe the generating function, which is discussed below, in the proof of \Cref{t:main}.
\end{rem}

\begin{proof}
  We focus on establishing the equality with the generating series.
  The second equality, between the $(N+1)$-fold contour integral and the 1-fold contour integral, is established in \Cref{p:w_residues}, and the third equality, between the 1-fold contour integral and the finite sum, is established in \Cref{t:bethe_expansion}.
  Thus, we treat the three expressions on the right side to be equivalent.
  Below, we check that the three equivalent expressions give the generating series.
  
  Let us denote the $(N+1)$-fold contour integral on the right side of \eqref{e:main_N+1} by $g_1 (Y, X; \zeta; t)$.
  Similarly, we denote the right side of \eqref{e:main_1} and the right side of \eqref{e:main_eigenfunctions} by $g_2 (Y, X; \zeta; t)$ and $g_3 (Y, X; \zeta; t)$, respectively.
  Then, we have
  \begin{equation}
    g_1 (Y, X; \zeta; t) = g_2 (Y, X; \zeta; t) = g_3 (Y, X; \zeta; t)
  \end{equation}
  by \Cref{p:w_residues} and \Cref{t:bethe_expansion}, respectively. Note that \Cref{p:last_res} only applies for all but finite values of $\zeta\in \mathbb{U}$, meaning that the rest of the arguments only apply for all but finite values of $\zeta$. We denote,
  \begin{equation}
    g(Y, X; \zeta; t) := g_1 (Y, X; \zeta; t) = g_2 (Y, X; \zeta; t) = g_3 (Y, X; \zeta; t),
  \end{equation}
  for the common value of all three expressions.
  Below, we show that $g(X; \zeta; t)$ is indeed the generating series for the PushASEP on the ring defined in \eqref{e:joint_gf}.
  
  The generating series is uniquely determined by the master equation \eqref{e:master_joint}.
  In particular, we show that the function $g(Y, X; \zeta; t)$ satisfies the two conditions for the master equation: (1) the differential equation and (2) the initial conditions.
  The differential equation is most readily checked using the finite sum expression, i.e.~$g_3 (Y, X; \zeta; t)$, and the initial conditions are most readily checked using the $(N+1)$-fold contour integral, i.e.~$g_1 (Y, X; \zeta; t)$. We then check the conditions for the master equations on different expression since all three expressions are equal.
  
  First, we check the initial condition for the master equation using the $(N+1)$-fold contour integral formula. The initial condition
  \begin{equation*}
    g(Y, X; \zeta; 0)= g_1 (Y, X; \zeta; t) = \mathds{1}_{Y}(X)
  \end{equation*}
  is due to \Cref{p:delta_contour_formula}.
  This establishes the second part of the master equation \eqref{e:master_joint}.
  
  The first part of the master equation,
  \begin{equation}
    \frac{\dd}{\dd t} g(Y, X; \zeta; t) = \Hzeta \, g(Y, X; \zeta; t),
  \end{equation}
  follows from linearity and \Cref{p:Bethe_Ansatz}. The operator $\left( \Hzeta - \dd / \dd t \right)$ commutes with summations due to linearity.
  So, we take the operator $\left( \Hzeta - \dd / \dd t \right)$ inside the sum for the expression of $g^{(3)}(Y, X; \zeta; t)$.
  Then, by \Cref{p:Bethe_Ansatz}, we have
  \begin{equation}
    \left( \Hzeta - \dd / \dd t \right) \left( u_{\vec{\lambda}} (X; t) e^{t E(\vec{\lambda})} \right) = 0.
  \end{equation}
  This establishes the first part of the master equation \eqref{e:master_joint}. 
  
  We have then shown that the function $g(Y, X; \zeta; t)$ satisfies the master equation \eqref{e:master_joint} by the arguments above. Therefore, $g(Y, X; \zeta; t)$ is indeed the generating series for the PushASEP on a ring, and the result follows.
\end{proof}

\subsection{Transition probability}
\label{s:transition_prob}

Similar to the generating series in \Cref{t:main}, \Cref{t:transition_prob} and \Cref{t:transition_prob_simple} gives three equivalent formulas of the transition probability.
In particular, we note that, in the case $(p, q) = (1, 0)$, \eqref{e:transition_prob_1} closely resembles \cite[Equation (5.4)]{baikliu2018}.
We will have a more precise connection to the formula in \cite{baikliu2018} in \Cref{s:BL_transition_prob}.
The third formula \eqref{e:transition_prob_eigenfunctions} can be viewed as the decomposition over the ``eigenbasis'', similar to the various formulas given by Prolhac, see Equations (1) and (2) in \cite{Prolhac2016} for example.

\begin{corollary}
  \label{t:transition_prob}
  Take the same conditions and the same notations as in \Cref{t:main}.
  Then, the transition probability is given by
  \begin{subequations}
    \label{e:transition_prob}
    \begin{align}
      \label{e:transition_prob_N+1}
      \bbP_Y(X(t)= X)
      & = \bigcontourintRzetaOne\\
      \label{e:transition_prob_1}
      & = u_0 + \CIpi{1} \diffcontoursZ \sumNbetherootszetaOne,
    \end{align}
  \end{subequations}
  where $X= (x_1, \cdots, x_N) \in \calX_{L, N}$ and $u_0 = u_0(1) = {L \choose N}^{-1}$.
\end{corollary}

\begin{proof}
  This result follows from \Cref{t:main} and the definition of $\QXgf$ given by \eqref{e:joint_gf}. Thus, we have
  \begin{equation}
    \bbP_Y(X(t)= X) = \lim_{\zeta \rightarrow 1} \sum_{q \in \bbZ} \zeta^q \cdot \bbP_Y( X(t) = X, Q(t) = q) = \lim_{\zeta \rightarrow 1} \QXgf
  \end{equation}
  by the definition of $\QXgf$. Then, the result follows by \Cref{t:main}. In particular, we take the limit $\zeta \rightarrow 1$ for equations \eqref{e:main_N+1} and \eqref{e:main_1}. In both of these equations, we may take take the limit inside of the integrals and compute the limit simply by evaluating the integrand at $\zeta =1$. The limit of $u_0(\zeta)$ as $\zeta \rightarrow 1$ is more tricky, and we establish this in \Cref{l:u0_residue}. Thus, from the limit computations, we obtain \eqref{e:transition_prob_N+1} and \eqref{e:transition_prob_1}. 
\end{proof}

\begin{rem}
  The limit of \eqref{e:main_eigenfunctions} as $\zeta \rightarrow 1$ is more subtle. In particular, it is difficult to control the term $r(\vec{\lambda})$, given in the statement of \Cref{t:main}. For instance, this term may vanish as $\zeta \rightarrow 1$, meaning that the term $\alpha_{\lambda}$ blows up as $\zeta \rightarrow 1$. If we can rule this case out, then we are able to take the limit $\zeta \rightarrow 1$; see \Cref{t:transition_prob_simple} below. On the other hand, if $r(\vec{\lambda}) \rightarrow 0$ as $\zeta \rightarrow 1$, we expect that there are cancellations from multiple terms in the summation in \eqref{e:main_eigenfunctions}, leading to a finite limit. Intuitively, this would follow by taking the limit $\zeta \rightarrow 1$ in \eqref{e:main_1} and then obtaining a finite sum expansion by residue computation, instead of the other way around. Moreover, one may show that the case where $r(\vec{\lambda}) \rightarrow 0$ as $\zeta \rightarrow 1$ corresponds to case when the roots of $p_z(\lambda; 1)$ have multiplicity higher than one, independent of the $(p,q)$ parameters. This means that the residue computation for the one dimensional integral \eqref{e:main_1} with $\zeta=1$ is not straightforward. In short, limit of \eqref{e:main_eigenfunctions} as $\zeta \rightarrow 1$ may be more involved than just setting $\zeta =1$, but it is not clear at the moment. 
  
\end{rem}

\begin{corollary}
  \label{t:transition_prob_simple}
  Take the same conditions and the same notations as in \Cref{t:main}. Additionally, assume that the roots of $p_z(\vec{\lambda}; 1)$ are simple, for $\vec{\lambda} = (\lambda_1, \dots, \lambda_N) \in \calQ(z)$ so that the elements of $\vec{\lambda}$ are pairwise distinct. Then, the transition probability is given by
  \begin{equation}\label{e:transition_prob_eigenfunctions}
    \bbP_Y(X(t)= X) = u_0 + \eigenexpansionzetaOne
  \end{equation}
  where $X= (x_1, \cdots, x_N) \in \calX_{L, N}$ and $u_0 = u_0(1) = {L \choose N}^{-1}$.
\end{corollary}

\begin{proof}
  The result follows by evaluating the contour integral \eqref{e:transition_prob_1} from \Cref{t:transition_prob}.
  First, note that the integrand vanishes for $\lambda_i = \lambda_j$ with $i\neq j$ due to the determinantal structure of the integrand, see \eqref{e:bethe_determinantal}.
  Then, we only need to consider the integral for $\vec{\lambda} \in \calQ(z)$ with pairwise distinct elements.
  Moreover, by the assumption of the Corollary, it follows that the integrand of \eqref{e:transition_prob_1} only has simple poles at the roots of $p_z(\vec{\lambda}; 1)$; \Cref{l:residue_computations} rules out higher order poles for $z\neq 0$.
  The result then follows by residue computations at the the roots of $p_z(\vec{\lambda}; 1)$ give by \Cref{l:residue_simple}.
\end{proof}

\subsection{Comparison to TASEP on ring}
\label{s:BL_transition_prob}

We show that the formula \eqref{e:transition_prob_1} for the transition probability function that we obtain is compatible to the formula obtained by Baik--Liu in \cite{baikliu2018} for the periodic TASEP, when we set our parameters to $(p, q) = (1, 0)$; see \Cref{c:method_images}.
Recall the difference bewteen the setting in \cite{baikliu2018} and our setting: we do not keep track of the winding number in our particle configuration space and Baik--Liu implicitly keep track of the winding number in their particle configuration space.
As a result, in \cite{baikliu2018}, the probability function is not periodic but, instead, the probability function does satisfy a periodic boundary condition.
More specifically, in \cite{baikliu2018}, the first and last particle are coupled to be at most a distance $L$ apart, but the particles may be anywhere on the integer lattice.
Thus, we show that our formulas are compatible by making the formula in \cite{baikliu2018} periodic or, equivalently, by taking the marginal distribution where the winding number is forgotten.
In particular, we make the formula in \cite{baikliu2018} periodic by taking sums over shifted configurations, a procedure often called the \emph{method of (mirror) images}. We give more precise details below.

Let us introduce the following function:
\begin{equation}\label{e:prob_fun_winding}
  \begin{split}
    u^{\rm BL(p,q)}(Y, X; t) := 
    & \CIpi{1} \oint_{C_{R'}} \frac{\dd z}{z} \sum_{\SumLambdas}
    \frac{ \sum_{\sigma \in S_N} A_{\sigma} \prod_{i =1}^N \left(  \lambda_{\sigma(i)}^{y_{\sigma(i)} - x_i} e^{tE(\lambda_i)} \right)}
    { \prod_{i =1}^N \left( \frac{L \lambda_i - (L-N)}{\lambda_i - 1} \right)}\\
    = & \CIpi{1} \oint_{C_{R'}} \frac{\dd z}{z} \sum_{\SumLambdas}
    \det \Bigg[ \frac{1}{L} \frac{ (1- \lambda_j^{-1})^{j-i+1}  \lambda_j^{y_j - x_i +1} e^{tE(\lambda_i)}}
    { \lambda_j - (1-\rho) } \Bigg]_{i,j=1}^N ,
  \end{split}
\end{equation}
for any $X, Y \in \bbZ^N$ so that $\mathcal{Q}(z)$ is given by \eqref{e:bethe_roots_deformed}, $E(\lambda) = p \lambda + q \lambda^{-1} -1$, and the contour is the counter-clockwise circle centered at the origin with radius $R' >0$. We note that $u^{\rm BL(1,0)}$ equals the transition probability distribution $\bbP^{\rm BL}$, from Baik--Liu in \cite[Equation (5.4)]{baikliu2018} (after a change of varaibles $\lambda \rightarrow w+1$ in the integrand). Also, define the $m^{th}$ periodic shift of for a particle configuration
\begin{equation}
  \label{e:Y_m_shift}
  Y^{m} = (y^m_i)_{i=1}^N := (y_{N-j+1} - (k+1) L  , \dots, y_{N} - (k+1)L, y_{1} - kL, \dots, y_{N-j} - kL)
\end{equation}
with $m = N k + j$ so that $k \in \bbZ$ and $0 \leq j \leq N-1$.

\begin{proposition}
  \label{p:method_images}
  Let $g(Y,X; \zeta;t)$ be the generating series for the joint probability function, given in \eqref{e:joint_gf}, for the PushASEP on a ring of length $L$, with $N$ particles so that $2N\leq L$, and initial conditions given by $Y =(y_i)_{i=1}^N \in \calX_{L,N}$.
  Then, we have the expansion
  \begin{equation}
    g(Y, X; \zeta;t) = \sum_{m=-\infty}^{\infty} u^{\rm BL(p,q)}(Y^m ,X; t) \prod_{i=1}\zeta^{x_i - y^{m}_i}
  \end{equation}
  so that $u^{\rm BL(p,q)}$ is given by \eqref{e:prob_fun_winding} and $Y^{m} = (y^m_i)_{i=1}^N$ is the $m^{th}$ periodic shift of $Y$ given by \eqref{e:Y_m_shift}.
\end{proposition}

In \Cref{p:method_images}, we give an additional series expansion of the generating series for the joint distribution; cf. \eqref{e:joint_gf}. The proof of this result requires some technical lemmas based on series expansions of the integrand for contour integral formula of $g(Y, X; \zeta;t)$ given by \eqref{e:main_1}. We provide the lemmas, with proofs, and the proof of this result in \Cref{s:prob_fun_expansion}.

Now, consider the PushASEP  with $N$ particles where the first and last particles are conditioned to be at most distance $L$ apart, i.e.~$x_N < x_1 +L$.The configuration space is given by the configuration space of Baik--Liu in \cite{baikliu2018}
\begin{equation}
  \label{e:config_space_BL}
  \calX^{\rm BL}_{L, N} = \{ (x_1, \dots, x_N) \in \bbZ^N \mid x_1 < x_2 < \dots < x_N < x_1 + L \}.
\end{equation}
We call this the \emph{periodic PushASEP (with winding)} and the denote the corresponding transition probability distribution by $\bbP^{\rm BL(p,q)}$, where $p, q=1-p \in [0,1]$ are the right and left jump rates.
We give a formula for the transition probability distribution in the following result.

\begin{corollary}\label{c:prob_fun_winding}
  The transition probability of the periodic PushASEP with winding, as described above, is given by \eqref{e:prob_fun_winding}. That is,
  \begin{equation}\label{e:prob_PPASEP}
    \bbP^{\rm BL (p,q)}_{Y}(X(t) = X) = u^{\rm BL(p,q)}(Y, X; t).
  \end{equation}
\end{corollary}

\begin{proof}
  Without loss of generality, we may assume that $X \in \mathcal{X}_{L,N}$ due to translation invariance on the integer lattice. Now, define the subspace of configurations
  \begin{equation}
    \left( \calX_{L,N}^2 \times \mathbb{Z} \right)_s : = \left\{(Y,X, Q) \in \calX_{L,N}^2 \times \mathbb{Z} \mid Q = \sum_{i=1}^N x_i - y_i^{\ell} \text{ for some } \ell \in \mathbb{Z}\right\},
  \end{equation}
  where $Y^{\ell} = (y_i^{\ell})_{i=1}^N$ is $\ell^{th}$ periodic shift given by \eqref{e:Y_m_shift}. Note that $\left( \calX_{L,N}^2 \times \mathbb{Z} \right)_s \subset \calX_{L,N}^2 \times \mathbb{Z}$ is the (time $t>0$) support of the joint probability distribution for the PushASEP on the ring, meaning that the joint probability distribution for the PushASEP on the ring is non-zero if and only if the configuration is in $\left( \calX_{L,N}^2 \times \mathbb{Z} \right)_s$ for $t>0$.
  Then, we have a bijection
  \begin{equation}\label{e:bijection_push}
    \left\{
      \begin{array}{ccc}
        \left( \calX_{L,N}^2 \times \mathbb{Z} \right)_s & \longrightarrow & \calX^{\rm BL}_{L, N} \times \calX_{L, N} \\
        (Y, X, Q) & \mapsto & (Y^{m(Q)}, X)
      \end{array}\right.,
  \end{equation}
  The function $m(Q)$ is defined so that
  \begin{equation}
    Q = \sum_{i=1}^N x_i - y^m_i 
  \end{equation}
  where $m = m(Q)$ and $Y^{m} = (y^m_i)_{i=1}^N$ is the $m^{th}$ periodic shift. The bijection \eqref{e:bijection_push} induces an identification between the PushASEP on the ring and the periodic PushASEP with winding. Then, it follows that
  \begin{equation}
    \bbP^{\rm BL(p,q)}_{Y^m}(X(t) = X) = \bbP_{Y}(X(t) = X, Q(t) =Q),
  \end{equation}
  where $\bbP_{Y}$ is the transition probability of the PushASEP on a ring. Now, by taking the definition of $g(Y, X; \zeta; t)$ (see \eqref{e:joint_gf}), \Cref{p:method_images} and \eqref{e:prob_PPASEP}, we have
  \begin{equation}
    \sum_{m=-\infty}^{\infty} u^{\rm BL(p,q)}(Y^m ,X; t) \prod_{i=1}^N\zeta^{x_i - y^{m}_i}
    = \sum_{m=-\infty}^{\infty} \bbP^{\rm BL(p,q)}_{Y^m}(X(t) =X) \prod_{i=1}^N \zeta^{x_i - y^{m}_i}.
  \end{equation}
  Thus, the result follows by matching coefficients of both series expansions.
\end{proof}

Finally, we establish the method of images for the PushASEP on the ring. In particular, take a formula for the PushASEP with winding number, e.g.~the transition probability distribution. Then, we make the function periodic by taking the sum over all the configurations where the particles are shifted to mirror images while maintaining relative order. We call this the \emph{method of (mirror) images}, based on a method to solve differential equations with certain symmetric boundary conditions. Then, in the next result, we show that the method of images indeed gives the transition probability distribution of the PushASEP on the ring from the transition probability distribution of the periodic PushASEP (with winding). 

\begin{corollary}\label{c:method_images}
  Let $\bbP_Y(X(t)=X)$ be the transition probability distribution for PushASEP on the ring of length $L$ with (no winding number,) $N$ particles, and initial conditions $Y= (y_i)_{i=1}^N \in \calX_{L,N}$ at time $t \in \mathbb{R}_{\geq 0}$ so that $X= (x_i)_{i=1}^N \in \calX_{L,N}$. Similarly, let $\bbP^{\rm BL (p,q)}$ be the transition probability distribution of the PushASEP with winding, as in \cite{baikliu2018} when $(p,q)=(1,0)$.  Then, we have
  \begin{equation}
    \bbP_Y(X(t)=X) = \sum_{m = -\infty}^{\infty} \bbP^{\rm BL(p,q)}_{Y^m}(X(t) =X)
  \end{equation}
  so that $Y^{m} = (y^m_i)_{i=1}^N $ is the $m^{th}$ periodic shift of $Y$ given by \eqref{e:Y_m_shift}.
\end{corollary}

\begin{proof}
  This result follows from \Cref{p:method_images}, \Cref{c:prob_fun_winding}, and the definition of $\QXgf$ given by\eqref{e:joint_gf}. 
  
  Recall that we have
  \begin{equation}
    \bbP_Y(X(t)= X) = \lim_{\zeta \rightarrow 1} \sum_{q \in \bbZ} \zeta^q \cdot \bbP_Y( X(t) = X, Q(t) = q) = \lim_{\zeta \rightarrow 1} \QXgf
  \end{equation}
  by the definition of $\QXgf$. Then, the result follows by taking the limit $\zeta \rightarrow 1$ for the result in \Cref{p:method_images} and taking the identification given by \eqref{e:prob_PPASEP}. Additionally, note that the summation is absolutely convergent, meaning that we can take the limit inside the summation.
\end{proof}

\subsection{Current}
\label{s:current_results}

In this section, we consider the local current $Q_x(t)$ and the global current $Q(t)$, introduced in detail in \eqref{e:local_current_intro} and \eqref{e:global_current}. In \Cref{t:main}, we have already given a formula for the generating series of the joint distribution of the particle configuration and global current for the PushASEP on the ring. In the next result, \Cref{l:global_local}, we show in \eqref{e:local_current} that the global and the local current are equal, up to a bounded shift. This result, along with \Cref{t:main}, leads to an exact formula for the local current $Q_{L-1}(t)$.

\begin{lemma}\label{l:global_local}
  At time $t \geq 0$, conditioned on the configuration $X(t) = X := (x_1, \cdots, x_N) \in \calX_{L, N}$, the local current $Q_j(t)$, with $0 \leq j \leq L-1$, and the total current $Q(t)$ have the following relation,
  \begin{equation}
    \label{e:local_current}
    Q_{j}(t) = \frac{Q(t)}{L} - \frac{1}{L} \sum_{i=1}^N (  [x_i]_j - [y_i]_j  ),
  \end{equation}
  where we define $[x_i]_j$ to be an integer between $0$ and $L-1$ which is the position $x_i$ in a \emph{shifted} ring where the origin is set to be at $j$.
  In other words,
  \begin{equation*}
    [x_i]_j = \begin{cases}
      x_i - (j+1), &  \mbox{ if } x_i \geq j+1, \\
      x_i - (j+1) + L,  & \mbox{ if } x_i \leq j. 
    \end{cases}
  \end{equation*}
  Equivalently, \eqref{e:local_current} is given by
  \begin{equation}
    \label{e:local_current2}
    Q_{j}(t) = \frac{Q(t)}{L} - \frac{1}{L} \sum_{i=1}^N (  x_i - y_i  ) - \# \{ 1 \leq i \leq N : x_i \leq j \} + \# \{ 1 \leq i \leq N : y_i \leq j \}.
  \end{equation}
  In particular, for $j = L-1$, one has
  \begin{equation}
    \label{e:local_current_L-1}
    Q_{L-1}(t) = \frac{Q(t)}{L} - \frac{1}{L} \sum_{i=1}^N (  x_i - y_i  ).
  \end{equation}
  
\end{lemma}

\begin{proof}
  We give the proof for the case $j=L-1$. Other cases work similarly.
  
  First, we tag the particles so that the particle starting at position $y_{i}$ ends up at position $x_{\tau(i)}$ at time $t \geq 0$, for some $\tau: \{1, \dots, N\} \rightarrow \{ 1, \dots, N \}$. 
  We now establish a relation between local currents $Q_{L-1}(t)$ and $Q_j(t)$, for any $0 \leq j \leq L-1$.
  Note that the difference in the local current for two positions on the ring is at most $N$, since the particles are moving on a ring.
  For instance, consider the local current at positions $j$ and $L-1$.
  If a particle makes a full cycle around the ring, the difference for the local current for each position will remain the same, i.e.~$Q_{L-1}(t)- Q_j(t)$ will be the same.
  On the other hand, if a particle does not go a full cycle around the ring, the difference in the total current for positions $j$ and $L-1$ may differ by $\pm1$, which depends on the starting and final positions of particle.
  Then, we have
  \begin{equation}
    Q_{L-1}(t) = Q_j(t) + \sum_{i=1}^N \left( \mathds{1}(0 < x_{\tau(i)} < j \leq y_i \leq L-1) - \mathds{1}(0 < y_i < j \leq x_{\tau(i)} \leq L-1)\right).
  \end{equation}
  Now, recall that the total current is equal to the sum of all local currents.
  Then, we have
  \begin{equation}
    \begin{split}
      L\, Q_{L-1}(t) & = \sum_{j=0}^{L-1} Q_j(t) + \sum_{j=0}^{L-1} \sum_{i=1}^N \left( \mathds{1}(0 < x_{\tau(i)} < j \leq y_i \leq L-1) - \mathds{1}(0 < y_i < j \leq x_{\tau(i)} \leq L-1)\right) \\
      & = Q(t) + \sum_{i=1}^N \sum_{j=0}^{L-1}\left( \mathds{1}(0 < x_{\tau(i)} < j \leq y_i \leq L-1) - \mathds{1}(0 < y_i < j \leq x_{\tau(i)} \leq L-1)\right) \\
      & = Q(t) + \sum_{i=1}^N \left( (y_i -x_{\tau(i)}) \mathds{1}(y_i < x_{\tau(i)}) - (x_{\tau(i)} - y_i)\mathds{1}(y_i < x_{\tau(i)})\right)\\
      & = Q(t) - \sum_{i=1}^N (x_{\tau(i)} - y_i)\\
      & = Q(t) - \sum_{i=1}^N (x_{i} - y_i). 
    \end{split}
  \end{equation}
  Note that the result is independent of the tag $\tau$.
  This establishes the result. 
\end{proof}

Define the generating function $\QjXgf{j}{}$ for the joint distribution of the particle configuration and the local current at position $i \in \{0, 1, \dots, L-1\}$ for the PushASEP on the ring as follows,
\begin{align}
  \label{e:local_joint_gf}
  \QjXgf{j}{} & = \sum_{q \in \bbZ} \zeta^q \cdot \bbP_Y( X(t) = X, Q_j(t) = q) \quad \mbox{ and} \\
  \label{e:local_current_gf}
  \Qjgf{j}{} & = \sum_{q \in \bbZ} \zeta^q \cdot \bbP_Y( Q_j(t) = q) = \sum_{X \in \calX_{L, N}} \QjXgf{j}{}.
\end{align}

\begin{corollary}
  \label{c:QjXgf_contour_integral}
  Take the same conditions and the same notation as in \Cref{t:main}.
  Then, we have the following relation,
  \begin{equation}\label{e:cont_local_x}
    \QjXgf{L-1}{L} = \zeta^{- \sum (x_i - y_i)} \QXgf,
  \end{equation}
  between $g^{(L-1)}(X; \zeta^L; t)$, the generating function for the local current at position $L-1$ with configuration $X \in \mathcal{X}_{L,N}$ given by \eqref{e:local_joint_gf}, and $\QXgf$, the generating function for the global current with configuration $X \in \mathcal{X}_{L,N}$ is given by \eqref{e:joint_gf}.
\end{corollary}

\begin{proof}
  The result follows from \Cref{l:global_local}.
  Consider the generating function of the global current \eqref{e:joint_gf} and the local current \eqref{e:local_joint_gf}, together with the identity \eqref{e:local_current_L-1}. Then, we have
  \begin{align*}
    \QjXgf{L-1}{L} & = \bbE_Y[ \zeta^{LQ_{L-1}(t)} \mid X(t) = X] \cdot \bbP_Y( X(t) = X ) \\
                   & = \bbE_Y[ \zeta^{Q(t) - \sum (x_i - y_i)} \mid X(t) = X] \cdot \bbP_Y( X(t) = X ) \\
                   & = \zeta^{- \sum (x_i - y_i)} \QXgf.
  \end{align*}
  
\end{proof}

In the next result, we are able to combine the formulas from \Cref{t:main} and \Cref{c:QjXgf_contour_integral} to obtain a formula for the cumulative distribution function for the local current at position $L-1$.
The proof of the following result is given in \Cref{s:cdf_proof}.

\begin{proposition}
  \label{p:current_cdf}
  The cumulative distribution function for the local current at position $L-1$ is given by the following two formulas. When $Q \in \bbZ$, we have
  \begin{align}
    \label{e:current_cdf1}
    \bbP_Y(Q_{L-1}(t) \geq Q)
    = (-1)^{(N+1)Q}\oint_{\epsilon'} \frac{\dd z}{z^{1 +Q L }}
    \sum_{\SumLambdas}
    \det \Bigg[ \frac{(1 -\lambda_j^{-1})^{Q+j-i} \lambda_j^{y_j+1}e^{t E(\lambda_j)}}{L \lambda_j -(L-N)} \Bigg]_{i,j=1}^N,
  \end{align}
  with $\mathcal{Q}(z)$ given by \eqref{e:bethe_roots_deformed}.
  Moreover, if $Q$ is a multiple of $N$, then we also have the following alternative expression
  \begin{align}
    \label{e:current_cdf2}
    \bbP_Y(Q_{L-1}(t) \geq Q)
    = (-1)^{{N-1 \choose 2}Q} \oint_{C_{\epsilon'}} \frac{\dd z}{z}
    \sum_{\SumLambdas}
    \det \Bigg[ \frac{1}{L} \frac{ (1- \lambda_j^{-1})^{j-i}  \lambda_j^{y_j +1 - Q d} e^{ t E(\lambda_j) }}{ \lambda_j - (1-\rho) } \Bigg]_{i,j=1}^N.
  \end{align}
\end{proposition}

\begin{rem}
  We note that the above formula \eqref{e:current_cdf2} takes a very similar form as the cumulative distribution function for the position observable in the setting of Baik--Liu, see \cite[Equation (6.1)]{baikliu2018}.
\end{rem}

\subsection{Asymptotic results}
\label{s:asymptotic_results}

We introduce the two probability distribution functions, $F_1 = F^\icflat$ and $F_2 = F^\icstep$, that first appeared in \cite{baikliu2018}.
We will show that the current observable can be described by these two probability distributions in the scaling limit when $N/L$ is fixed, $t \propto L^{3/2}$, and $L \rightarrow \infty$;  see \Cref{t:asymptotic_flat_current} and \Cref{t:asymptotic_step_current}.

\subsubsection{Flat initial condition}

Define the following function
\begin{equation}
  \label{e:F1_defn}
  F_1(x; \tau) = \oint e^{x A_1(z) + \tau A_2(z) + A_3(z) + B(z)} \det \big( I - \calK_z^\icflat \big) \frac{\dd z}{z}, \qquad x \in \bbR,
\end{equation}
for $\tau > 0$, where the other terms in the formula are described in the following.
The functions $A_1(z)$, $A_2(z)$, and $A_3(z)$ are the polylogarithms,
\begin{equation}
  \label{e:A_i(z)}
  A_1(z) := - \frac{1}{\sqrt{2\pi}} \Li_{3/2}(z), \quad
  A_2(z) := - \frac{1}{\sqrt{2\pi}} \Li_{5/2}(z), \quad
  A_3(z) := - \frac{1}{4} \log(1-z),
\end{equation}
and the function $B(z)$ is given by
\begin{equation}
  \label{e:B(z)}
  B(z) := \frac{1}{4\pi} \int_0^z \frac{ (\Li_{1/2}(y))^2}{y} \dd y,
\end{equation}
where the path integral for $B$ is taken to be any curve in $\bbC \backslash \bbR_{\geq 1}$.
Note that these four functions are analytic in $z \in \bbC \backslash \bbR_{\geq 1}$.
The operator $\calK_z^\icflat$ acts on $\ell^2(\calS_{-}(z))$, i.e.~the space of finite $\ell^2$-norm sequences on $\calS_{-}(z)$, where $\calS_{-}(z)$ is given by
\begin{equation}
  \label{e:S-}
  \calS_{-}(z) = \{ \xi \in \bbC \mid e^{-\xi^2/2} = z, \myre(\xi) < 0 \}.
\end{equation}
Additionally, before giving an explicit expression of the kernel $\calK_z^\icflat$, let us introduce the following function
\begin{equation}
  \label{e:Psi_z}
  \Psi_z(\xi; x, \tau) := - \frac13 \tau \xi^3 + x \xi - \frac{1}{\sqrt{2\pi}} \int_{-\infty}^\xi \Li_{1/2}(e^{-\omega^2/2}) \dd \omega, \quad \arg(\xi) \in (\tfrac34 \pi, \tfrac54 \pi),
\end{equation}
where the integration path from $-\infty$ to $\xi$ should be understood as going from $-\infty + \icomp 0$ to $\xi$ lying in the sector with $\arg(w) \in (\tfrac34 \pi, \tfrac54 \pi)$.
Then, we define the kernel $\calK_z^\icflat$ by
\begin{equation}
  \calK_z^\icflat(\xi_1, \xi_2) = \calK_z^\icflat(\xi_1, \xi_2; x, \tau) = \frac{ \exp( \Psi_z(\xi_1; x, \tau) + \Psi_z(\xi_2; x, \tau) ) }{ \xi_1(\xi_1 + \xi_2) }, \quad \xi_1, \xi_2 \in \calS_-(z).
\end{equation}

Following the discussion from \cite[Section~4.1]{baikliu2018} and \cite[Theorem~8.1 and 8.2]{BLS-limiting-F}, the function $F_1$ is well-defined and satisfies the following properties; see \Cref{fig:F1_CDF} for a plot.
\begin{enumerate}[label=(\alph*)]
  \item For each $\tau > 0$, the function $x \mapsto F_1(x; \tau)$ is a probability cumulative distribution function.
  Moreover, it is continuous in $\tau > 0$.
  \item For each $x \in \bbR$, $\lim_{\tau \to 0} F_1(\tau^{1/3} x; \tau) = F_{\text{GOE}}(2^{2/3} x)$ where $F_{\text{GOE}}$ is the Tracy-Widom GOE distribution.
  \item For each $x \in \bbR$,
  \begin{equation}
    \lim_{\tau \to \infty} F_1 \big( -\tau + \frac{\pi^{1/4}}{\sqrt{2}} \tau^{1/2} x ; \tau \big) = \frac{1}{\sqrt{2\pi}} \int_{-\infty}^x e^{-y^2/2} \dd y.
  \end{equation}
\end{enumerate}

\begin{figure}[htb!]
  \centering
  \includegraphics[scale=0.5]{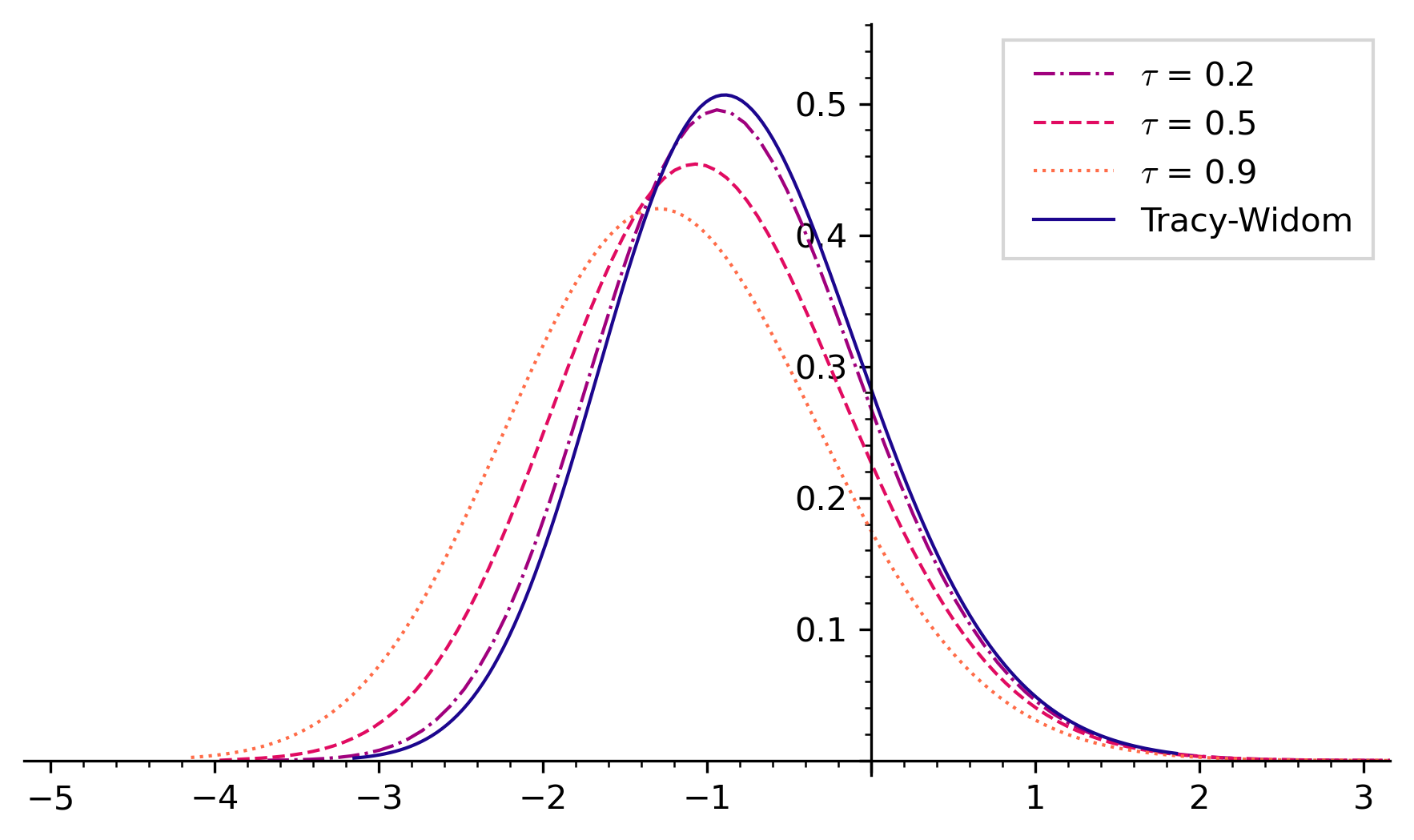}
  \hfill
  \includegraphics[scale=0.5]{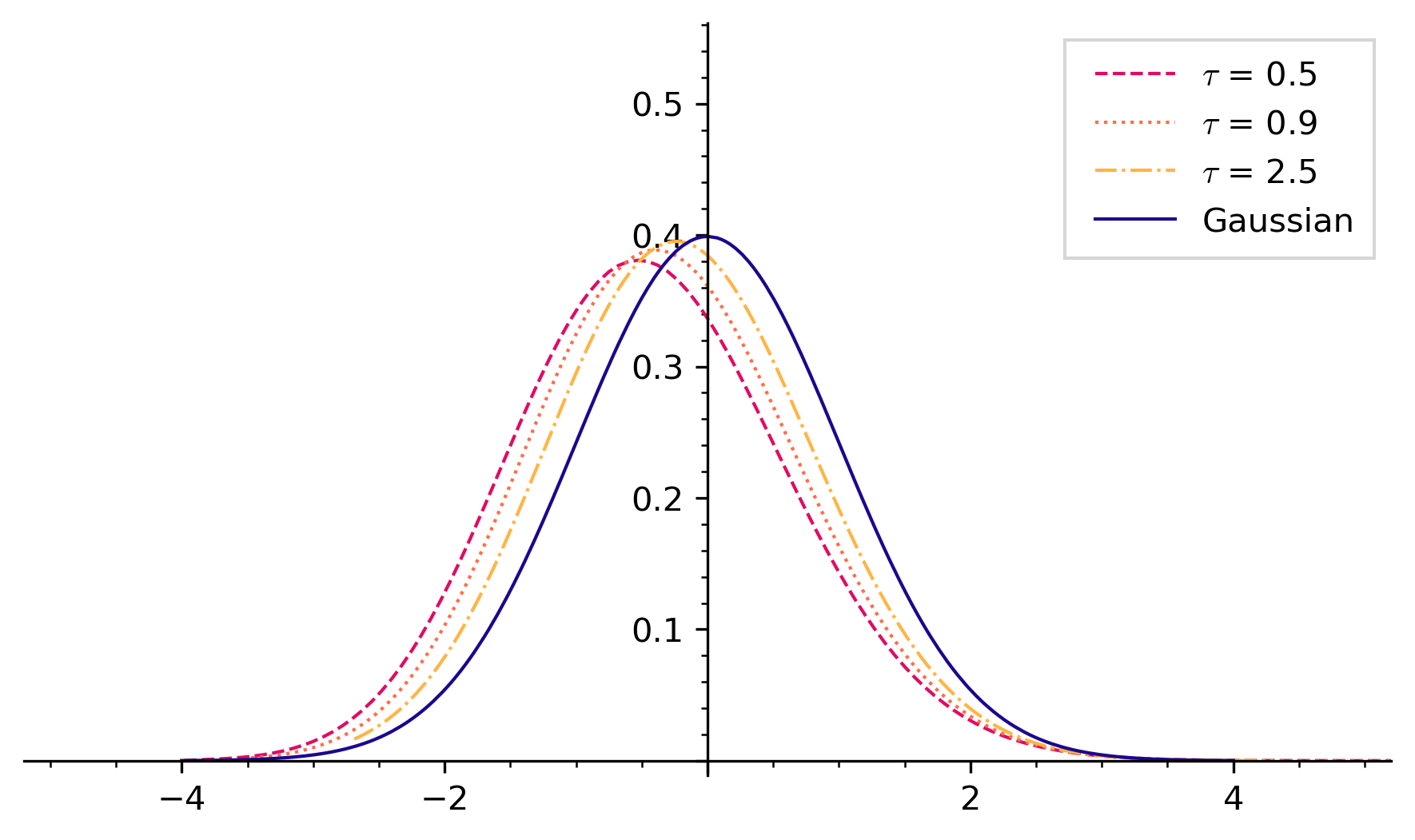}
  \caption{\textbf{Left:} The three curves from the left to right correspond to the probability distribution function of $F_1(\tau^{1/3} x; \tau)$ for the values $\tau = 0.9$, $0.5$ and $0.2$. The solid curve is the distribution of $F_{\text{GOE}}(2^{2/3} x)$.
    \textbf{Right:} The three curves from the left to right correspond to the probability distribution function of $F_1( -\tau + \frac{\pi^{1/4}}{\sqrt{2}} \tau^{1/2} x ; \tau)$ for the values $\tau = 0.5$, $0.9$ and $2.5$. The solid curve is the distribution of the standard Gaussian density.}
  \label{fig:F1_CDF}
\end{figure}

\begin{theorem}
  \label{t:asymptotic_flat_current}
  Fix an integer $d >2$. Take the PushASEP on the ring of length $L$ with $N$ particles so that $L= Nd$.
  Additionally, assume that we take the flat initial condition
  \[
    (y_1, \dots, y_N) = (\delta, d+\delta, \dots, (N-1)d+\delta) \in \calX_{L, N} \quad \text{for some integer} \, 0 \leq \delta \leq d-1.
  \]
  Fix $\tau > 0$ and set
  \[
    t = \frac{\tau}{ \sqrt{ \rho (1-\rho) } } \, L^{3/2}.
  \]
  Then, the asymptotic fluctuations of the local current are given as follows 
  \[
    \lim_{L \to \infty} \bbP^{\icflat}_{d, \delta} \left( \frac{Q_{L-1}(t) - vt}{ \rho^{2/3} (1-\rho)^{2/3} t^{1/3} } \geq -x \right) = F_1(\tau^{1/3} x; r \tau),
    \quad x \in \bbR,
  \]
  where $v = v(p, q) := \rho \big[ p (1-\rho) - \frac{q}{1-\rho} \big]$ and $r = r(p, q) := p + \frac{q}{(1-\rho)^3}$.
\end{theorem}

In \Cref{s:intro_results}, we had already introduced a version of this result, c.f.~\Cref{t:intro_asymptotic_flat_current}, where we had set $\delta =0$ for simplicity since the asymptotic fluctuations are independent of $\delta$. Additionally, we note that \Cref{r:asymptotic_flat_current}, for \Cref{t:intro_asymptotic_flat_current}, also applies to this result.

\subsubsection{Step initial condition}

Define the following function
\begin{equation}
  \label{e:F2_defn}
  F_2(x; \tau, \gamma) = \oint e^{x A_1(z) + \tau A_2(z) + 2 B(z)} \det \big( I - \calK_z^\icstep \big) \frac{\dd z}{z}, \qquad x \in \bbR.
\end{equation}
for $\tau > 0$ and $\gamma \in \bbR$. The functions $A_1(z)$, $A_2(z)$ and $B(z)$ are given in \eqref{e:A_i(z)} and \eqref{e:B(z)}. Similarly, as in the previous section, the operator $\calK_z^\icstep$ also acts on $\ell^2(\calS_1(z))$, where $\calS_1(z)$ is given in \eqref{e:S-}.
The kernel $\calK_z^\icstep$ is expressed as follows
\begin{equation}
  \begin{split}
    \xi_1, \xi_2 \in \calS_-(z), \quad
    & \calK_z^\icstep(\xi_1, \xi_2) = \calK_z^\icstep(\xi_1, \xi_2; x, \tau, \gamma) \\
    & \quad = \sum_{\eta \in \calS_{-}(z)} \frac{ \exp( \Psi_z(\xi_1; x, \tau) + \Psi_z(\xi_2; x, \tau) + \frac{\gamma}{2} (\xi_1^2 - \eta^2) ) }{ \xi_1 \eta (\xi_1 + \eta) (\xi_2 + \eta) },
  \end{split}
\end{equation}
where
\begin{equation}
  \label{e:Phi_z}
  \Phi_z(\xi; x, \tau) := - \frac13 \tau \xi^3 + x \xi - \sqrt{\frac{2}{\pi}} \int_{-\infty}^\xi \Li_{1/2}(e^{-\omega^2/2}) \dd \omega, \quad \arg(\xi) \in (\tfrac34 \pi, \tfrac54 \pi),
\end{equation}
which is exactly the same as \eqref{e:Psi_z} in the flat case, except that the integral part has its coefficient doubled.

Following the discussion from \cite[Section~4.2]{baikliu2018} and \cite[Theorem~1.5 and 1.6]{BLS-limiting-F}, the function $F_2$ is well defined and satisfies the following properties; see \Cref{fig:F2_CDF_gamma_05} for a plot.
\begin{enumerate}[label=(\alph*)]
  \item For each $\tau > 0$ and $\gamma \in \bbR$, the function $x \mapsto F_2(x; \tau, \gamma)$ is a probability cumulative distribution function.
  Moreover, it is continuous in $\tau > 0$ and $\gamma \in \bbR$.
  \item $F_2(x; \tau, \gamma)$ is an even function and is $1$-periodic in $\gamma\in \bbR$.
  \item For each $x \in \bbR$ and $\gamma \in \bbR$,
  \[
    \lim_{\tau \to 0} F_2(\tau^{1/3} x - \frac{\gamma^2}{4\tau}; \tau, \gamma) =
    \left\{
      \begin{array}{ll}
        F_{\text{GUE}}(x), & \gamma \in (-\tfrac12, \tfrac12), \\
        F_{\text{GUE}}(x)^2, & \gamma = \tfrac12,
      \end{array}
    \right.
  \]
  where $F_{\text{GUE}}$ is the Tracy-Widom GUE distribution.
  \item For each $x \in \bbR$,
  \begin{equation}
    \lim_{\tau \to \infty} F_2 \big( -\tau + \frac{\pi^{1/4}}{\sqrt{2}} \tau^{1/2} x ; \tau, \gamma \big) = \frac{1}{\sqrt{2\pi}} \int_{-\infty}^x e^{-y^2/2} \dd y.
  \end{equation}
\end{enumerate}

\begin{figure}[htb!]
  \centering
  \includegraphics[scale=0.5]{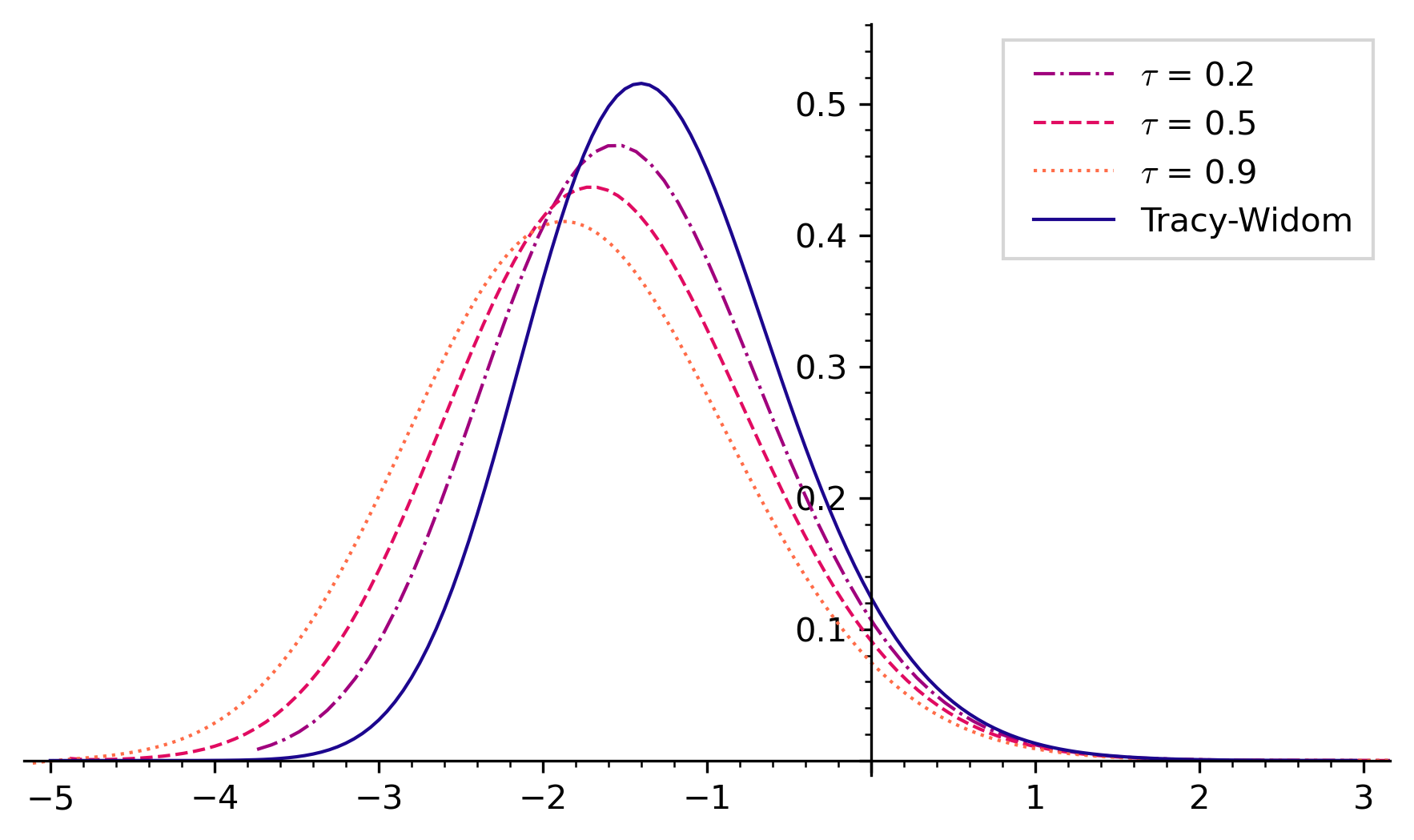}
  \hfill
  \includegraphics[scale=0.5]{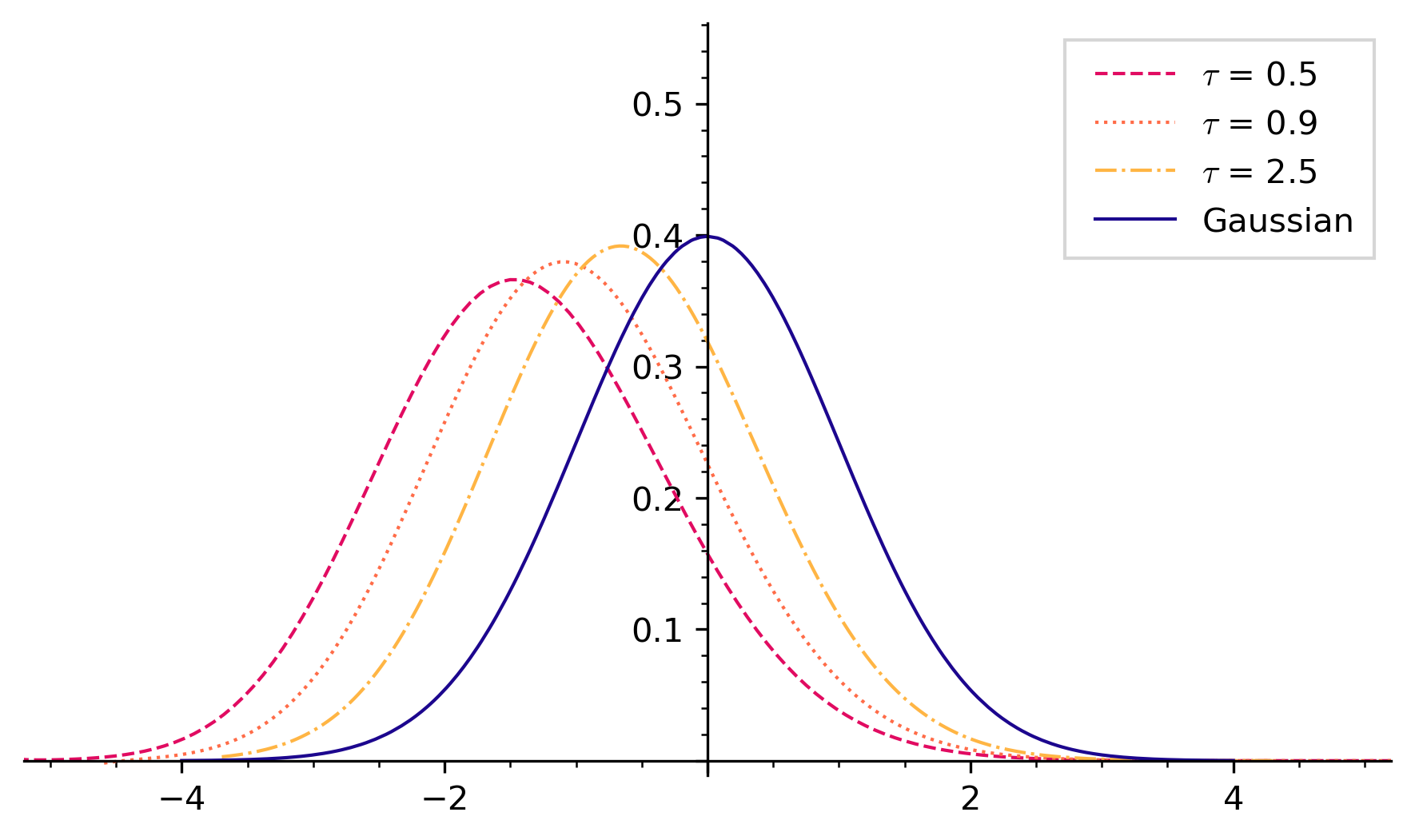}
  \caption{\textbf{Left:} The three curves from the left to right correspond to the probability distribution function of $F_2(\tau^{1/3} x - \frac{\gamma^2}{4\tau}; \tau, \gamma=\frac12)$ for the values $\tau = 0.9$, $0.5$ and $0.2$. The solid curve is the distribution of $F_{\text{GUE}}(x)^2$.
    \textbf{Right:} The three curves from the left to right correspond to the probability distribution function of $F_2(-\tau + \frac{\pi^{1/4}}{\sqrt{2}} \tau^{1/2} x; \tau, \gamma=\frac12)$ for the values $\tau = 0.5$, $0.9$ and $2.5$. The solid curve is the distribution of the standard Gaussian density.}
  \label{fig:F2_CDF_gamma_05}
\end{figure}

The statistics of the local current depends on the observed location of the step initial conditions.
This is in contrast to the flat initial conditions, where the initial conditions are translation invariant in the relaxation time scale, i.e.~$t \propto L^{3/2}$.
Note that the current $Q_{L-1}(t)$ on the edge $(L-1)$ has a well suited formula for asymptotic analysis, see \Cref{p:current_cdf}.
Whereas, the current for any other edge does not appear to be suitable for asymptotic analysis.
Thus, to study the asymptotic behavior of the current on different relative locations, we simply just consider shifted step initial conditions.
More precisely, let us consider the PushASEP model with the step initial condition defined by
\begin{equation}
  \label{e:step_IC1_result}
  (y_1, \dots, y_N) = (0, 1, \dots, N-1) + (m, \dots, m) \in \calX_{L, N},
\end{equation}
for an integer $0 \leq m \leq L-N$, or,
\begin{equation}
  \label{e:step_IC2_result}
  (y_1, \dots, y_N) = (0, 1, \dots, k-1, L-N+k, \dots, L-1),
\end{equation}
for and integer $0 \leq k \leq N$.
Note that the difference between the two cases above is that, in the second case \eqref{e:step_IC2_result}, the position of the current considered is in the middle of the step initial configuration.
This leads to slightly different formulations of the results and  computations that we will see below in \Cref{p:step_current}.

For the step initial conditions, we expect to observe shocks when the front of the step crashes with the back of the step. Additionally, we expect the shock to be traveling at a certain speed $\vshock$.
For the (T)ASEP, the speed of the shock may be determined by the considering the hydrodynamic limit, where space and time are scaled linearly.
The macroscopic density function $\rho(x,t)$ satisfies the inviscid Burgers’ equation, see \cite{Rost81, BenassiFouque87, DerridaLebowitzSpeer97} for details.
In our case, based on some heuristic analysis, we expect in the macroscopic density function for the PushASEP to satisfy the following Burgers-type partial differential equation
\begin{equation}
  \frac{\partial }{\partial t} \rho + \frac{\partial}{\partial x} \left[ p \rho (1- \rho) -q \frac{\rho}{1- \rho}  \right] = 0.
\end{equation}
Furthermore, we may expect the speed of the shock in the PushASEP to be given by
\begin{equation}
  \vshock := p (1-2\rho)  - \frac{q}{(1-\rho)^2}.
\end{equation}
Although we do not give a rigorous argument for the speed of the shock based on the hydrodynamic limit approach, we do see the parameter $\vshock$ appear naturally in our asymptotic analysis for the current of the step initial conditions and, thus, interpret the parameter as the speed of the shock.
In the following theorem, \Cref{t:asymptotic_step_current}, we determine the fluctuations of the local current at position $L-1$ for general step initial conditions under the relaxation time scale.

In particular, we consider a sequence $(L_n, N_n)_{n \geq 1}$ where $L_n$ denotes the size of the ring and $N_n$ denotes the number of particles.
We set $\rho_n := L_n / N_n$ for all $n \geq 1$ and assume that both $N_n$ and $L_n$ increase to infinity.
Then, we write $\bbP^\icstep_{1, n}$ for the probability measure of the PushASEP system with parameters $L_n$ and $N_n$ with the initial condition given as in \eqref{e:step_IC1_result} where the shift is set to be $m_n$; similarly, we write $\bbP^\icstep_{2, n}$ for the probability measure with the initial condition given as in \eqref{e:step_IC2_result} where the shift is set to be $k_n$.

\begin{theorem}
  \label{t:asymptotic_step_current}
  Take the PushASEP on ring of length $L_n$ with $N_n$ particles, where both $L_n$ and $N_n$ are increasing to infinity as $n\rightarrow \infty$, and set $\rho_n:= N_n/ L_n$. Also, introduce the parameters
  \begin{equation}
    v = v(p, q) := \rho \big[ p (1-\rho) - \frac{q}{1-\rho} \big], \quad r(p, q) := p + \frac{q}{(1-\rho)^3}.
  \end{equation}
  Additionally, assume that we take the step initial conditions given by \eqref{e:step_IC1_result} or \eqref{e:step_IC2_result}.
  \begin{enumerate}
    \item For the step initial condition given by \eqref{e:step_IC1_result}, the asymptotic fluctuations of the local current are given as follows
    \[
      \lim_{n \to \infty} \bbP^{\icstep}_{1, n} \Big( \frac{Q_{L_n-1}(t) - vt + \rho_n [(1-\rho_n) L_n - m_n]}{ \rho_n^{2/3} (1-\rho_n)^{2/3} t^{1/3} } \geq -x \Big) = F_2(\tau^{1/3} x; r \tau, \gamma),
      \quad x \in \bbR,
    \]
    where the parameters are given below under the following two assumptions.
    \begin{enumerate}
      \item Assume that $\rho_n = \rhoc + O(L_n^{-1})$.
      Then, let $\gamma \in [0, 1-\rho]$, $m_n = \lfloor (1 - \rho_n - \gamma) L_n \rfloor$, and set
      \[
        t =  \frac{\tau}{ \sqrt{\rho_n(1-\rho_n)} } L_n^{3/2}
      \]
      for any fixed $\tau > 0$.
      \item Assume that $|\rho_n - \rhoc|$ is bounded away from zero for $n$ large enough.
      Then, let $\gamma \in \bbR$ and $m_n$ be an arbitrary sequence of integers such that $0 \leq m_n \leq L_n - N_n$, and set
      \[
        t =  \frac{L_n}{|\vshock|} \bigg\lfloor \frac{|\vshock| \tau}{ \sqrt{\rho_n(1-\rho_n)} } L_n^{1/2} \bigg\rfloor - \frac{(\rho_n + \gamma) L_n + m_n}{\vshock}
      \]
      for any fixed $\tau > 0$.
    \end{enumerate}
    \item For the step initial condition is given by \eqref{e:step_IC2_result}, the asymptotic fluctuation of the local current is given as follows
    \[
      \lim_{n \to \infty} \bbP^{\icstep}_{2, n} \Big( \frac{Q_{L_n-1}(t) - vt + (1 - \rho_n) k_n}{ \rho_n^{2/3} (1-\rho_n)^{2/3} t^{1/3} } \geq -x \Big) = F_2(\tau^{1/3} x; r \tau, \gamma),
      \quad x \in \bbR,
    \]
    where the parameters are given below under the following two assumptions.
    \begin{enumerate}
      \item Assume that $\rho_n = \rhoc + O(L_n^{-1})$.
      Then, let $\gamma \in [1-\rho, 1]$, $k_n = \lfloor (1 - \gamma) L_n \rfloor$, and for any fixed $\tau > 0$, we set
      \[
        t =  \frac{\tau}{ \sqrt{\rho_n(1-\rho_n)} } L_n^{3/2}.
      \]
      \item Assume that $|\rho_n - \rhoc|$ is bounded away from zero for large enough $n$.
      Then, let $\gamma \in \bbR$ and $k_n$ be an arbitrary sequence of integers such that $0 \leq k_n \leq N_n$, and set
      \[
        t =  \frac{L_n}{|\vshock|} \bigg\lfloor \frac{|\vshock| \tau}{ \sqrt{\rho_n(1-\rho_n)} } L_n^{1/2} \bigg\rfloor - \frac{k_n + \gamma L_n}{\vshock}
      \]
      for any fixed $\tau > 0$.
    \end{enumerate}
  \end{enumerate}
\end{theorem}

\begin{rem}
  The results in \Cref{t:asymptotic_step_current}, when $(p, q) = (1, 0)$, coincides with those from \cite[Theorem~3.4]{baikliu2018} after an appropriate change of variables.
\end{rem}

\begin{rem}\label{r:step_shock}
  The time parametrization is different depending on whether the shock speed is zero or not.
  When the shock travels at speed zero, the position of the shock does not move, so only the shift in the initial condition determines the parameter $\gamma$ in the limit.
  When the shock travels at some speed, the quantity $\frac{L_n}{|\vshock|}$ is the necessary length of time for a shock to travel through the whole ring.
  The integer part $\bigg\lfloor \frac{|\vshock| \tau}{ \sqrt{\rho_n(1-\rho_n)} } L_n^{1/2} \bigg\rfloor$ counts the number of shocks that have traveled through the whole system, and the remaining term,  normalized by $\vshock$, measures how far we are away from the shock, giving rise to the parameter $\gamma$ in the limit.
\end{rem}

\section{Contour deformations and residues}
\label{s:contour_deformations}

We present several formulas obtained through contour deformations and residue computation techniques.
Contour deformations are particularly useful when we want to show that certain contour integrals vanish.
In the deformation process, we need to control, up to a certain order, the magnitude of different variables or functions. We will use the following notation, in the regime $|z| \to 0$ or $|z| \to \infty$ depending on the context which should be clear the most of the time.
\begin{align*}
  f \approx g & \quad \Longleftrightarrow \quad f =  \Theta (g)
                \quad \Longleftrightarrow \quad f = \calO(g) \text{ \ and \ } g = \calO(f), \\
  f \sim g & \quad \Longleftrightarrow \quad f = g (1 + o(1)), \\
  f \ll g & \quad \Longleftrightarrow \quad f = o(g), \\
  f \gg g & \quad \Longleftrightarrow \quad g = o(f).
\end{align*}

\subsection{From $(N+1)$-fold integral to 1-fold integral}

We take the contour integral formula of the generating series for the joint density function given in \eqref{e:main_N+1}, removing the $\zeta$-dependent term in the numerator, since it may be factored outside the integral, and we compute the contour integrals with respect to the $w$-variables by residue computations.
The result is a linear combination of 1-dimensional contour integrals on the $z$-variables; see \Cref{p:w_residues}. This provides the bridge between \eqref{e:main_N+1} and \eqref{e:main_1} in \Cref{t:main}.

\begin{proposition}
  \label{p:w_residues}
  Fix an initial configuration $Y = (y_1, \cdots, y_n) \in \calX_{L,N}$ for the PushASEP with $N$ particles on a ring of length $L$, and take $X = (x_1, \cdots, x_N) \in \calX_{L,N}$.
  Then, we have
  \begin{equation}
    \label{e:w_residues}
    \begin{split}
      & \bigcontourintRzetaOneNumerator \\
      & = u_0(\zeta) + \CIpi{1} \diffcontoursZ \sumNbetherootszetaOneNumerator,
    \end{split}
  \end{equation}
  where $u_0(\zeta)$ is a function of $\zeta$ and will be explicited in \Cref{l:u0_residue}.
\end{proposition}

\begin{proof}
  This proposition is a consequence of \Cref{l:contour_deformation_induction} and \Cref{l:u0_residue}, where the former consists of an induction step and the latter consists in computing the function $u_0(\zeta)$ in \eqref{e:w_residues}.

  Using the notation from \Cref{l:contour_deformation_induction}, the left side of \eqref{e:w_residues} is exactly the left side of \eqref{e:contour_deformation_induction} for $k = 0$ and the right side of \eqref{e:w_residues} is exactly the right side of \eqref{e:contour_deformation_induction} for $k = N-1$.
  Then, the result follows by applying the identity \eqref{e:contour_deformation_induction} $N$ times.
\end{proof}

In the next \Cref{l:contour_deformation_induction}, we compute the contour integrals with respect to the $w_k$-variables, inductively. We compute the contour integrals by residue computations.
As we will show in the proof of \Cref{l:contour_deformation_induction}, the residues are located at the deformed Bethe roots \eqref{e:bethe_roots_deformed} and (in the case $k = N-1$) at the roots of the coupling term \eqref{e:coupled_roots}.
That is, at the points in the sets $\mathcal{Q}(z)$ and $\mathcal{P}(z; \zeta)$, respectively.
When $|z|\rightarrow 0$, the deformed Bethe roots approach either 0 or 1 (see \eqref{e:P1_small_z} below) and this will play a role in our computations. Thus, we introduce the following decomposition of the set of deformed Bethe roots
\begin{equation}
  \label{e:q_small_sol}
  \calQ_0(z) = \calQ(z) \cap \{ w : \myre(w) < 1-\rho \}, \qquad
  \calQ_1(z) = \calQ(z) \cap \{ w : \myre(w) > 1-\rho \}.
\end{equation}
for $|z| < \rho^{\rho}(1- \rho)^{1- \rho}$ with $\rho = N/L$. Moreover, we note the asymptotic behavior of the deformed Bethe roots
\begin{equation}\label{e:q_small_asymptotic}
  \begin{cases}
    |\lambda| =|z|^{1/(1-\rho)}(1 + o(1)), &\quad \lambda \in \mathcal{Q}_0(z), \\
    |1- \lambda| =|z|^{1/\rho}(1 + o(1)), &\quad \lambda \in \mathcal{Q}_1(z),
  \end{cases}
\end{equation}
as $|z| \rightarrow 0$.
\begin{lemma}
  \label{l:contour_deformation_induction}
  We have the following identity \footnote{In the notation used here, empty products and empty integrals are taken to be the identity, and so, the right side of the case $k=N-1$ equals the integrand in the right side of \eqref{e:w_residues}.} between $(N-k+1)$-fold and $(N-k)$-fold contour integrals,
  \begin{equation}
    \label{e:contour_deformation_induction}
    \begin{split}
      & \CIpi{N+1-k} \diffcontoursZ 
      \diffcontoursw \frac{\dd w_{k+1}}{w_{k+1}} \cdots \diffcontoursw \frac{\dd w_N}{w_N} I_k(w_{k+1}, \cdots, w_N; z , \zeta , t)\\
      & = \CIpi{N-k} \diffcontoursZ 
      \diffcontoursw \frac{\dd w_{k+2}}{w_{k+2}}\cdots \diffcontoursw \frac{\dd w_N}{w_N} I_{k+1}(w_{k+2}, \cdots, w_N; z , \zeta , t)
      + u_0(\zeta) \mathds{1}(k = N-1),
    \end{split}
  \end{equation}
  for $k=0, \cdots, N-1$; where the contours satisfy the conditions \eqref{e:conditions} following the notation in \eqref{e:diffcontours}, the function $u_0(\zeta)$ along with its value at $\zeta = 1$ will be given later in \Cref{l:u0_residue}, and the integrands are given by
  \begin{equation}
    \label{e:I_k}
    \begin{split}
      I_0(w_{1}, \cdots, w_N; z , \zeta , t)
      & := \frac{ \sum_{\sigma \in S_N} A_{\sigma} \prod_{i =1}^N w_{\sigma(i)}^{y_{\sigma(i)} - x_i} e^{t E(\vec{w})}}%
      { p_z(\vec{w ; \zeta }) \prod_{i=1}^N q_z(w_i)}, \qquad \mbox{ and}  \\
      I_k(w_{k+1}, \cdots, w_N; z, \zeta , t)
      & :=
      \sum_{\SumLambdas[k]}
      \Bigg( \frac{ \sum_{\sigma \in S_N} A_{\sigma} \prod_{i=1}^N w_{\sigma(i)}^{y_{\sigma(i)} - x_i}  e^{t E(\vec{w})} }%
      { p_z(\vec{w} ; \zeta ) \prod_{i=k+1}^N q_z(w_i) \prod_{i=1}^{k} w_i q_z'(w_i)} \Bigg)_{(w_1, \cdots, w_k) = (\lambda_1, \cdots, \lambda_k)}.
    \end{split}
  \end{equation}
\end{lemma}

\begin{rem}
  Note that the integrand $I_{k+1}$ is equal to the residue of the integrand $I_{k}$ at a Bethe root, $w_{k+1} = \lambda_{k+1} \in \calQ(z)$.
  In other words, $I_{k+1}$ can be defined ``recursively'' from $I_k$ by replacing the term  $ q_z(w_{k+1})$ in the denominator with the term $ \lambda_{k+1} q_z'(\lambda_{k+1})$ and summing over all deformed Bethe roots for $\lambda_{k+1}$.
\end{rem}

\begin{proof}
  The arguments given below only apply for $2N < L$ due to \eqref{e:beta1_bounds}.
  The case $2N= L$ may be treated similarly given more technical constraints on the contours; see.
  
  The result follows by residue computations. We compute the integral on the left side of \eqref{e:contour_deformation_induction} with respect to the $w_{k+1}$-variable by computing the residues inside the region bounded by the contours, $C_R$, $C_{\epsilon_1}$ and $1 + C_{\epsilon_2}$.
  Residues may arise from two possible type of poles given by the following equations
  \begin{subequations}
    \subtag{P}
    \begin{align}
      q_z(w_{k+1}) = & \ 1 - z^L w_{k+1}^{-L}(1- w_{k+1}^{-1})^{-N} = 0, \label{e:pole_c1_v2} \\
      p_z( \vec{w}; \zeta )_{(w_1, \cdots, w_k) = (\lambda_1, \cdots, \lambda_k)} = & \ 1 + (-1)^N z^{-L} \zeta^L \prod_{i=1}^k(1-\lambda_i^{-1})  \prod_{i=k+1}^N(1-w_i^{-1}) = 0. \label{e:pole_b1_v2}
    \end{align}
  \end{subequations}
  It is straightforward to check that the rest of the terms in the integrand $I_k(w_{k+1}, \cdots, w_N; z)$ do not introduce poles or residues.
  Below, we start by checking the location of the poles coming from \eqref{e:pole_c1_v2} and \eqref{e:pole_b1_v2} in Step 1 and Step 2, i.e.~whether they lie inside the region bounded by the contours $C_R$, $C_{\epsilon_1}$ and $1 + C_{\epsilon_2}$.
  Then, we show that each pole is simple in Step 3, leading to a computation of the residues in Step 4 and Step 5.
  In Step 4, we compute the residue arising from \eqref{e:pole_c1_v2}, giving the main contribution in the induction step.
  In Step 5, we compute the residue arising from \eqref{e:pole_b1_v2} and show that apart from the case $k = N-1$, the residue vanishes after a contour deformation in $z$-variable, thus giving the indicator term in the right side of \eqref{e:contour_deformation_induction}.

  We denote $\bbA$ to be the region bounded by the contours $C_{R}$, $C_{\epsilon_1}$ and $1+ C_{\epsilon_2}$. This region is a sphere with three discs remove. In particular, the boundary of the region is given by three circles,
  \begin{equation}
    \partial \bbA = C_{R} - C_{\epsilon_1} - (1+ C_{\epsilon_2}),
  \end{equation}
  where the negative sign indicates the negative orientation on the circles. In the following, we first determine if the poles $w_{k+1} \in \bbC$, given by \eqref{e:pole_c1_v2} and \eqref{e:pole_b1_v2}, lie inside the region $\bbA$. This result will depend on the contours $C_{R}$, $C_{R'}$, $C_{\epsilon_1}$, $C_{\epsilon_2}$ and $C_{\epsilon'}$ since the solution to equations \eqref{e:pole_c1_v2} and \eqref{e:pole_b1_v2} depend on the variables $w_{k+2}, \cdots, w_{N}, z \in \bbC$ that lie on those contours; see \eqref{e:diffcontours} and \eqref{e:conditions} for a precise description of the contours.
  
  \paragraph{Step 1 (Poles from (P1))}
  Assume $w_{k+1}\in \bbC$ is given by \eqref{e:pole_c1_v2}.
  In this case, note that $w_{k+1}$ only depends on $z \in \bbC$, and we either have $z \in C_{R'}$ or $z \in C_{\epsilon'}$.
  In either case, we show that $w_{k+1} \in \bbA$.
  
  If $z \in C_{R'}$, we have
  \begin{equation}\label{e:P1_large_z}
    w_{k+1} = \eta z + \rho + o(1),
  \end{equation}
  so that $\eta$ is a $L^{th}$ root of unity and $\rho = N/L$. Also, by \eqref{e:powers} and \eqref{e:beta}, we have $R \gg R'$.
  Then, it follows that $w_{k+1} \in \bbA$. If $z \in C_{\epsilon'}$, there are two possible behaviours for $w_{k+1}$.
  We have
  \begin{equation}\label{e:P1_small_z}
    |w_{k+1}| \sim |z|^{1/(1-\rho)}, \quad |1- w_{k+1}^{-1}| \sim |z|^{d},
  \end{equation}
  with $d = L/N = 1/\rho$, according to \eqref{e:q_small_asymptotic}. Note that $w_{k+1} \rightarrow 0$ or $w_{k+1} \rightarrow 1$ as $z \rightarrow 0$. Either way, by \eqref{e:powers}, \eqref{e:beta1_bounds} and \eqref{e:beta2_bounds}, we have that $w_{k+1} \in \bbA$.
  
  \paragraph{Step 2 (Poles from (P2))}
  Assume $w_{k+1} \in \bbC$ is given by \eqref{e:pole_b1_v2}.
  In this case, note that $w_{k+1}$ depends on $\lambda_1, \cdots, \lambda_{k}, w_{k+2}, \cdots, w_{N}, z, \zeta \in \bbC$.
  In particular, we have $w_i \in \partial \bbA$, for $i = k+2, \cdots, N$, and $z \in C_{R'} \cup C_{\epsilon'}$.
  Additionally, note that the Bethe roots $\lambda_1, \cdots, \lambda_k \in \bbC$ depend only on $z \in \bbC$ since these are given by the equation $q_z(\lambda_j) = 0$, which correspond to the poles \eqref{e:pole_c1_v2}. Moreover, note that $\lambda_j \approx 1$ for some solutions of $q_z(\lambda_j)=0$, see \eqref{e:q_small_sol} and \eqref{e:q_small_asymptotic}.
  In the following, we show that $w_{k+1} \in \bbA$ only if $\lambda_j \approx 1$ for all $j=1, \cdots, k$. 
  
  First, we solve \eqref{e:pole_b1_v2} with respect to the $w_{k+1}$-variable and write $\mu$ for its root,
  \begin{equation}
    \label{e:P2_pole_mu}
    \restatableeq{\PoleMu}{
      \mu = \left( \frac{1}{1 + (-1)^N z^L \zeta^{-L} \prod_{i\neq k+1} (1- w_i^{-1})^{-1}}\right)_{(w_1, \cdots, w_k) = (\lambda_1, \cdots, \lambda_k)}.
    }
  \end{equation}
  Note that $\mu$ depends on $z$, $\zeta$, $\lambda_1, \cdots, \lambda_k$ and $w_{k+2}, \cdots, w_N$.
  Recall that we have $w_i\in \partial \bbA$ for $i =1, \cdots, k$.
  Then, we have
  \begin{equation}
    |1 - w_{i}^{-1}| = \begin{cases}
      1 + o(1),& \quad w_i \in C_{R}, \\
      \epsilon_1^{-1}(1 + o(1)),& \quad w_i \in C_{\epsilon_1}, \\
      \epsilon_2,& \quad w_i \in 1 +C_{\epsilon_2},
    \end{cases}
  \end{equation}
  for $i = k+2, \cdots, N$. Let $m_0$ (resp.~$m_1$ and $m_2$) be the number of $w_i$, for $i = k+2, \cdots, N$, with $w_i \in C_{R}$ (resp.~$w_i \in C_{\epsilon_1}$ and $w_i \in 1+ C_{\epsilon_2}$). Then, $m_0 + m_1 + m_2 = N-k-1$ with $m_0, m_1, m_2 \in \bbZ_{\geq 0}$. 
  
  Now, consider the behaviour of the Bethe roots $\lambda_j\in \calQ(z)$, $j =1, \cdots, k$.
  Since the $\lambda_j$ are determined by $q_z(\lambda_j) = 0$ just as in \eqref{e:pole_c1_v2}, we have the same behaviour as in Step 1 above, which depends on $z \in C_{R'}$ or $z \in C_{\epsilon'}$.
  We consider both cases below. 
  
  If $z \in C_{R'}$, we have
  \begin{equation*}
    \lambda_j  = \eta_j\, z + \rho + o(1), 
  \end{equation*}
  for $j = 1, \cdots, k$ and $\eta_j$ an $L^{th}$ root of unity.
  Then, we have
  \begin{equation}
    \label{e:P2_large_z}
    |1 - \mu^{-1} |
    = \left| z^L \zeta^{-L} \prod_{i\neq k+1} (1- w_i^{-1})^{-1}\right|
    \sim (R')^{L} \epsilon_1^{m_1} \epsilon_2^{-m_2}
    = (R')^{L  - m_1 \beta_1 + m_2 \beta_2}
  \end{equation}
  with the last equality due to \eqref{e:powers}.
  Moreover, since $\beta_1, \beta_2 >0$, we have
  \begin{equation}\label{e:P2_large_z_ineq}
    L - m_1 \beta_1 + m_2 \beta_2 \geq L - (N-k-1) \beta_1
    = N (d - \beta_1) + (k + 1) \beta_1 > \beta_1 .
  \end{equation}
  In particular, this means that $|1 - \mu^{-1}| \gg (R')^{\beta_1}$ and $\mu \ll (\epsilon')^{\beta_1} = \epsilon_1$.
  Thus, if $z \in C_{R'}$, we have $\mu \notin \bbA$.
  
  If $z \in C_{\epsilon'}$, we have
  \begin{equation*}
    |\lambda_j| \sim |z|^{1/(1-\rho)}, \quad \text{or} \quad |1 - \lambda_j^{-1}| \sim |z|^{d},
  \end{equation*}
  for $j =1, \cdots, k$.
  In particular, $\lambda_j \in \calQ_0(z)$ or $\lambda_j \in \calQ_1(z)$, see \eqref{e:q_small_sol} and \eqref{e:q_small_asymptotic}.
  Let $n_0$ (resp.~$n_1$) be the number of $\lambda_j$, for $j=1, \cdots, k$, with $\lambda_j \in \calQ_0(z)$  (resp.~$\lambda_j \in \calQ_1(z)$).
  Then, we have 
  \begin{equation}
    \label{e:P2_small_z}
    |1 - \mu^{-1}|
    = \left| z^L \prod_{i\neq k+1} (1- w_i^{-1})^{-1}\right|
    \sim (\epsilon')^{L + n_0/(1-\rho) - n_1 d + m_1 \beta_1 - m_2 \beta_2 }
  \end{equation}
  by \eqref{e:powers}.
  Below, we consider two cases $n_0 = 0$ and $n_0 > 0$ separately.

  If $n_0>0$, we bound the exponent in \eqref{e:P2_small_z} as follows
  \begin{align*}
    L + \frac{n_0}{1-\rho} - n_1 d + m_1 \beta_1 - m_2 \beta_2
    & > L + \frac{1}{1-\rho} - (k-1)d  - (N-k-1) \beta_2 \\
    &= (L- (N-1) \beta_2) + (k-1)(\beta_2-d)+ \frac{1}{1-\rho}   + \beta_2\\
    & >\beta_2,
  \end{align*}
  with the first inequality due to setting $(n_0, n_1)= (1, k-1)$ and $(m_1, m_2) = (0, N-k-1)$, and the second inequality due to \eqref{e:beta2_bounds}. In particular, this means that
  \begin{equation}
    |1 -\mu^{-1}| \ll (\epsilon')^{\beta_2} = \epsilon_2.
  \end{equation}
  Thus, we have $\mu \notin \bbA$ if $z \in C_{\epsilon'}$ and $n_0 >0$.
  
  If $n_0=0$, the exponent in \eqref{e:P2_small_z} is given by
  \begin{equation}
    L + \frac{n_0}{1-\rho} - n_1 d + m_1 \beta_1 - m_2 \beta_2 = L - k d + m_1 \beta_1 - m_2 \beta_2 = (N-k) d + m_1 \beta_1 - m_2 \beta_2,
  \end{equation}
  with the first equality due to setting $(n_0, n_1) = (0, k)$ and the second equality by recalling $d = L/N$.
  Again, we have to consider two cases, $m_2 < N-k-1$ and $m_2 =N-k-1$.
  If $m_2 < N-k-1$, we have
  \begin{equation}
    (N- k)d +m_1 \beta_1 - m_2 \beta_2 \geq k(\beta_2-d) + (L - (N-1) \beta_2) + \beta_2  >  \beta_2,
  \end{equation}
  with the first inequalities due to setting $(m_0, m_1, m_2) = (1, 0, N-k-2)$ and recalling $d = L/N$, and the second inequality due to the bounds $d< \beta_2 < L/(N-1)$ given by \eqref{e:beta2_bounds}.
  Then, by this inequality and \eqref{e:P2_small_z} we have
  \begin{equation}
    |1 - \mu^{-1}| \ll (\epsilon')^{\beta_2}= \epsilon_2,
  \end{equation}
  which means that $\mu \notin \mathbb{A}$ if $m_2 < N-k-1$ and $n_0=0$. The only remaining case is $n_1 = k, m_0 = 0, m_1 = 0 , m_2 = N-k-1$.
  In this case, we have
  \begin{equation}
    (N-k) d + m_1 \beta_1 - m_2 \beta_2= (N-k)( d-\beta_2) + \beta_2 < \beta_2,
  \end{equation}
  with the inequality due to the bound $d< \beta_2$ given by \eqref{e:beta2_bounds}.
  Then, by this inequality and \eqref{e:P2_small_z}, we have
  \begin{equation}
    |1 - \mu^{-1}| \gg (\epsilon')^{\beta_2}= \epsilon_2,
  \end{equation}
  which means that $\mu \in \bbA$.
  Thus, in conclusion, we have that $w_{k+1} \in \bbA$ if $n_0 = 0$ and $m_0 = m_1 = 0$, or equivalently, if $|1-\lambda_j| \sim |z|^d$ for all $j = 1, \cdots, k$ and $w_i \in 1 + C_{\epsilon_2}$ for all $i = k+2, \cdots, N$.
  
  \paragraph{Step 3 (Double roots from (P1) and (P2))}
  We show that the poles due to \eqref{e:pole_c1_v2} and \eqref{e:pole_b1_v2} are simple.
  In particular, we show that the solution sets for \eqref{e:pole_c1_v2} and \eqref{e:pole_b1_v2} are disjoint.
  The other cases that might lead to higher order poles, i.e.~$p_z(\vec{w})$ or $q_z(w)$ having repeated roots, are easy to rule out and we omit the details.
  
  First, we may show that some of the solutions of \eqref{e:pole_c1_v2} and \eqref{e:pole_b1_v2} are disjoint by showing that some solutions lie inside and some solutions lie outside the region $\bbA$. If $w_{k+1}$ is a solution of \eqref{e:pole_c1_v2}, we know that $w_{k+1} \in \bbA$, by Step 1, for $z \in C_{R'} \cup C_{\epsilon'}$.
  If $w_{k+1}$ is a solution to \eqref{e:pole_b1_v2}, we know that $w_{k+1} \in \bbA$ only if $z \in C_{\epsilon'}$, $\lambda_j \in \calQ_1(z)$ for $j =1, \cdots, k$, and $w_i \in 1 + C_{\epsilon_2}$ for $i = k+2, \cdots, N$.
  In all other cases, we have that the solution sets are disjoint. So, below, we only consider the case with $z \in C_{\epsilon'}$, $\lambda_j \in \calQ_1(z)$ for $j =1, \cdots, k$, and $w_i \in 1 + C_{\epsilon_2}$ for $i = k+2, \cdots, N$.
  
  If $w_{k+1}$ is a solution to \eqref{e:pole_c1_v2}, we have
  \begin{equation}
    |w_{k+1}| = (\epsilon')^{1/(1-\rho)}(1+ o(1)), \quad \text{or}\quad |1 - w_{k+1}^{-1}| = (\epsilon')^d(1+ o(1))
  \end{equation}
  by \eqref{e:q_small_asymptotic}. On the other hand if $w_{k+1}$ is a solution to \eqref{e:pole_b1_v2}, we have
  \begin{equation}
    \label{e:w_k+1_exponent}
    |1- w_{k+1}^{-1}| = (\epsilon')^{(N-k)(d-\beta_2) +\beta_2}(1+o(1))
  \end{equation}
  by \eqref{e:P2_small_z} and setting $(n_0, n_1, m_0, m_1, m_2) =(0, k, 0, 0, N-k-1)$, as it was shown in Step 2 to be a necessary condition for $w_{k+1} \in \bbA$.
  Moreover, note that it is necessary for the exponents $d$ and $(N-k)(d- \beta_2) + \beta_2$ to be equal for the solutions to \eqref{e:pole_c1_v2} and \eqref{e:pole_b1_v2} to be the same.
  Then, the case $k =N-1$ is the only case when the solutions to \eqref{e:pole_c1_v2} and \eqref{e:pole_b1_v2} might agree due to the exponents in the previous formulas.
  In particular, this case is equivalent to the system of equations
  \begin{equation}
    \begin{cases}
      p_{z}(\vec{w}; \zeta) = 0,\\
      q_{z}(w_i) = 0, \quad & i =1, \cdots, N,
    \end{cases}
  \end{equation}
  with $z \in C_{\epsilon'}$ and $|w_i| = |z|^d(1 + o(1))$ for $i=1, \cdots, N$.
  By \Cref{l:bethe_roots_location}, the solution set for this system of equations is empty. Therefore, it follows that the solution sets \eqref{e:pole_c1_v2} and \eqref{e:pole_b1_v2} are disjoint.
  
  \paragraph{Step 4 (Residues from (P1))}
  
  We compute the residues due to the poles from \eqref{e:pole_c1_v2}.
  Let $\lambda_{k+1} \in \bbC$ be a solution to \eqref{e:pole_c1_v2} or, equivalently, $\lambda_{k+1} \in \calQ(z)$, see \eqref{e:bethe_roots_deformed}.
  We know, from Step 1 above, that $\lambda_{k+1} \in \bbA$ if $z \in C_{R'} \cup C_{\epsilon'}$.
  By Step 3, the terms in the integrand $I_{k}$ are regular at the pole $w_{k+1} = \lambda_{k+1}$, except for the simple pole due to $q_{z}(w_{k+1})$.
  Then, we have
  \begin{equation}
    \Res_{w_{k+1} = \lambda_{k+1}} \frac{1}{q_{z}(w_{k+1})} = \frac{1}{q'_{z}(\lambda_{k+1})}.
  \end{equation}
  By computing the residues from $\lambda_{k+1} \in \calQ(z)$, we obtain,
  \begin{equation}
    \sum_{\lambda_{k+1} \in \calQ(z)} \Res_{w_{k+1} = \lambda_{k+1}} \frac{I_{k}(w_{k+1}, \cdots, w_N; z , \zeta , t)}{w_{k+1}}
    = I_{k+1}(w_{k+2}, \cdots, w_N; z , \zeta , t),
  \end{equation}
  for $k= 0, \cdots, N-1$. This gives rise to the first term on the right side of \eqref{e:contour_deformation_induction}.
  
  \paragraph{Step 5 (Residues from (P2) and $k \neq N-1$)}
  We compute the residues due to poles from \eqref{e:pole_b1_v2} with $k \neq N-1$.
  We will show that these residues vanish and do not contribute to the right side of \eqref{e:contour_deformation_induction}.
  Let $\mu \in \bbC$ be a solution to \eqref{e:pole_b1_v2} so that $\mu \in \bbA$.
  In Step 2, we obtained the conditions $n_0 = m_0 = m_1 = 0, n_1 = k$ and $m_2 = N-k-1$.
  This means that
  \begin{equation}
    \begin{cases}
      z \in C_{\epsilon'}, \\
      w_{i} \in 1 + C_{\epsilon_2}, \quad & \text{for } i = k+2, \cdots, N, \\
      \lambda_j  \in \calQ_1(z) , \quad & \text{for } j = 1, \cdots, k,
    \end{cases}
  \end{equation}
  with $\calQ_1(z)$ given by \eqref{e:q_small_sol}.
  Then, by Step 3, the resulting residue at $w_{k+1} = \mu$ is given by
  \begin{equation}
    \label{e:step5}
    \oint_{C_{\epsilon'}} \frac{\dd z}{z} \oint_{1 + C_{\epsilon_2}} \frac{\dd w_{k+2}}{w_{k+2}} \cdots \oint_{1 + C_{\epsilon_2}} \frac{\dd w_{N}}{w_{N}}
    \Res_{w_{k+1} = \mu}
    \frac{ I_{k}(w_{k+1}, \cdots, w_{N}; z , \zeta , t) }{ w_{k+1} },
  \end{equation}
  with
  \begin{align}
    \Res_{w_{k+1} = \mu} \frac{ I_{k}(w_{k+1}, \cdots, w_{N}; z , \zeta , t) }{w_{k+1} }
    = \sum_{\SumLambdasR[k]}
    \Bigg( \frac{ (1 - w_{k+1}) \sum_{\sigma \in S_N} A_{\sigma} \prod_{i =1}^N w_{\sigma(i)}^{y_{\sigma(i)} - x_i}  e^{t E(\vec{w})} }%
    { \prod_{i=k+1}^N q_z(w_i) \prod_{i=1}^{k} w_i q_z'(w_i)} \Bigg)_{\subalign{ & (w_1, \cdots, w_k) = (\lambda_1, \cdots, \lambda_k) \\  & \phantom{(} w_{k+1} = \mu }}, \label{e:step5_residue}
  \end{align}
  which is a consequence of the following computation,
  \begin{equation}
    \label{e:pz_residue}
    \restatableeq{\PzResidue}{
      \Res_{w_{k+1} = \mu } \bigg( \frac{1}{w_{k+1} \; p_z(\vec{w} ; \zeta )} \bigg)
      = \left( \frac{1}{ \partial_{k+1} [ w_{k+1} \; p_z ( \vec{w} ; \zeta )] } \right)_{w_{k+1} = \mu}
      = 1 - \mu,
    }
  \end{equation}
  so that $\partial_{k+1} p_z(\vec{w} ; \zeta )$ is the partial derivative of $p_z(\vec{w} ; \zeta )$ with respect to $w_{k+1}$.  
  
  We show below that this residue term vanishes.
  In particular, we show that the integrand can be bounded by some positive power of $\epsilon'$.
  This means that the resulting residue in \eqref{e:step5} must equal zero since it is arbitrarily close to 0 by taking $\epsilon'$ small enough.

  First, we make sure that when taking $\epsilon'$ to 0, we do not cross any poles.
  For $i = k+2, \cdots, N$, we have
  \begin{equation}
    \label{e:no_pole_z_deform_1}
    |1-q_z(w_i)|
    = |z^L w_{i}^{-L}(1- w_i^{-1})^{-N}|
    \approx (\epsilon')^{L - \beta_2 N},
  \end{equation}
  due to $w_i \in 1 + C_{\epsilon_2}$.
  Moreover, \eqref{e:beta2_bounds} gives $L - \beta_2 N < 0$, meaning that for $\epsilon'$ small enough, $|1 - q_z(w_i)| \gg 1$ for $i = k+2, \cdots, N$.
  In other words, for $\epsilon'$ small enough, $q_z(w_i)$ stays away from zero.
  Similarly, for $i = k+1$ with $w_{k+1} =\mu$, we have
  \begin{equation}
    \label{e:no_pole_z_deform_2}
    |1 - q_z(w_{k+1})|
    = |z^L w_{k+1}^{-L}(1- w_{k+1}^{-1})^{-N}|
    \sim (\epsilon')^{ L - N( (N-k)(d-\beta_2) + \beta_2 ) }
    = (\epsilon')^{(N-k-1)(\beta_2 N - L)},
  \end{equation}
  where we recall from \eqref{e:w_k+1_exponent} that
  \begin{equation*}
    |1 - w_{k+1}^{-1}| \sim (\epsilon')^{(N-k)(d-\beta_2) + \beta_2} \ll 1,
  \end{equation*}
  resulting from the fact that the exponent has the following bound,
  \begin{equation*}
    (N-k)(d-\beta_2) + \beta_2 = k (\beta_2 -d) + L - (N-1) \beta_2 >0.
  \end{equation*}
  Note that, again by \eqref{e:beta2_bounds} and the condition $k < N-1$, the exponent in \eqref{e:no_pole_z_deform_2} satisfies
  \begin{equation}
    \label{e:step5_positive_exponent}
    (N-k-1)(\beta_2 N - L) > 0.
  \end{equation}
  This means that $|1 - q_z(w_{k+1})| \ll 1$ and hence $q_z(w_{k+1}) \neq 0$ when $\epsilon'$ is small enough.
  It is important to note that the bound in \eqref{e:step5_positive_exponent} is only possible when $N \neq k-1$ and in the case $N = k-1$, we actually pick up a residue, which will be discussed later in Step 6.

  Now, we show that the residue term given in \eqref{e:step5_residue} can be bounded by some power of $\epsilon'$.
  Let us start with the most complicated term: $A_\sigma$, $\sigma \in S_N$.
  Recall that
  \begin{equation}
    \label{e:step5_roots}
    \begin{cases}
      |1 - \lambda_i^{-1}| \sim (\epsilon')^d \quad & i =1, \cdots, k, \\
      |1- w_i^{-1}| = \epsilon_2 = (\epsilon')^{\beta_2}, \quad & i = k+2, \cdots, N, \\ 
      |1-w_{k+1}^{-1}| \sim (\epsilon')^{(N-k)(d-\beta_2) + \beta_2},
    \end{cases}
  \end{equation}
  since $\lambda_i \in \calQ_1(z)$, $w_i \in 1 + C_{\epsilon_2}$ and $w_{k+1}$ is a solution to \eqref{e:pole_b1_v2}.
  Then, given $\sigma \in S_N$, recall the expansion of $A_\sigma$ given in \eqref{e:bethe_coeffs_zeta} and write
  \begin{align*}
    |A_\sigma(\vec{w})|
    & = \prod_{j=1}^N | 1 - w_{\sigma(j)}^{-1} |^{\sigma(j) - j}
      = \prod_{j=1}^N | 1 - w_{j}^{-1} |^{j - \sigma^{-1}(j)},
  \end{align*}
  where the evaluation $(w_1, \cdots, w_k) = (\lambda_1, \cdots, \lambda_k)$ is omitted in the above expression, and also below, to make the notation more concise.
  Using the orders in \eqref{e:step5_roots}, the fact that $\beta_2 > d > (N-k)(d-\beta_2) + \beta_2$, and the rearrangement inequality, we find
  \begin{align*}
    |A_\sigma(\vec{w})|
    \leq \sup_{\sigma \in S_N} \prod_{j=1}^N | 1 - w_{j}^{-1} |^{j - \sigma^{-1}(j) }
    \sim \prod_{j=1}^{N} | 1 - w_{j}^{-1} |^{j - \sigma_0(j)},
  \end{align*}
  where $\sigma_0 \in S_N$ is such that
  \begin{align*}
    \sigma_0( \{ 1, \cdots, k \} ) & = \{ 2, \cdots, k+1 \}, \\
    \sigma_0( k+1 ) & = 1, \\
    \sigma_0( \{ k+2, \cdots, N \} ) & = \{ k+2, \cdots, N \}.
  \end{align*}
  Hence, we obtain
  \begin{equation}
    \label{e:Asigma_bound}
    |A_\sigma(\vec{w})| < \mathrm{Const} \times (\epsilon')^{(N-k-1)(d-\beta_2)k}.
  \end{equation}
  for some constant $\mathrm{Const}$ independent of $\epsilon'$. Let us analyse the other terms.
  We first note that due to \eqref{e:step5_roots},
  \begin{equation}
    \label{e:ws_bounds}
    \begin{cases}
      \lambda_i \approx 1, & \forall i = 1, \cdots, k, \\
      w_{i} \approx 1, & \forall i = k+1, k+2, \cdots, N.
    \end{cases}
  \end{equation}
  Besides, \eqref{e:no_pole_z_deform_1} and \eqref{e:no_pole_z_deform_2} also give that
  \begin{equation}
    \label{e:qz_bound}
    \forall i = 1, \cdots, k, \quad |q_z(w_i)| \approx |1 - q_z(w_i)| \approx (\epsilon')^{L - \beta_2 N}
    \qquad \text{and} \qquad |q_z(w_{k+1})| \approx  1.
  \end{equation}
  The residue computation at of $q_z$ at $\lambda_i$, for $i = 1, \cdots, k$, gives
  \begin{equation}
    \label{e:residue_P1_bound}
    |\lambda_i q'_z(\lambda_i)| = \left| \frac{L \lambda_i - (L-N)}{ \lambda_i - 1 } \right| \approx (\epsilon')^{-d}, \quad \forall i = 1, \cdots, k.
  \end{equation}

  Putting all the above bounds \eqref{e:Asigma_bound}, \eqref{e:ws_bounds} \eqref{e:qz_bound} and \eqref{e:residue_P1_bound} together, we obtain
  \begin{align*}
    \left| \Res_{w_{k+1} = \mu} I_{k}(w_{k+1}, \cdots, w_{N}; z; t) \right|
    & < \mathrm{Const} \times 
      \frac{ (\epsilon')^{(N-k)(d-\beta_2) + \beta_2} (\epsilon')^{(N-k-1)(d-\beta_2)k} }%
      {(\epsilon')^{(L-\beta_2 N)(N-k-1)} (\epsilon')^{-dk}} \\
    & = \mathrm{Const} \times  (\epsilon')^{ (N-k-1)^2 (\beta_2 - d) + d (k+1) }
      \leq \mathrm{Const} \times  (\epsilon')^{ d (k+1) }
  \end{align*}
  Thus, the integral in \eqref{e:step5} can be bounded by
  \begin{align*}
    \left| \oint_{C_{\epsilon'}} \frac{\dd z}{z} \oint_{1 + C_{\epsilon_2}} \frac{\dd w_{k+2}}{w_{k+2}} \cdots \oint_{1 + C_{\epsilon_2}} \frac{\dd w_{N}}{w_{N}} \Res_{w_{k+1} = \mu} I_{k}(w_{k+1}, \cdots, w_{N}; z, \zeta t) \right|
    &< \mathrm{Const} \times  (\epsilon')^{ N-k-1 + d (k+1) }\\
    &= \mathrm{Const} \times  (\epsilon')^{N + (d-1)(k+1)}.
  \end{align*}
  Since we may take $\epsilon' > 0$ arbitrarily small, it follows that the integral vanishes.
  Therefore, the residues vanish for the poles due to \eqref{e:pole_b1_v2} with $k \neq N-1$.

  \paragraph{Step 6 (Residues from (P2) and $k =N-1$)}
  We compute the residue due to the poles form \eqref{e:pole_b1_v2} with $k = N-1$.
  Note that the analysis from Step 5 fails since the term in \eqref{e:no_pole_z_deform_2} is of constant order and thus $q_z(w_{k+1})$ might be 0.
  In this case, the residue we obtain corresponds to the second term on the right side of \eqref{e:contour_deformation_induction}.
  
  Let $\mu \in \bbC$ be a solution of \eqref{e:pole_b1_v2} with respect to the variable $w_{N}$ so that $\mu \in \bbA$.
  By Step 2 above, we must have that $\lambda_j \in \calQ_1(z)$ for all $j = 1, \cdots, N-1$ and $z \in C_{\epsilon'}$.
  By Step 3, the terms in the integrand are regular at the pole $w_N = \mu$, except for the simple pole due to $p_z(\vec{w}; \zeta )$.
  Then, we define $u_0(\zeta)$ as the contour integral with respect to $\frac{\dd z}{z}$ of the residue of $I_{N-1}(w_N; z; t) / w_N$ at $w_N = \mu$,
  \begin{equation}
    \label{e:u0_definition}
    u_0(\zeta) := -\oint_{C_{\epsilon'}} \frac{\dd z}{z} \Res_{w_N = \mu} \frac{ I_{N-1}(w_N; z , \zeta , t) }{w_N}.
  \end{equation}
  A priori, the above integral depends on $t$ but, for $\zeta^L =1$, we will see in \Cref{l:u0_residue} that $u_0$ is indeed independent of $t$.
  This establishes the second term on the right side of \eqref{e:contour_deformation_induction}.
\end{proof}

\begin{rem}\label{r:conditions_p}
  The result of \Cref{l:contour_deformation_induction} may be extended to the case $2N \leq L$ by changing the radius of the contours as follows:
  \begingroup
  \renewcommand{\theequation}{Cond'}
  \begin{subequations}
    \subtag{Cond'}
    \label{e:conditions_p}
    \begin{empheq}[left = \empheqlbrace]{align}
      & \epsilon' \ll 1 \text{ is arbitrarily small and } R' = \alpha_{0}/\epsilon', \label{e:z_contour_p} \\
      & R = (R')^{\beta}, \epsilon_1 = \alpha_1 (\epsilon')^{\beta_1}, \epsilon_2 = (\epsilon')^{\beta_2}, \label{e:powers_p}  \\
      & \beta > 1,\label{e:beta_p}\\
      & \alpha_1 < 2^{-N/(L-N)}, \quad (1 + \epsilon_1)(1+ \epsilon_1^{N-1}) < \alpha_{0}^L \alpha_1^N,\label{e:alpha}\\
      & \tfrac{1}{1-\rho} = \tfrac{L}{L-N}\leq \beta_1 \leq d = \tfrac{L}{N}, \label{e:beta1_bounds_p}\\
      &  d  < \beta_2 < \tfrac{L}{N-1}. \label{e:beta2_bounds_p}
    \end{empheq}
  \end{subequations}
  \endgroup
\end{rem}

In particular, for \eqref{e:conditions} and for \eqref{e:beta1_bounds} in particular, the root given by the left side of \eqref{e:P1_small_z} may not lie inside $\bbA$ and the inequality \eqref{e:P2_large_z_ineq} does not hold for $m_1 = N-1$.
By introducing \eqref{e:conditions_p}, we can treat these two exceptional cases and the rest of the arguments are the same. We omit the details of the case $2N=L$ and only consider the $2N < L$ case to make the arguments less cumbersome.
Additionally, we note that \Cref{l:contour_deformation_induction} (and \Cref{l:CI_zero}) is a key result and, accordingly, the rest of the results in the paper extend to the case $2N = L$ by extending \Cref{l:contour_deformation_induction} and \Cref{l:CI_zero} to the case $2N = L$ by using the conditions \eqref{e:conditions_p} on the contours. 

\subsection{Expressions for the $u_0$ function.}

We give a formula for the $u_0$ function in \Cref{l:u0_residue}. This function is defined as an integral of a residue term in \eqref{e:u0_definition}. In the following \Cref{l:u0_residue}, we give a formula for the $u_0$ function by evaluating the residue. Additionally, in \Cref{l:u0_alt}, we give an alternate expression for the $u_0$ function that will be more useful for further computations.

\begin{lemma}
  \label{l:u0_residue}
  The function $u_0 : \bbU \to \bbC$ in \Cref{p:w_residues} is given by the following formula
  \begin{align}
    \label{e:u0_CI}
    u_0(\zeta)
    & = \oint_{C_{\epsilon'}} \frac{\dd z}{z} \CISumAsymmetricT,
  \end{align}
  where $\mu$ is a function in $\zeta$, and is defined in \Cref{e:P2_pole_mu}.
  Moreover, we have
  \begin{align*}
    u_0(\zeta) & = {L \choose N}^{-1}
  \end{align*}
  for $\zeta^L = 1$.
\end{lemma}

\begin{proof}
  First, recall that $u_0$ is defined in \eqref{e:u0_definition}. Then, \eqref{e:u0_CI} follows from a similar residue computation as in \eqref{e:pz_residue}; we omit the details. For the second part, assume $\zeta \in \bbU$ is such that $\zeta^L = 1$.
  We write
  \begin{align*}
    u_0(\zeta)
    & = \lim_{z \to 0} \CISumAsymmetricT \\
    & = \lim_{z \to 0}   \sum_{\SumLambdasR[N-1]}
      \Bigg(
      \frac{ \prod_{i=1}^{N} ( w_i -1) }{q_z(w_N)}
      \frac{ u_{\vec{w}}(X) }{ N^{N-1} } 
      \Bigg)_{\subalign{ & (w_1, \dots, w_{N-1}) = (\lambda_1, \dots, \lambda_{N-1}) \\  & \phantom{(} w_{N} = \mu }}.
  \end{align*}
  The first equality above can be obtained using the following two facts.
  First, we can easily see that the integrand of \eqref{e:u0_CI} does have a pole at $z=0$; moreover, since the integrand is holomorphic around 0 and is not identically zero, there is no other singularity in an arbitrarily small neighborhood.
  Second, the contour integral, which corresponds to a residue computation, can be translated into a limit if the limit is not zero, that is, the pole is of order 1. We provide more details about this residue computation below.

  In the second equality, we use the property that when $z$ tends to 0, all the $w_i$'s tend to 1, implying $E(\vec{w}) \to 1$ under the condition $(w_1, \cdots, w_{N-1}) = (\lambda_1, \cdots, \lambda_{N-1}) \in \calQ_1(z)^{N-1}$.
  For $1 \leq i \leq N-1$, we rewrite the quantity $w_i q_z'(w_i) = \frac{L w_i - (L-N)}{w_i - 1}$ and use the fact that $\lim_{z \to 0} (L \lambda_i(z) - (L-N)) = N$.
  
  To conclude the proof, first note that if we have $w_i = w_j$ for $i \neq j$, then $u_{\vec{w}}(X) = 0$ due to its determinantal structure.
  Hence, we can reduce the summation set from $\{ (\lambda_i)_{1 \leq i \leq N-1} : \lambda_i \in \calQ_1(z) \}$ to
  \begin{equation}
    \label{e:disjoint_Q1}
    {\calQ}_1^\neq (z) := \{ (\lambda_i)_{1 \leq i \leq N-1} \in \bbC^{N-1} \mbox{ such that } \lambda_i \in \calQ_1 (z) \mbox{ are pairwise distinct} \}.
  \end{equation}
  \Cref{l:limit_ratio} shows that for any choice of $(\lambda_i)_{1 \leq i \leq N-1}$ in $\calQ_1^\neq(z)$, the following limit is the same and has the value
  \begin{align*}
    \lim_{z \to 0}
    \Bigg( \frac{\prod_{i=1}^{N-1} (\lambda_i -1) (w_N-1)}{q_z(w_N)} \Bigg)_{w_N = \mu}
    = N^{-1} {L \choose N}^{-1}.
  \end{align*}
  Additionally, by \Cref{l:sum_permutations}, we may compute the summation
  \begin{align*}
    \sum_{\vec{\lambda} \in \calQ_1^\neq (z)} u_{(\vec{\lambda}, w_N)}(X) = N^N.
  \end{align*} 
  Thus, the second equality of this lemma follows from the previous identities.
\end{proof}

In the rest of this section, we establish an alternate expression for the $u_0$ function.
This following expression will be useful in further arguments.
For instance, note that the following expression for $u_0$ easily combines with the other terms in the 1-fold integral expression for the generating series for the joint distribution $g(X; \zeta; t)$ given by \eqref{e:main_1}; see \Cref{s:cdf_proof} and, in particular, \eqref{e:prob_distr_alt} for instance.

\begin{lemma}
  \label{l:u0_alt}
  Let $u_0$ be given by \eqref{e:u0_CI}.
  Then, we have
  \begin{equation}\label{e:u0_alt}
    u_0(\zeta) =  \oint_{C_{\epsilon'}} \frac{\dd z}{z}
    \sum_{\SumLambdasR}
    \Bigg(
    \frac{ h_{\vec{\lambda}}(Y) u_{\vec{\lambda}}(X) e^{t E(\vec{\lambda})} }{ p_z(\vec{\lambda}; \zeta) \prod_{i=1}^N \lambda_i q_z'(\lambda_i) }
    \Bigg).
  \end{equation}
\end{lemma}

\begin{rem}
  We may note the two following facts.
  \begin{enumerate}
    \item The summation inside the integrand of \eqref{e:u0_alt}, an alternative formula for $u_0$, is symmetric in the roots $\lambda_i \in \calQ_1(z)$.
    This allows us to interprete it, in the case $\zeta^L = 1$, as the \emph{stationary solution} to the infinitesimal generator $\calH_\zeta$, as given in the first part of \eqref{e:master_joint}.
    \item The function $u_0$ depends both on $\zeta$ and $\epsilon'$.
    For any fixed $\epsilon' > 0$, the function $u_0$ is continuous in $\zeta$; however, for any fixed $\zeta$ with $\zeta^L \neq 1$, and $\epsilon'$ arbitrarily small, the function $u_0(\zeta)$ is actually zero.
    Another way to understand this fact is that, the parameter $\zeta^L$ should be seen as a perturbation parameter of the system.
    When $\zeta^L = 1$, the stationary solution does indeed correspond to the case $z=0$.
    When $\zeta^L \neq 1$, the corresponding \emph{perturbed} stationary solution corresponds to some $z$ very close to $0$, which indeed does have contribution when $\epsilon' > |z| > 0$ but vanishes when $\epsilon' < |z|$.
  \end{enumerate}
  
\end{rem}

\begin{proof}
  This a direct consequence of \Cref{l:u0_zero_int} and \Cref{l:u0_res_int} below. In particular, we obtain
  \begin{equation}
    u_0(\zeta) - \oint_{C_{\epsilon'}} \frac{\dd z}{z}
    \sum_{\SumLambdasR}
    \Bigg(
    \frac{ h_{\vec{\lambda}}(Y) u_{\vec{\lambda}}(X) e^{t E(\vec{\lambda})} }{ p_z(\vec{\lambda}; \zeta) \prod_{i=1}^N \lambda_i q_z'(\lambda_i) }
    \Bigg) = 0
  \end{equation}
  from the lemmas.
  The result follows by solving for $u_0$ in the equation above.
\end{proof}

For the following technical lemmas, we introduce a new variable $\epsilon_3$ that is given by $\epsilon_3 = (\epsilon')^{\beta_3}$, similar to $\epsilon_1$ and $\epsilon_2$ in \eqref{e:conditions}. Additionally, the exponent $\beta_3$ is chosen so that
\begin{equation}
  \label{e:beta3}
  \tfrac{L- \beta_2}{N-1} < \beta_3 < \tfrac{L}{N}, \quad \text{and} \quad
  L(N-k) - N (N-k-m) \beta_2 - N m \beta_3 \neq 0
  \tag{Cond6}
\end{equation}
for $0 \leq k \leq N-1$ and $0 \leq m \leq N-k$. One may check that such a $\beta_3$ always exists.

\begin{lemma}\label{l:u0_zero_int}
  Set $\epsilon_3 = (\epsilon')^{\beta_3}$ with $\beta_3$ satisfying \eqref{e:beta3}. Then,
  \begin{equation}
    \oint_{C_{\epsilon'}} \frac{\dd z}{z} \diffcontourA{1} \diffcontourA{N}
    \Bigg( \frac{ h_{\vec{w}} (Y) u_{\vec{w}} (X) e^{t E(\vec{w}) }}
    { p_z( \vec{w} ; \zeta ) \prod_{i=1}^N q_z(w_i) } \Bigg) = 0.
  \end{equation}
\end{lemma}

\begin{proof}
  The proof is similar to the argument in Step 5 for the proof of \Cref{l:contour_deformation_induction}.
  We bound the integrad uniformly by a non-negative power of $\epsilon'$.
  Then, by taking $\epsilon' > 0$ arbitrarily small, we have that the integral must equal zero.
  
  Let us approximate the different terms in the integrand. Recall that
  \[
    h_{\vec{w}} (Y) u_{\vec{w}} (X) = \sum_{\sigma \in S_N} A_\sigma(\vec{w}) \prod_{i=1}^N w_{\sigma(i)}^{y_{\sigma(i)} - x_i}, \quad A_\sigma(\vec{w}) = (-1)^\sigma \prod_{j=1}^N (1 - w_{\sigma(j)}^{-1})^{\sigma(j) - j}
  \]
  and note that $|w_k-1| = \epsilon_i$, with $i = 2$ or $i = 3$, since $w_k \in (1 + C_{\epsilon_2}) \cup (1+ C_{\epsilon_3})$. Then, it is straightforward to check that all the terms in the integrand are bounded by constants except for the terms $q_z, p_z,$ and $A_{\sigma}$.
  For the term $q_z$, we have the following approximation:
  \begin{equation}
    |q_z(w)| \sim
    \begin{cases}
      (\epsilon')^{L-\beta_2 N}, \quad & w \in 1+C_{\epsilon_2}, \\ 
      1, \quad & w \in 1 +C_{\epsilon_3},
    \end{cases}
  \end{equation}
  since $\beta_3 < L/N$ by \eqref{e:beta3}.
  For the terms $p_z$, we have the following approximation,
  \begin{equation}
    |p_z(\vec{w}; \zeta) -1| \sim (\epsilon')^{-L + m_2 \beta_2 +m_3 \beta_3}
  \end{equation}
  where $m_i$, $i=2, 3$, are the number of $w$-variables that lie in the contour $1+C_{\epsilon_i}$.
  Note that the exponent on the right side is only negative when $m_2 = 0$, since $\beta_3 <L/N$, and the exponent is positive when $m_2 \geq 1$, since $(L - \beta_2)/(N-1) < \beta_3$.
  Then, we have
  \begin{equation}
    |p_z(\vec{w}; \zeta)| \sim \begin{cases}
      1, \quad & m_2 \geq 1 \\
      (\epsilon')^{-L + N \beta_3}, \quad & m_2 = 0
    \end{cases}.
  \end{equation}
  Lastly, we consider the $A_{\sigma}$ terms.
  The magnitude of these terms depends on $\sigma, m_2, m_3$, and the contours, $1 + C_{\epsilon_2}$ or $1 + C_{\epsilon_3}$, for the  $w_j$-variables. For fixed $m_2, m_3$, the largest magnitude of $A_\sigma$ is obtained by considering the rearrangement inequality. In particular, since $\beta_2 > \beta_3$, the magnitude of $A_\sigma$ is maximized when
  \begin{align*}
    \sigma(j) \in
    \begin{cases}
      \{ 1, \dots, m_3 \}, & w_j \in 1 + C_{\epsilon_3}, \\
      \{ m_3 + 1, \dots, N \}, & w_j \in 1 + C_{\epsilon_2}.  
    \end{cases}
  \end{align*}
  with the choice $w_j \in 1 + C_{\epsilon_2}$ for $1 \leq j \leq m_2$ and $w_j \in 1 + C_{\epsilon_3}$ for $m_2 + 1 \leq j \leq N$. Thus, for any $\sigma \in S_N$, we have 
  \begin{equation}
    |A_{\sigma}| \sim (\epsilon')^{b}
  \end{equation}
  with $b = (\beta_3 - \beta_2) m_2 m_3$.
  
  Now, we use the approximations above to give a uniform bound for the integrand. Let $I_N$ denote the integrand. Then, we have
  \begin{equation}
    |I_N| \sim
    \begin{cases}
      (\epsilon')^{m_2(-L + m_2 \beta_2 + m_3 \beta_3 )}, \quad & m_2 \geq 1, \\
      (\epsilon')^{L-N \beta_3}, \quad & m_2 = 0.
    \end{cases}
  \end{equation}
  Note that in either case the exponent is positive since we have chosen $\beta_3$ so that $(L-\beta_2)/(N-1) < \beta_3 < L/N$. It follows that the integral is arbitrarily small and, thus, equal to zero. This establishes the result.
\end{proof}

\begin{lemma}
  \label{l:u0_res_int}
  Set $\epsilon_3 = (\epsilon')^{\beta_3}$ with $\beta_3$ satisfying \eqref{e:beta3}. Then,
  \begin{equation}
    \begin{split}
      & \oint_{C_{\epsilon'}} \frac{\dd z}{z} \diffcontourA{1} \cdots \diffcontourA{N}
      \left( \frac{ h_{\vec{w}} (Y) u_{\vec{w}} (X) e^{t E(\vec{w}) } }
        { p_z( \vec{w} ; \zeta ) \prod_{i=1}^N q_z(w_i) } \right) \\
      & \qquad \qquad \qquad =  -u_0(\zeta) + \oint_{C_{\epsilon'}} \frac{\dd z}{z} \sum_{\SumLambdasR} \left(\frac{h_{\vec{\lambda}}(Y) u_{\vec{\lambda}}(X) e^{t E(\vec{\lambda})}}{p_z(\vec{\lambda}; \zeta)\prod_{i=1}^N \lambda_i q_z'(\lambda_i)} \right).
    \end{split}
  \end{equation}
\end{lemma}

\begin{proof}
  This result follows by residue computations similar to those in the proof of \Cref{l:contour_deformation_induction}. In particular, we inductively show
  \begin{equation}
    \label{e:J_k_recurrence}
    \begin{split}
      & \oint_{C_{\epsilon'}} \frac{\dd z}{z} \diffcontourA{k} \cdots \diffcontourA{N}
      J_{k-1}(w_{k}, \dots, w_N; z, \zeta, t) \\
      & =   \oint_{C_{\epsilon'}} \frac{\dd z}{z}
      \diffcontourA{k+1} \cdots \diffcontourA{N}
      J_{k}(w_{k+1}, \dots, w_N; z, \zeta, t)
      - u_0(\zeta) \mathds{1}(k = N)
    \end{split}
  \end{equation}
  for $k = 1, \dots, N$ with the integrands given as follows
  \begin{align*}
    J_0(w_{1}, \cdots, w_N; z , \zeta , t)
    & := \frac{ h_{\vec{w}} (Y) u_{\vec{w}} (X) e^{t E(\vec{w}) } }{p_z( \vec{w} ; \zeta ) \prod_{i = k}^N q_z(w_i)}, \qquad \mbox{ and}  \\
    J_k(w_{k+1}, \cdots, w_N; z, \zeta , t) & := \sum_{\SumLambdasR[k]}
                                              \Bigg( \frac{ h_{\vec{w}} (Y) u_{\vec{w}} (X) e^{t E(\vec{w})} }{ p_z(\vec{w} ; \zeta ) \prod_{i=1}^k w_i q_z'(w_i) \prod_{i=k+1}^N q_z(w_i) } \Bigg)_{(w_1, \cdots, w_k) = (\lambda_1, \cdots, \lambda_k)}.
  \end{align*}
  Note that the $J_k$ integrands are defined similarly to the $I_k$ integrands from \eqref{e:I_k}, except that the summations are taken over the set $\calQ_1(z)$ instead of the set $\calQ(z)$.
  Also, note that there is a negative sign infront of the $u_0$ function due to the negative sign in the definition of the $u_0$ function given by \eqref{e:u0_definition}. The result then follows from this sequence of equalities.
  We provide the details of the recursive equation \eqref{e:J_k_recurrence} below.

  Fix $k = 1, \dots, N$ and consider the left side of \eqref{e:J_k_recurrence}.
  Take the integral with respect to the $w_k$-variable by residue computation.
  There are two possible poles from the zeros of the functions $q_z$ and $p_z$. The poles from $q_z(w_{k}) = 0$ are given by $w_{k} = \lambda_{k} \in \calQ_1(z)$.
  The unique pole from $p_z(\vec{w}; \zeta)_{(w_1, \dots, w_{k-1}) = (\lambda_1, \dots, \lambda_{k-1})} = 0$ is given by \eqref{e:P2_pole_mu} $w_{k} = \mu$, where
  \begin{equation}
    \label{e:pole_mu}
    \mu = \left( \frac{1}{1 + (-1)^N \zeta^{-L} z^L \prod_{i \neq k}(1 - w_i^{-1})^{-1} } \right)_{(w_1, \dots, w_{k-1}) = (\lambda_1, \dots, \lambda_{k-1})}
  \end{equation}

  Let us first consider the contribution form the pole $w_{k} = \lambda_{k} \in \calQ_1(z)$. This pole always lies in the region bounded by the contours $1 + C_{\epsilon_2}$ and $1 + C_{\epsilon_3}$.
  Indeed, we have $|\lambda_{k} - 1| \sim (\epsilon')^{L/N}$ and $\beta_3 < \tfrac{L}{N} < \beta_2$.
  This gives rise to the first term on the right side of \eqref{e:J_k_recurrence}.

  Now, let us consider the contribution from the pole $w_{k} = \mu$ given in \eqref{e:pole_mu}.
  In this case, the pole may or may not lie in the region bounded by the contours $1 + C_{\epsilon_2}$ and $1 + C_{\epsilon_3}$. We have
  \[
    |1 - \mu^{-1} | \sim (\epsilon')^{L - (k-1) (L/N) - m_2 \beta_2 - m_3 \beta_3 }.
  \]
  where $m_j$ denotes the number of $w_i$-variables with $w_i \in 1 + C_{\epsilon_j}$, for $j=2,3$ and $i > k$. Note that 
  \[
    L - (k-1)(L/N) - m_2 \beta_2 - m_3 \beta_3 \leq L - (N-1) \beta_3 < \beta_2
  \]
  due to \eqref{e:beta3} and $L/N < \beta_2$. On the other hand, for $k \neq N$, the exponent $L - (k-1)(L/N) - m_2 \beta_2 - m_3 \beta_3 $ may be less than or greater than $\beta_3$.
  Still, we do know that the exponent is not equal to $\beta_3$ due to the second condition on \eqref{e:beta3} for $k \neq N$.
  Also, for the case $k= N$ which implies $m_2 = m_3=0$, we know that 
  \[
    L - (k-1)(L/N) - m_2 \beta_2 - m_3 \beta_3 = L/N > \beta_3,
  \]
  meaning that the pole lies in the region bounded by the contours $1 + C_{\epsilon_2}$ and $1 + C_{\epsilon_3}$.
  Then, let us consider the other cases.
  If $L - (k-1)(L/N) - m_2 \beta_2 - m_3 \beta_3  < \beta_3$ which necessarily means that $k \neq N$, the pole does not lie in the region bounded by the contours $1 + C_{\epsilon_2}$ and $1 + C_{\epsilon_3}$, meaning that there is no contribution of this pole in the residue computation and \eqref{e:J_k_recurrence} follows for this case. If $L - (k-1)(L/N) - m_2 \beta_2 - m_3 \beta_3  > \beta_3$ and $k \neq N$, the pole does lie in the region bounded by the contours $1 + C_{\epsilon_2}$ and $1 + C_{\epsilon_3}$.
  We treat this case in \Cref{l:u0_ind_step}, below, to show that this contribution is equal to zero.
  Lastly, if $k =N$, we have already noted that the pole does lie in the region bounded by the contours $1 + C_{\epsilon_2}$ and $1 + C_{\epsilon_3}$.
  Additionally, recall the computation from \eqref{e:pz_residue} that
  \[
    \PzResidue.
  \]
  Then, the residue computation in this case leads to the term $u_0(\zeta)$, given in \Cref{l:u0_residue}, on the right side of \eqref{e:J_k_recurrence}. This establishes \eqref{e:J_k_recurrence} and the result of this lemma.
\end{proof}

We conclude this section with a technical lemma that we used in the proof of \Cref{l:u0_res_int}. 

\begin{lemma}
  \label{l:u0_ind_step}
  Let $N \geq 2$ and set $\epsilon_3 = (\epsilon')^{\beta_3}$ with $\beta_3$ satisfying \eqref{e:beta3}. Additionally, take integers $1 \leq k \leq N-1$, $m_2, m_3 \geq 0$ such that $k + m_2 + m_3 = N$ and $L - (k-1)(L/N) - m_2 \beta_2 - m_3 \beta_3 > \beta_3$. Then, for any partition $I_2 \sqcup I_3 = \{ k+1, \dots, N \}$, we have
  \begin{equation}
    0 = \oint_{C_{\epsilon'}} \frac{\dd z}{z}
    \Bigg( \prod_{i \in I_2} \oint_{1 + C_{\epsilon_2}} \frac{\dd w_i}{w_i} \Bigg)
    \Bigg( \prod_{i \in I_3} \oint_{1 + C_{\epsilon_3}} \frac{\dd w_i}{w_i} \Bigg)
    \Bigg(
    \frac{(w_{k}-1) h_{\vec{w}} (Y) u_{\vec{w}} (X) e^{t E(\vec{w}) }}
    { \prod_{i=1}^{k-1}  w_i q_z'(w_i)  \prod_{i=k}^N q_z(w_i) }
    \Bigg)_{\subalign{ & (w_1, \dots, w_{k-1}) = (\lambda_1, \dots, \lambda_{k-1}) \\ & \phantom{(} w_k = \mu}} 
  \end{equation}
  so that $\lambda_i \in \calQ_1(z)$ for $1 \leq i \leq k-1$ and the variable $\mu$ is a root (with respect to  $w_k$) of $p(\vec{w}; \zeta)$ with $(w_1, \dots, w_{k-1}) = (\lambda_1, \dots, \lambda_{k-1})$.
\end{lemma}

\begin{proof}
  The proof is similar to the proof of \Cref{l:u0_zero_int}.
  In particular, we bound the integral by some positive power of $\epsilon'$ before taking $\epsilon' \rightarrow 0$.
  We just need to introduce more technical bounds on the terms of the integrand since some of the variables are now determined by roots of polynomials.
  
  Let us approximate the different terms in the integrand.
  Note that $|\lambda_j-1| \sim (\epsilon')^{L/N}$ for $j =1, \dots, k-1$ since $\lambda_j \in \calQ_1(z)$.
  Also, for $j > k$,  we have $|w_j-1| =  \epsilon_i $ with $i=2$ or $i=3$, since $w \in (1 + C_{\epsilon_2}) \cup (1+ C_{\epsilon_3})$.
  Let $m_i$ be equal to the number of $w_j$-variables that lie on the contour $1 + C_{\epsilon_i}$, for $i=2, 3$ with $j >k$.
  Then, we have $|\mu-1| \sim (\epsilon')^{b_1}$ with $b_1 = L - (k-1)(L/N) - m_2 \beta_2 - m_3 \beta_3$ where $w_k = \mu$ is a root of $p_z(\vec{w}; \zeta)$ with $(w_1, \dots, w_{k-1}) = (\lambda_1, \dots, \lambda_{k-1})$. 
  
  It is straightforward to check that all the terms in the integrand are bounded by constants except for the terms $q_z$ and $A_{\sigma}$.
  For the term $q_z$, we have the following approximation,
  \begin{align}
    \label{e:q_approx1}
    j > k, \qquad
    |q_z(w_j) - 1| & \sim \begin{cases}
      (\epsilon')^{L-N \beta_2 }, \quad & w_j \in 1+C_{\epsilon_2}, \\ 
      (\epsilon')^{L-N \beta_3 }, \quad & w_j \in 1 +C_{\epsilon_3}, \\
    \end{cases} \\
    |q_z(\mu) - 1| & \sim (\epsilon')^{L - N b_1}. \label{e:q_approx1_mu}
  \end{align}
  Recall the condition on the exponents, $\beta_3 < L/N < \beta_2 < L/(N-1)$.
  It follows that the exponent, in the first case of \eqref{e:q_approx1} above, is always negative and, in the second case of \eqref{e:q_approx1}, is always positive.
  For \eqref{e:q_approx1_mu}, the exponent may be positive or negative, but not equal to zero due to the second condition in \eqref{e:beta3}.
  Then, we have the following approximations,
  \begin{align}
    \label{e:q_approx}
    j > k, \qquad
    |q_z(w_j)^{-1}| & \sim
                      \begin{cases}
                        (\epsilon')^{N\beta_2 -L}, \quad & w_j \in 1+C_{\epsilon_2}, \\
                        1, \quad & w_j \in 1 +C_{\epsilon_3},
                      \end{cases} \\
    |q_z(\mu)^{-1}| & \sim
                      \begin{cases}
                        (\epsilon')^{N b_1-L}, \quad & \, b_1 > \frac{L}{N}, \\
                        1, \quad & b_1 < \frac{L}{N}.
                      \end{cases} \label{e:q_approx_mu}
  \end{align}
  
  Next, we consider the $A_{\sigma}$ terms.
  The magnitude of these terms depends on $\sigma, k, m_2, m_3$. In particular, let us fix $k$, $m_2$ and $m_3$ such that $k + m_2 + m_3 = N$.
  Recall that
  \[
    A_\sigma(\vec{w}) = \prod_{j=1}^N (1 - w_j^{-1})^{j} \prod_{j=1}^N (1 - w_{\sigma(j)}^{-1})^{-j},
  \]
  where the first part is independent of the choice of $\sigma$ and is maximized for the choice $w_j \in 1 + C_{\epsilon_2}$ for $j = k+1, \dots, k+m_2$ and $w_j \in 1 + C_{\epsilon_3}$ for $j = k+m_2+1, \dots, N$, since $\beta_3 < \beta_2$, and the second term can be maximized according to this choice for some specific $\sigma$.
  Then, we have $|A_{\sigma}| = \mathcal{O}((\epsilon')^{b_2})$ with
  \begin{equation}
    b_2 = \begin{cases}
      -(k-1)m_3 (L/N) - m_3 b_1 - m_2 m_3 \beta_2 + m_3(m_2 +k) \beta_3,& \quad \beta_3 < \frac{L}{N}<b_1 < \beta_2, \\
      -(k-1)(m_3+1)(L/N) + (k-1 -m_3) b_1 - m_3m_2 \beta_2 + m_3(m_2+k) \beta_3,& \quad \beta_3 < b_1 < \frac{L}{N} < \beta_2, 
    \end{cases}.
  \end{equation}
  We obtain the value for $b_2$ via the rearragement inequality, where large exponents of $(1 - w_{\sigma(j)}^{-1})$ are paired with large indices $j$. For instance, in the second case, the largest amplitude $(\epsilon')^{b_2}$ can be obtained when
  \begin{align*}
    \sigma^{-1}( \{ 1, \dots, k-1 \} ) & = \{ m_3 + 2, \dots, m_3 + k \}, \\
    \sigma^{-1}( k ) & = m_3 + 1, \\
    \sigma^{-1}( \{ k + 1, \dots, k + m_2 \} ) & = \{ m_3 + k + 1, \dots, N \}, \\
    \sigma^{-1}( \{ k+m_2+1, \dots, N \} ) & = \{ 1, \dots, m_3 \}.
  \end{align*}
  The first case can be treated similarly.
  
  Now, we use the approximations above to give a uniform bound for the integral.
  Let $I_N$ denote the integral with $m_i$ variables, with $i=2,3$, that lie on the contour $1 + C_{\epsilon_i}$.
  Then, we have $|I_N| = \mathcal{O}((\epsilon')^{b_3})$ with
  \begin{equation}
    b_3 = b_2 + (b_1 + (k-1)(L/N)) + (N b_1 -L)\mathds{1}(N b_1 -L >0) + m_2(N \beta_2 -L) + (m_2 \beta_2 + m_3 \beta_3)
  \end{equation}
  The first term is due to the bound on $A_{\sigma}$; the second term is due to the bound on $\prod_{i=1}^k(w_i-1)$, arising from the term in the numerator and the product of $q_z'(w_i)$ in the denominator, with $(w_1, \dots, w_{k-1}) =(\lambda_1, \dots, \lambda_{k-1})$ and $w_k = \mu$; the third term is due to the bound $1/q_z(w_k)$ with $w_k = \mu$; the fourth term is due to the bound on $\prod_{i=k+1}^N (1/q_z(w_k))$; and the last is due to the length of the contours. Then, we have
  \begin{equation}
    \label{e:b3}
    b_3 = \begin{cases}
      L, \\
      (L/N)(k(1-k)+ N(2k-N)) + (N-k)(N-k+1) \beta_3 + (N-k+1)(\beta_2 - \beta_3) m_2,
    \end{cases}
  \end{equation}
  for the different cases $\beta_3 < \frac{L}{N}<b_1 < \beta_2$ and $\beta_3 < b_1 < \frac{L}{N} < \beta_2$, respectively.
  Now, we just need to show that $b_3$ is positive.
  The first case is clear. For the second case, note that, when $k$ is fixed, $b_3$ is minimized when $m_2 = 0$ since $\beta_2 > \beta_3$. Moreover, setting $m_2 = 0$ in \eqref{e:b3}, we obtain a quadratic polynomial in $k$ with leading coefficient $\beta_3 - L/N < 0$, which is a concave parabola. Its minimum is reached either at $k = 1$ or $k = N-1$, giving the following quantities,
  \begin{align*}
    & m_2 = 0, k = 1, \quad
    & & b_3 = -L(N-2) + N (N-1) \beta_3
        > 2L - N \beta_2
        > \tfrac{L}{N-1} ( N-2 ) \geq 0, \\
    & m_2 = 0, k = N-1, \quad
    & & b_3 = \tfrac{L}{N} (N-2) + 2 \beta_3 > 0.
  \end{align*}
  where in the first case, we use $(N-1) \beta_3 > L - \beta_2$ and $\beta_2 < \frac{L}{N-1}$ Thus, $b_3 > 0$, for all cases, meaning that the integral goes to zero as $\epsilon'$ goes to zero.
  This establishes the result.
\end{proof}

\begin{rem}
  Note that \Cref{l:u0_ind_step} does not hold for the case $k=N$.
  The argument above breaks due to the approximation \eqref{e:q_approx_mu}, since the second condition in \eqref{e:beta3} is never true for $k=N$.
  In particular, the exponent on \eqref{e:q_approx_mu} is identically zero.
  This means that there may be pole arising from this term.
  In fact, this is the case and that is reflected in \Cref{l:u0_res_int}.
\end{rem}

\subsection{Residue computations for 1-fold integral}

In this section, we give the last residue computations needed for the proof of \Cref{t:main}.
That is, the computations needed to establish the equality between \eqref{e:main_1} and \eqref{e:main_eigenfunctions}.
The main result is given by \Cref{p:last_res}, which is broken up into \Cref{l:residue_computations} and \Cref{l:residue_simple}.
The last lemma in this section, \Cref{l:well_def_hol}, is a technical lemma showing that the integrand is well defined without any square root singularies arising from the double roots of the decoupled Bethe equations.

Recall, the function defined in \Cref{l:contour_deformation_induction}
\begin{equation}\label{e:i_N}
  I_N (z , \zeta,t) :=
  \sum_{\lambda_1, \cdots, \lambda_N \in \calQ(z)}
  \frac{ \sum_{\sigma \in S_N} A_{\sigma} \prod_{i=1}^{N} \lambda_{\sigma(i)}^{y_{\sigma(i)} - x_i} e^{t E(\vec{\lambda})} }{ p_{z}(\vec{\lambda}; \zeta ) \prod_{i=1}^N \lambda_i q_z'(\lambda_i) },
\end{equation}
with the set $\calQ(z)$ given by \eqref{e:bethe_roots_deformed}, the coefficients $A_{\sigma}$ given by \eqref{e:bethe_coeffs}, the function $p_z$ given by \eqref{e:dBE2}, the functions $\lambda_i q_z'(\lambda_i)$ and $E(\vec{\lambda})$ given in Theorem \eqref{e:transition_prob}, and the particle configurations $(x_i)_{i=1}^N, (y_i)_{i=1}^N \in \calX_{L,N}$.
We compute the contour integral of this function over a pair of arbitrarily large and small contours for general values of $\zeta \in \bbU$.

\begin{proposition}
  \label{p:last_res}
  \label{t:bethe_expansion}
  For all but finitely many $\zeta \in \bbU$, the following holds,
  \begin{equation}
    \oint_{\epsilon'}^{R'} \frac{\dd z}{z} I_N(z, \zeta,t)
    = \sum_{(\lambda_1, \dots, \lambda_N)\in \calR(\zeta) }
    \frac{\sum_{\sigma \in S_N} A_{\sigma} \prod_{i=1}^{N} \lambda_{\sigma(i)}^{y_{\sigma(i)} - x_i} e^{t E(\vec{\lambda})}}{\left(L - \sum_{i=1}^N \frac{1}{ \lambda_i - (1- \rho) } \right) \prod_{i=1}^N \lambda_i q_{z'}'(\lambda_i)}
  \end{equation}
  with the function $I_N(z, \zeta,t)$ given by \eqref{e:i_N}, the set $\calR(\zeta)$ given by \eqref{e:bethe_roots}, $\epsilon' = (R')^{-1}>0$ is arbitrarily small.
\end{proposition}

\begin{proof}
  This result is a direct consequence of \Cref{l:residue_computations} and \Cref{l:residue_simple}.
  Recall that
  \begin{equation}
    \calR(\zeta) = \bigcup_{\substack{z \in \mathcal{Z}(\zeta)\\z \neq 0}} \left( \calQ^N(z) \cap \mathcal{P}(z; \zeta) \right),
  \end{equation}
  see \eqref{e:bethe_roots}, \eqref{e:bethe_roots_deformed}, \eqref{e:coupled_roots}, and \eqref{e:bethe_roots_decomp}.
\end{proof}

The main point of the next lemma is that the non-trivial residues of the contour integral are only due to the roots of the function $p_z(\vec{\lambda}; \zeta)$.

\begin{lemma}
  \label{l:residue_computations}
  Take the function $I_N(z, \zeta, t)$ given by \eqref{e:i_N}. Then, we have
  \begin{equation}
    \oint_{\epsilon'}^{R'} \frac{\dd z}{z} I_N(z, \zeta,t)
    = \sum_{\substack{z'\in \calZ(\zeta) \\ z' \neq 0}} \Res_{z = z'} \frac{ I_N(z , \zeta,t) }{z} 
  \end{equation}
  with the set $\calZ(\zeta)$ given by \eqref{e:calZ_def} and $\epsilon' >0$ arbitrarily small.
\end{lemma}

\begin{proof}
  The equality is obtained by residue computations.
  Note that residue computations are well defined due to \Cref{l:well_def_hol}.
  Indeed, the only issue may arise at $z = z_0 = (1-\rho)^{1-\rho}(- \rho)^{\rho}$ where there may be a square root singularity when $q_z'(\lambda_i)$ vanishes, everywhere else the roots $\lambda_i \in \calQ(z)$ have multiplicity one and are locally holomorphic.
  We resolve the potential issue at the square root singularity by \Cref{l:well_def_hol}.
  Note that $p_z(\vec{\lambda}; \zeta ) \neq 0$ if $q_z'(\lambda_i)=0$ due to \Cref{p:double_root_no_solution}, meaning that we may apply the result of \Cref{l:well_def_hol}.
  Then, we have that the integrand is holomorphic at $z = z_0$ and meromorphic for $z \neq 0, \infty$.
  Moreover, the poles of the integrand are located at the roots of $p_z(\vec{\lambda}; \zeta )$ with $\lambda_i \in \calQ(z)$, which is equivalent to the set $\calZ(\zeta) \cap \bbC^{*}$. The result then follows by residue computation.
\end{proof}

Now, we actually compute the residues arising from the contour integral.
We will start by showing that the poles are simple for generic $\zeta \in \bbU$, i.e.~for all but finite values of $\zeta$, so that the corresponding residues can be obtained easily by taking derivatives.

\begin{lemma}\label{l:residue_simple}
  Take the function $I_N(z, \zeta, t)$ given by \eqref{e:i_N}.
  Then, for all but finitely many $\zeta \in \bbU$, the pole of the function $I_N(z;t)/z$ at $z = z'$, for $z' \in \calZ(\zeta) \cap \bbC^{*}$, is simple.
  Moreover, when the pole is simple, we have
  \begin{equation}
    \underset{z = z'}{\Res}\,\, \frac{I_{N}(z;t)}{z}
    = \sum_{\vec{\lambda} \in \calQ(z')^N \cap \mathcal{P}(z'; \zeta)} \frac{\sum_{\sigma \in S_N} A_{\sigma} \prod_{i=1}^{N} \lambda_{\sigma(i)}^{y_{\sigma(i)} - x_i} e^{t E(\vec{\lambda})}}{ \left(L - \sum_{i=1}^N \frac{1}{ \lambda_i - (1- \rho)} \right) \prod_{i=1}^N \lambda_i q_{z'}'(\lambda_i)},
  \end{equation}
  where $\vec{\lambda} = (\lambda_1, \dots, \lambda_N)$.
\end{lemma}

\begin{proof}
  The poles of the function $I_N(z;t)/z$ at $z = z'$ for $z' \in \calZ(\zeta) \cap \bbC^{*}$ correspond to roots of $p_z(\vec{\lambda}; \zeta)$ with $\vec{\lambda} = (\lambda_1, \dots, \lambda_N)$ and $\lambda_i \in \calQ(z)$.
  Note that a pole $z' \in \calZ(\zeta) \cap \bbC^{*}$ is not simple if and only if $p_z(\vec{\lambda}; \zeta)_{z = z'} = \partial_zp_z(\vec{\lambda}; \zeta)_{z = z'} =0$.
  In particular, we can write this system of equations as follows
  \begin{equation}\label{e:sys_eq}
    \zeta^{-L} = (-1)^{N+1} z^{-L} \prod_{i=1}^N(1 - \lambda_i^{-1}),
    \quad
    L = \sum_{i=1}^N\frac{1}{\lambda_i - (1- \rho)}.
  \end{equation}
  Note that the second equation is independent of $\zeta$ and it has finite number of solutions since the right side is an algebraic equation in the variable $z$ that is not identically equal to $L$. On can check this fact by noting that the limit of the right side is equal to zero as $z \rightarrow \infty$.
  Then, the system of equations given by \eqref{e:sys_eq} has finite number of solutions since we showed that the second equation has finite number of solutions and, given this, the first equation determines finitely many values of $\zeta$.
  Thus, it follows that the poles are simple except for finite number of $\zeta \in \bbU$.
  The rest of the result of the lemma follows from straightforward residue computations. 
\end{proof}

This last lemma of the section, \Cref{l:well_def_hol}, is a technical result.
In the computations above, we compute the contour integrals by contour deformation and the residue theorem.
So, we must check that the integrand is a meromorphic function.
Potentially, the integrand may be a multi-valued function since it is given as a function on the roots of $q_z(w)$.
This issue only arises in neighborhood of a root of $q_z(w)$ with higher multiplicity.
Otherwise, one may check that the integrand is locally meromorphic.
Below, we check that the integrand is well defined by showing that it is single-valued and, additionally, we show that the integrand is holomorphic at the double root of $q_z(w)$.

\begin{lemma}\label{l:well_def_hol}
  Let $g(z; \vec{w})$ be a symmetric function on $\vec{w} = (w_1, \dots, w_N)$ and set
  \begin{equation}\label{e:f}
    f(z) := \sum_{\lambda_1, \cdots, \lambda_N \in \calQ(z)}
    \frac{\det [(1 -\lambda_j^{-1})^{j-i}\lambda_j^{-x_i}]_{i,j=1}^N}{\prod_{i=1}^N q_z'(\lambda_i)} g(z; \vec{\lambda}).
  \end{equation}
  Also, assume that $g(z; \vec{w})$ is analytic in a neighborhood of $z_0:= (1-\rho)^{(1-\rho)}(- \rho)^{\rho}$ and $w_i \neq 0, \infty$ for $i=1, \dots, N$.
  Then, the function $f(z)$ is holomorphic on a neighborhood of $z_0$.
\end{lemma}

\begin{proof}
  We need to prove two things: (1) that the function $f(z)$ is indeed well-defined as a function of $z \in \bbC$, and (2) that the function $f(z)$ is holomorphic.
  A priori, the roots $\lambda_i$, for $i =1, \dots, L$, of the polynomial $q_z(w)$ are defined on a branched cover of $\bbC$.
  We show that $f$ is indeed defined on the base curve of the branch cover, i.e.~on $z \in \bbC$, since it is given as a symmetric function on the roots $\lambda_i$.
  Then, we show that $f(z)$ is holomorphic.
  Note that the function is locally holomorphic away from the branching point $z_0$ since the roots of $q_z(w)$ are locally holomorphic.
  Then, it suffices to show that $f(z)$ is holomorphic at $z_0$.
  Moreover, by Morera's theorem, $f$ is holomorphic at $z_0$ if
  \begin{equation}\label{e:morera}
    \oint_{\gamma} f(z) \dd z =0
  \end{equation}
  for every closed piecewise $C^1$ curve $\gamma$ in a neighborhood of $z_0$ and containing $z_0$. We show this below.

  First, we show that the function $f(z)$ is well-defined on $\bbC$ since it is a symmetric function on the roots of the polynomial $q_z(w)$, meaning that we do not have to choose branch cuts for the roots.
  To be more precise, we may define a symmetric function in $w_1, \dots, w_L$,
  \[
    h(w_1, \dots, w_L; z) := \sum_{1 \leq k_1, \dots, k_N \leq L}
    \frac{\det [(1 -w_{k_j}^{-1})^{j-i} w_{k_j}^{-x_i}]_{i,j=1}^N}{\prod_{i=1}^N q_z'(w_{k_i})} g(z; w_{k_1}, \dots, w_{k_N}),
  \]
  and write
  \begin{equation}
    f(z) = h(\lambda_1, \dots, \lambda_L; z),
  \end{equation}
  where $\lambda_i$'s denote the roots of $q_z(w)$, i.e.~$\calQ(z) = \{ \lambda_1(z), \dots, \lambda_L(z) \}$.
  Now, fix a base point $z \neq z_0$ and let $\gamma : [0,1] \rightarrow \bbC$ be a closed curve with $\gamma(0)= \gamma(1) = z$.
  Also, fix a labeling of the roots of $q_z(w)$, $\lambda_i$ for $i=1, \dots, L$ and the base point $z$.
  Let $\lambda_i(t)$ be the analytic continuation of $\lambda_i$ along the curve $\gamma$.
  We have that $\lambda_i(0) = \lambda_i$ and $\lambda_i(1) = \lambda_{\sigma(i)}$ for some permutation $\sigma \in S_L$.
  In particular, we have
  \begin{equation}
    \begin{split}
      f(\gamma(0)) & = h(\lambda_1(0), \dots, \lambda_L(0); \gamma(0) ) = h(\lambda_1, \dots, \lambda_L; \gamma(0) )\\
      & = h(\lambda_{\sigma(1)}, \dots, \lambda_{\sigma(L)}; \gamma(0) )
      = h(\lambda_1(1), \dots,\lambda_L(1); \gamma(1) )
      = f(\gamma(1))
    \end{split}
  \end{equation}
  since $h$ is symmetric on the first $L$ arguments and $\gamma(0)= \gamma(1)=z$.
  Since the denominator of $f(z)$ only vanishes at $z = z_0$ for a small neighboorhood of $z_0$, we have that $f$ is independent of the analytic continuation along $\gamma$ and well-defined for a punctured disk of $z_0$, $B = \{z \mid \delta_0>|z-z_0|>0 \}$ for some $\delta_0 > 0$ small enough.
  Note that the above argument only holds for $f$ but not for individual terms inside the summation defining $f$, underlying the importance of the symmetrization.
  
  Now, note that $f(z)$ is locally holomorphic on $B = \{z \mid 0< |z-z_0| < \delta_0 \}$ for $\delta_0 > 0$ small enough.
  Then, we have
  \begin{equation}\label{e:m_def}
    \oint_{\gamma} f(z) \dd z = \oint_{|z-z_0| = \delta} f(z) \dd z
  \end{equation}
  for every $\gamma \subset B$ containing $z_0$ and $ 0< \delta <\delta_0$.
  Thus, by Morera's theorem, it suffices to show that the limit of the right side of \eqref{e:m_def} is equal to zero as $\delta$ goes to zero.

  Let us now write
  \begin{equation}\label{e:m_def2}
    \oint_{|z-z_0| = \delta} f(z) \dd z = \int_{\substack{|z-z_0| = \delta\\ \myre(z-z_0) \geq 0}} f(z) \dd z + \int_{\substack{|z-z_0| = \delta\\ \myre(z-z_0) \leq 0}} f(z) \dd z.
  \end{equation}
  In the following, we show that each term on the right side goes to zero as $\delta \rightarrow 0$. That is,
  \begin{equation}\label{e:m_limit}
    \lim_{\delta \rightarrow 0^+} \int_{\substack{|z-z_0| = \delta\\ \myre(\pm (z-z_0) ) \geq 0}} f(z) \dd z  = 0.
  \end{equation}
  We take this decomposition so that we can express the roots of $q_z(w)$ as single valued series expansion with a fixed labeling. The argument for the both cases is practically identical. 
  
  We show that the terms in the sum of \eqref{e:f} cancel out in pairs in the limit $\delta \rightarrow 0$.
  We take $\delta >0$ in \eqref{e:m_def2} to be arbitrarily small so that we may take a series expansion of the $\lambda_i \in \calQ(z)$.
  Recall that the integrand is given as a sum over the set $\calQ(z)^N$.
  The terms in the sum will cancel in pairs due to the fact that the zeros of the denominator, i.e.~$q_z'(\lambda_i) =0$, arise from a double root of $q_z(\lambda)$.
  Below, we will first consider the series expansion of $\lambda_i$ in terms of $z \in B $ with $\myre(\pm (z-z_0)) >0$.
  Then, we introduce the pairing among the terms in the integrand.
  Lastly, we show the cancellations that lead to the result. 
  
  Consider $z \in \bbC$ with $0 < |z- z_0| < \delta$, $\myre (\pm(z-z_0))$ and $\delta > 0$ arbitrarily small. In this region, we may fix a labeling on the elements of the set $\calQ(z)$,
  \begin{equation}\label{e:Q_labelling}
    \calQ(z) = \{\lambda_1, \cdots, \lambda_L\},
  \end{equation}
  so that each $\lambda_i =\lambda_i(z)$, for $i=1, \cdots, N$, is a single-valued holomorphic function for $0 < |z- z_0| < \delta$ and $\myre(\pm(z-z_0))$.
  The function $f(z)$ is independent of the labelling since $f$ is symmetric on the $\lambda_i$.
  Note that the elements $\lambda_i$'s are pairwise distinct since  $z \neq z_0$.
  In fact, the roots $\lambda_i$'s only have higher multiplicity when $z = z_0$.
  Moreover, the multiplicity of each root is at most two when $z =z_0$, since there are no solutions to the system of equations $q_z(\lambda) = q_z'(\lambda) = q_z''(\lambda) = 0$ for any $z \neq  0 , \infty$ since $L\neq N$.

  Let $\lambda_0 = 1- \rho$ denote the double root in $\calQ(z_0)$.
  We distinguish two roots, say $\lambda_{1}, \lambda_{2} \in \calQ(z)$, to be given by
  \begin{equation}
    \label{e:root_approx}
    \lambda_{i} = \lambda_0  + (-1)^i c \, (z- z_0)^{1/2} + \calO(z- z_0), \quad i=1,2
  \end{equation}
  for some constant $c \in \bbC$, $0 < |z - z_0 | < \delta$, and $\myre(\pm(z-z_0))$.
  The approximation of the two distinguished roots $\lambda_{1}, \lambda_2 \in \calQ(z)$ follows by solving the functional equation $q_z(\lambda_i) = 0$ around $z = z_0$.

  We now give a paring among the terms in the function that will lead to cancellations.
  In the definition of the function $f(z)$, given by \eqref{e:f}, we take a sum over all $N$-tuples of elements of $Q(z)$, i.e.~a sum over all $(\lambda_{i_1}, \cdots, \lambda_{i_N})$ with $\lambda_{i_n} \in Q(z)$\footnote{We have introduced sub-indices since we have fixed a labelling, \eqref{e:Q_labelling}, for the elements of the set $\calQ(z)$}.
  We denote each term in the sum by
  \begin{equation}
    f^{(i_1, \cdots, i_N)}(z) := \frac{\det [(1 -\lambda_{i_m}^{-1})^{m-n}\lambda_{i_m}^{-x_{i_n}}]_{n,m=1}^N}{\prod_{n=1}^N q_z'(\lambda_{i_n})} g(z; \lambda_{i_1}, \dots, \lambda_{i_N}),
  \end{equation}
  for each $N$-tuple $(\lambda_{i_1}, \cdots, \lambda_{i_N}) \in \calQ(z)^N$. Note that the function vanishes,
  \begin{equation}
    f^{(i_1, \cdots, i_N)}(z) =0,
  \end{equation}
  if $i_n = i_m$ for some $n \neq m$ and $0 < |z - z_0| < \delta$, since the denominator does not vanish for $\delta$ small enough.
  Thus, we may assume that the $N$-tuples $(i_1, \cdots, i_N)\in Q(z)^N$ are pairwise distinct in the sum for the function $f(z)$.
  We then denote
  \begin{equation}
    \Delta^N := \left\{(i_1 , \cdots, i_N) \mid 1 \leq i_k \leq N, \quad i_n \neq i_m\quad\text{for}\quad n \neq m\right\}
  \end{equation}
  to be the set of pairwise distinct $N$-tuples of indices.
  It follows that we may write the function as follows
  \begin{equation}
    f(z) = \sum_{ i \in \Delta^N } f^i(z)
  \end{equation}
  with $i = (i_1, \cdots, i_N)$. 
  
  We now introduce a bijection $d:\Delta^N \rightarrow \Delta^N$ given by $d(i) := (d(i)_1, \cdots, d(i)_N)$ so that
  \begin{equation}
    d(i)_n =
    \begin{cases}
      2, \quad & i_n= 1, \\
      1, \quad  & i_n = 2, \\
      i_n, \quad & i_n \neq 1, 2,
    \end{cases}
  \end{equation}
  for $i  = (i_1, \cdots, i_N) \in \Delta^N$.
  The bijection is the action induced on $\Delta^N$ by the natural action of the transposition $(1\,2)$ on the set $\{ 1, \cdots, N \}$.
  Also, recall that the indices $1$ and $2$ correspond to the distinguished roots that approach the double root when $z \rightarrow z_0$, see \eqref{e:root_approx}.
  As a consequence, we may write
  \begin{equation}
    \label{e:f_sum_bijection}
    f(z) = \frac12 \sum_{i \in \Delta^N} \Big( f^{i}(z) + f^{d(i)}(z) \Big).
  \end{equation}

  Below, we establish the result of the Lemma by showing the following (approximate) cancellation
  \begin{equation}\label{e:pair_cancel}
    f^{i}(z) + f^{d(i)}(z) = \calO(1)
  \end{equation}
  for any $i \in \Delta^N$ as $z \rightarrow z_0$. By taking the sum over all $i \in \Delta^N$ of the identity above, we obtain that $f(z) = \calO(1)$ as $z \rightarrow z_0$ and, then, the limit given by \eqref{e:m_limit} follows. Moreover, the result for the lemma then  follows by Morera's Theorem \eqref{e:morera} and the identities \eqref{e:m_def} and \eqref{e:m_def2}.

  Below, for each $i \in \Delta^N$, we show the cancellation \eqref{e:pair_cancel} based on three different cases: (1) $|i \cap \{1,2 \}| = 0$, (2) $|i \cap \{ 1,2\}| =1$, and (3) $|i \cap \{1, 2\}| =2$.

  In the first case, $|i \cap \{1,2\}| = 0$, the functions $f^{i}(z)$ and $f^{d(i)}(z)$ do not have any poles for $|z-z_0| < \delta$ since none of the roots $\lambda_{i_n}$ or $\lambda_{b(i)_n}$, $n=1, \cdots, N$, correspond to the double roots when $z =z_0$.
  Then, we have
  \begin{equation}\label{e:pair_cancel_1}
    f^{i}(z)  ,  f^{d(i)}(z) = \mathcal{O}(1).
  \end{equation}
  since the denominator only vanishes when $\lambda_j = \lambda_1$ or $\lambda_j = \lambda_2$ for some $j$.
  This establishes \eqref{e:pair_cancel} for the case (1) $|i \cap \{ 1, 2 \}| = 0$.
  
  In the second case, $|i \cap \{1, 2 \}| =1$, we assume without loss of generality that $i \cap \{1, 2 \} = i_k =1$ for some $k \in \{ 1, \cdots, N \}$.
  By \eqref{e:root_approx}, we have
  \begin{equation}
    \begin{split}
      q_z'(\lambda_1) & = c_1 (z - z_0)^{1/2} + \calO(z - z_0), \\
      q_z'(\lambda_2) & = - c_1 (z - z_0)^{1/2} + \calO(z - z_0),
    \end{split}
  \end{equation}
  for some constant $c_1 \in \bbC$ independent of $z$ with $0 <|z -z_0| < \delta$.
  Moreover, since $i \cap \{1, 2\} =1$, the rest of the terms in the denominator do not vanish.
  Now, note that the numerators of the functions $f^i(z)$ and $f^{d(i)}(z)$ are equal when $z = z_0$, since $\lambda_1 = \lambda_2$ and the rest of the $\lambda_{i_n}$ and $\lambda_{d(i)_n}$ are equal.
  Then, we have
  \begin{equation}
    f^{i}(z)  = \frac{c_2}{(z - z_0)^{1/2}} \Big( 1 + \calO( \sqrt{z - z_0} ) \Big), \quad f^{d(i)}(z)  = -\frac{c_2}{(z - z_0)^{1/2}} \Big( 1 + \calO( \sqrt{z - z_0} ) \Big),
  \end{equation}
  for some non-zero constant $c_2$, which easily leads to $f^{i}(z) + f^{d(i)}(z) = \calO(1)$. This establishes the result in the second case $|i \cap \{1, 2\}| = 1$.
  
  In the third case, $|i \cap \{1, 2\}|= 2$, we assume $i_k = 1$ and $i_s = 2$ for some $k, s \in \{1, \cdots, N \}$.
  Similar to the second case, we have
  \begin{equation}
    \prod_{n=1}^N  q_z'(\lambda_{i_n})
    = \prod_{n=1}^N  q_z'(\lambda_{b(i)_n})
    = c_3 (z - z_0) + \calO \big( (z - z_0)^{3/2} \big)
  \end{equation}
  for some constant $c_3 \in \bbC$ independent of $z$.
  Note that the determinantal part in the numerator of each function, $f^i(z)$ and $f^{d(i)}(z)$, vanishes when some of the $\lambda_i$'s coincide.
  This implies that $\lambda_1 - \lambda_2$ is a factor of the determinant and it is not hard to show that its multiplicity is 1.
  Hence, as $z \rightarrow z_0$ and due to the determinant factor and $\lambda_1 = \lambda_2$ for $z = z_0$, we obtain that
  \begin{equation}
    \det \left[ (1-\lambda_{i_m}^{-1})^{m-n} \lambda_{i_m}^{x_n} \right]_{m,n=1}^N g(z;\vec{\lambda})
    = (\lambda_1 - \lambda_2) g_1(z; \vec{\lambda})
    = c_4 (z - z_0)^{1/2} g_1(z; \vec{\lambda})
  \end{equation}
  for some holomorphic function $g_1$ on the entries $\lambda_{i_n}$ that do not vanish at $z = z_0$ and some constant $c_4$.
  This gives
  \begin{equation*}
    f^{i}(z)  = \frac{c_5 (z-z_0)^{1/2}}{z - z_0} \big( 1 + \calO( \sqrt{z-z_0} ) \big), \quad
    f^{d(i)}(z)  = -\frac{c_5 (z-z_0)^{1/2}}{z - z_0} \big( 1 + \calO( \sqrt{z-z_0} ) \big),
  \end{equation*}
  leading to $f^{i}(z) + f^{d(i)}(z) = \calO(1)$.
  This establishes the result in the third case $|i \cap \{ 1, 2 \} | = 2$.

  Now, we can deduce \eqref{e:morera} by \eqref{e:m_def}, \eqref{e:f_sum_bijection} and \eqref{e:pair_cancel}.
\end{proof}

\section{Current observable and asymptotics}
\label{s:current}

In this section, we give the proofs of the results presented in \Cref{s:current_results}.

\subsection{Cumulative distribution function for the current observable}\label{s:cdf_proof}

We begin by giving a proof for \Cref{p:current_cdf} which gives an exact formula for the cumulative distribution function (CDF) of the current observable.
Additionally, we provide some technical lemmas needed for the proof of \Cref{p:current_cdf}.
We start with the proof of \Cref{p:current_cdf} and provide the technical lemmas afterwards.

\begin{proof}[Proof of \Cref{p:current_cdf}]
  We obtain the probability distribution by taking the Fourier inverse of the generating function $\QjXgf{L-1}{L}$, given by \eqref{e:local_joint_gf}, and summing over all the possible configurations for $X \in \mathcal{X}_{L,N}$. 
  
  We have
  \begin{equation}
    \bbP_Y(X(t) = X; Q_{L-1}(t) = Q) = \oint_{\bbU}\frac{d \zeta}{\zeta^{1+Q L}}\QjXgf{L-1}{L}.
  \end{equation}
  This follows directly from the definition of $\Qjgf{L-1}{L}$ given in \eqref{e:local_joint_gf}.
  
  We use \Cref{c:QjXgf_contour_integral} to compute the Fourier inverse. Additionally, we use the contour formula \eqref{e:main_1} and the expression \eqref{e:u0_alt} for $u_0(\zeta)$ to express the right side of \eqref{e:cont_local_x}. Then, we have
  \begin{equation}
    \label{e:prob_distr_alt}
    \bbP_Y(X(t) = X; Q_{L-1}(t) = Q) =
    \oint_{\bbU} \frac{d \zeta}{\zeta^{1 + Q L}} \oint_{\epsilon'}^R \frac{d z}{z}
    \sum_{\SumLambdas}' \left(\frac{h_{\vec{\lambda}}(Y) u_{\vec{\lambda}}(X) e^{tE(\vec{\lambda})} }{p_z(\vec{\lambda}; \zeta) \prod_{i=1}^N\lambda_i q'(\lambda_i)}  \right) 
  \end{equation}
  so that the summation excludes the case where $|z|= \epsilon'$ and $\lambda_i \in \calQ_1(z)$ for all $i=1, \dots, N$ and we denote this by writing a prime on the summation sign. Note that the case we are excluding follows from a cancellation due to the expression \eqref{e:u0_alt} for $u_0(\zeta)$.
  
  Next, we compute the contour integral with respect to the $\zeta$-variable by taking a series expansion of the integrand with respect to the $\zeta$-variable.
  In particular, we take a series expansion of the term $1/p_z(\vec{\lambda};z)$ with respect to the $\zeta$-variable.
  The series expansion depends on the location of the $\lambda_i$-variables and the $z$-variable.
  Note that we have
  \begin{equation}\label{e:rate}
    \left|\zeta^L z^{-L} \prod_{i=1}^N(1- \lambda_i^{-1}) \right| \sim
    \begin{cases}
      (\epsilon')^{-L - n_0 / (1-\rho) + n_1 d}, & \quad z \in C_{\epsilon'}, \\
      (\epsilon')^{L}, & \quad z \in C_{R'},
    \end{cases}
  \end{equation}
  where $n_0$ denotes the number of $\lambda_i$-roots that lie in the set $\calQ_0(z)$ and $n_1$ denotes the number of $\lambda_i$-roots that lie in the set $\calQ_1(z)$.
  It follows, by \eqref{e:beta1_bounds} and \eqref{e:beta2_bounds}, that the magnitude of the term in \eqref{e:rate} is greater than 1 if and only if $z \in C_{\epsilon'}$ and $n_1 \neq N$.
  Otherwise, the magnitude of the term in \eqref{e:rate} is less than one.
  Then, for $Q \leq -1$, we have
  \begin{equation}
    \label{e:Q_neg}
    \bbP_Y(X(t) = X; Q_{L-1}(t) = Q)
    = (-1)^{(N+1)Q} \oint_{C_{\epsilon'}} \frac{\dd z}{z}
    \sum_{\SumLambdas}'
    \Bigg( \frac{ h_{\vec{\lambda}}(Y) u_{\vec{\lambda}}(X) e^{t E(\vec{\lambda})} }{ \prod_{i=1}^N \lambda_i q'(\lambda_i)} \Bigg) \Big( z^{-L} \prod_{i=1}^N (1- \lambda_i) \Big)^Q.
  \end{equation}
  and similarly, for $Q \geq 0$, we have
  \begin{equation}
    \label{e:Q_pos}
    \bbP_Y(X(t) = X; Q_{L-1}(t) = Q)
    = (-1)^{(N+1)Q} \oint_{C_{R'}} \frac{\dd z}{z}
    \sum_{\SumLambdas}
    \Bigg( \frac{ h_{\vec{\lambda}}(Y) u_{\vec{\lambda}}(X) e^{t E(\vec{\lambda})} }{ \prod_{i=1}^N \lambda_i q'(\lambda_i)} \Bigg) \Big( z^{-L} \prod_{i=1}^N (1- \lambda_i) \Big)^Q.
  \end{equation}
  Note the difference between the two cases: the contours of integration are different and the case for negative current has a restricted summation.
  
  We will now deform the contour, from $C_{R'}$ to $C_{\epsilon'}$, in the formula \eqref{e:Q_pos}.
  We do not obtain a residue term from the deformation due to \Cref{l:well_def_hol}.
  More precisely, the only possible pole in the deformation is due to the roots of the terms $\lambda_i q'(\lambda_i) = (L \lambda_i - (L-N))/(\lambda_i -1)$ in the denominator.
  It turns out that the integrand is holomorphic at that point since the zeros of $\lambda_i q'(\lambda)$ correspond to double roots of $q_z(w)$ and the poles cancel pairwise in the summation over $\lambda_i \in \calQ(z)$ for $1 \leq i \leq N$; see \Cref{l:well_def_hol} for more details. Thus, we find
  \begin{equation}
    \label{e:Q_pos2}
    \bbP_Y(X(t) = X; Q_{L-1}(t) = Q)
    = (-1)^{(N+1)Q}  \oint_{C_{\epsilon'}} \frac{\dd z}{z}
    \sum_{\SumLambdas}
    \Bigg( \frac{ h_{\vec{\lambda}}(Y) u_{\vec{\lambda}}(X) e^{t E(\vec{\lambda})} }{ \prod_{i=1}^N \lambda_i q'(\lambda_i)} \Bigg) \Big( z^{-L} \prod_{i=1}^N (1- \lambda_i) \Big)^Q.
  \end{equation}

  Now, consider the term in the summation of \eqref{e:Q_pos2} with all $\lambda_i \in \calQ_1(z)$.
  We have that
  \begin{equation}
    \left|z^{-Q L}\prod_{i=1}^N\frac{(1 - \lambda_i^{-1})^Q}{\lambda_i q'(\lambda_i)} \right| = \mathcal{O}((\epsilon')^L )
  \end{equation}
  when $\lambda_i \in \calQ_1(z)$ for $1 \leq i \leq N$.
  Also, we have that the rest of the terms in the integrand are uniformly bounded as $\epsilon' \rightarrow 0$ when $\lambda_i \in \calQ_1(z)$ for $1 \leq i \leq N$.
  It follows that the integral of the term (in the summation of \eqref{e:Q_pos2} with $\lambda_i \in \calQ_1(z)$ for all $1 \leq i \leq N$) is identically equal to zero by taking $\epsilon'$ arbitrarily small.
  Then, for $Q \geq 0$, we have
  \begin{equation}
    \label{e:Q_pos3}
    \bbP_Y(X(t) = X; Q_{L-1}(t) = Q)
    = (-1)^{(N+1)Q}  \oint_{C_{\epsilon'}} \frac{\dd z}{z}
    \sum_{\SumLambdas}'
    \Bigg( \frac{ h_{\vec{\lambda}}(Y) u_{\vec{\lambda}}(X) e^{t E(\vec{\lambda})} }{ \prod_{i=1}^N \lambda_i q'(\lambda_i)} \Bigg) \Big( z^{-L} \prod_{i=1}^N (1- \lambda_i) \Big)^Q,
  \end{equation}
  which is exactly the same formula as in \eqref{e:Q_neg}.
  Therefore, for all $Q \in \bbZ$, we have
  \begin{equation}
    \label{e:Q}
    \bbP_Y(X(t) = X; Q_{L-1}(t) = Q)
    = (-1)^{(N+1)Q}  \oint_{C_{\epsilon'}} \frac{\dd z}{z}
    \sum_{\SumLambdas}'
    \Bigg( \frac{ h_{\vec{\lambda}}(Y) u_{\vec{\lambda}}(X) e^{t E(\vec{\lambda})} }{ \prod_{i=1}^N \lambda_i q'(\lambda_i)} \Bigg) \Big( z^{-L} \prod_{i=1}^N (1- \lambda_i) \Big)^Q.
  \end{equation}
  
  Next, we obtain the probability function for the current by summing over all possible configurations $X \in \mathcal{X}_{L,N}$. This summation is performed in \Cref{l:sum_config_space}. The result is
  \begin{equation}\label{e:Qneg4}
    \begin{split}
      & \bbP_Y(Q_{L-1}(t) = Q)= \sum_{X\in\mathcal{X}_{L,N}}\bbP_Y(X(t)= X; Q_{L-1}(t) = Q) \\
      & = \oint_{\epsilon'} \frac{\dd z}{z} \sum_{\substack{\lambda_i \in \calQ(z) \\ 1 \leq i\leq N}}' \det\left[\frac{(1 -\lambda_j^{-1})^{j-i-1}\lambda_j^{y_j+1}e^{tE(\lambda_j)}}{L \lambda_j -(L-N)} \right]_{i,j=1}^N \\
      & \qquad \qquad \times \left((-1)^{(N+1)Q}z^{-QL}\prod_{i=1}^N(1- \lambda_i^{-1})^{Q+1}- (-1)^{(N+1)(Q+1)} z^{-(Q+1)L}\prod_{i=1}^N(1- \lambda_i^{-1})^{Q+2}\right)
    \end{split}
  \end{equation}
  for any $Q\in \mathbb{Z}$. 
  
  We now aim to take a telescoping sum to obtain the cumulative distribution function. Note that
  \begin{equation}
    \left|z^{-L} \prod_{i=1}^N (1- \lambda_j^{-1}) \right|
    \sim (\epsilon')^{-L - n_0 / (1-\rho) + n_1 d }
    \gg 1,
  \end{equation}
  for $n_1 \neq N$, where $n_i$ are the number of Bethe roots in the set $\calQ_i(z)$.
  Recall that the case $n_1 = N$ is excluded by the restricted sum.
  It follows that telescopic sum for the cumulative distribution function of the current is convergent,
  \begin{equation}\label{e:Q6}
    \begin{split}
      & \bbP_Y(Q_{L-1}(t) \leq Q-1)= \sum_{k \leq Q-1} \bbP_Y( Q_{L-1}(t) = k) \\
      & = - (-1)^{(N+1) Q} \oint_{C_{\epsilon'}}
      \frac{\dd z}{z} \sum_{\substack{\lambda_i \in \calQ(z) \\ 1 \leq i \leq N}}'
      \det \left[ \frac{(1 -\lambda_j^{-1})^{j-i-1} \lambda_j^{y_j+1} e^{tE(\lambda_j)}}{L \lambda_j -(L-N)} \right]_{i,j=1}^N z^{-QL} \prod_{i=1}^N (1- \lambda_i^{-1})^{Q+1} \\
      & = -(-1)^{(N+1) Q} \oint_{C_{\epsilon'}} \frac{\dd z}{z^{1 + Q L}} \sum_{\substack{\lambda_i \in \calQ(z) \\ 1 \leq i \leq N}}'
      \det \left[ \frac{ (1 -\lambda_j^{-1})^{Q+j-i} \lambda_j^{y_j+1} e^{t E(\lambda_j)}}{L \lambda_j -(L-N)} \right]_{i,j=1}^N.
    \end{split}
  \end{equation}
  Lastly, we may remove the restriction on the summation by applying \Cref{l:stat_current}. Then, we have
  \begin{equation}\label{e:Q7}
    \begin{split}
      & \bbP_Y (Q_{L-1}(t) \leq Q-1) \\
      & = 1 -(-1)^{(N+1) Q}\oint_{\epsilon'} \frac{\dd z}{z^{1 +(Q+1)L }}
      \sum_{\substack{\lambda_i \in \calQ(z) \\ 1 \leq i \leq N}}
      \det \left[ \frac{(1 -\lambda_j^{-1})^{Q+j-i} \lambda_j^{y_j+1} e^{tE(\lambda_j)}}{L \lambda_j -(L-N)} \right]_{i,j=1}^N.
    \end{split}
  \end{equation}
  Moreover, by the complement formula, we have
  \begin{equation}
    \bbP_Y (Q_{L-1}(t) \geq Q) = (-1)^{(N+1)Q}\oint_{\epsilon'} \frac{\dd z}{z^{1 +Q L }} \sum_{\substack{\lambda_i \in \calQ(z)\\1\leq i \leq N}} \det\left[\frac{(1 -\lambda_j^{-1})^{Q+j-i}\lambda_j^{y_j+1}e^{tE(\lambda_j)}}{L \lambda_j -(L-N)} \right]_{i,j=1}^N.
  \end{equation}
  This establishes the first formula of \Cref{p:current_cdf}.
  The second formula of \Cref{p:current_cdf} follows by a straightforward application of the the deformed Bethe equations \eqref{e:dBE}.

\end{proof}

\begin{lemma}
  \label{l:sum_config_space}
  Take $\lambda_i \in \calR(z)$ for $i =1, \cdots, N$ and $z \in \bbC^\star$.
  Then, we have
  \begin{equation}
    \sum_{X \in \calX_{L, N}} \det[(1 - \lambda_j^{-1})^{-i} \lambda_j^{-x_i}]
    = \left( 1 + (-1)^N z^{-L} \prod_{i=1}^N (1 - \lambda_i^{-1}) \right) \det[ (1 - \lambda_j^{-1})^{-i-1} ].
  \end{equation}
\end{lemma}

\begin{rem}
  Note that this identity provides an alternative proof that the formula \eqref{e:transition_prob_1} given in \Cref{t:transition_prob} is indeed normalized to one, i.e.~the sum of the probability distribution over the state space is equal to one. More precisely, note that the summation over the state space $\calX_{L, N}$ cancels out the coupling term in the denominator of \eqref{e:transition_prob_1}. Then, the resulting contour integral in \eqref{e:transition_prob_1} after the summation equals zero since the coupling terms has cancelled out and we may deform the large contour to the small contour due to \Cref{l:well_def_hol}. Thus, the only contribution in the summation over the state space is due to $u_0(1) = {L \choose N}^{-1}$, which is equal to one as desired.
  
  The proof of this result is similar to \cite[Lemma 6.1]{baikliu2018}.
  We include the proof for self-contained argument.
\end{rem}

\begin{proof}
  We perform the summation using the linearity property of matrices.
  Then, we simplify the result of the summation by using the deformed Bethe equations \eqref{e:dBE1} and \eqref{e:dBE2} since $w_i \in \calR(z)$.

  Denote the entries of the matrix by
  \begin{equation}
    M^{(0)}(i,j) = (1 - \lambda_j^{-1})^{-i} \lambda_j^{-x_i}.
  \end{equation}
  We write the summation in the following order
  \begin{equation} \label{eqn:sum_dets0}
    \sum_{X \in \calX_{L, N}} \det[ M^{(0)}(i,j) ]_{i,j=1}^N
    = \sum_{x_N=0}^{L-1} \sum_{x_{N-1} = 0}^{x_N-1} \cdots \sum_{x_1 = 0}^{x_2-1} \det[ M^{(0)}(i,j) ]_{i,j=1}^N.
  \end{equation}
  Note that only the first row of the matrix $M^{(0)}$ depends on the variable $x_1$.
  Then, by the linearity of the determinant, the sum of the determinants, with respect to $x_1$, is equal to the determinant of the resulting matrix by adding all of the first rows.
  The sum of the first rows in the $j$-th column is
  \begin{equation}
    \sum_{x_1 = 0}^{x_2-1} (1 - \lambda_j^{-1})^{-1} \lambda_j^{-x_1} = (1- \lambda_j^{-1})^{-2} (1 - \lambda_j^{-x_2}).
  \end{equation}
  The resulting first row is composed of two terms, the first one depending only on the column number $j$ and the second one being the term in the second row with a negative sign.
  Adding the second row row to the resulting first row, one gets,
  \begin{equation}
    \sum_{x_1 = 0}^{x_2-1} \det[M^{(0)}(i,j)]_{i,j=1}^N = \det[M^{(1)}(i,j)]_{i,j=1}^N.
  \end{equation}    
  where we introduce the matrix
  \begin{equation}
    M^{(1)}(i,j) = \begin{cases}
      (1- \lambda_j^{-1})^{-2}, & i = 1, \\
      M^{(0)}(i,j), & i \geq 2.
    \end{cases}
  \end{equation}
  We repeat this procedure inductively for the rest of the summations.
  The corresponding summation of the $k$-rows, for $1 \leq k \leq N-1$, is
  \begin{equation}
    \sum_{x_{k} = 0}^{x_{k+1}-1} (1 - \lambda_j^{-1})^{-k} \lambda_j^{-x_{k}}
    = (1- \lambda_j^{-1})^{-k-1} (1- \lambda_j^{-x_{k+1}}),
  \end{equation}
  and we introduce the matrices
  \begin{equation}
    M^{(k)}(i,j) = \begin{cases}
      (1- \lambda_j^{-1})^{-i-1}, & i \leq k, \\
      M^{(0)}(i,j), & i \geq k+1.
    \end{cases}
  \end{equation}
  Then, we have
  \begin{equation}
    \sum_{x_{k} = 0}^{x_{k+1} - 1} \det[M^{(k-1)}(i,j)]_{i,j=1}^N = \det[M^{(k)}(i,j)]_{i,j=1}^N,
  \end{equation}
  for $1 \leq k \leq N-1$.
  In particular, we have
  \begin{equation}
    \sum_{X \in \calX_{L, N}} \det[M^{(0)}(i,j)]_{i,j=1}^N = \sum_{x_N = 0}^{L-1} \det[M^{(N-1)}(i,j)]_{i,j=1}^N.
  \end{equation}
  The sum of the $N$-rows gives 
  \begin{align*}
    \sum_{x_{N} = 0}^{L-1} (1 - \lambda_j^{-1})^{-N} \lambda_j^{-x_{N}}
    & = (1- \lambda_j^{-1})^{-N-1} (1- \lambda_j^{-L}) \\
    & = (1- \lambda_j^{-1})^{-N-1} - (1- \lambda_j^{-1})^{-1} z^{-L},
  \end{align*}
  where we use the fact that $\lambda_j \in \calR(z)$ satisfying the deformed Bethe equation \eqref{e:dBE1}.
  Then, we have
  \begin{equation}\label{e:sum_det}
    \sum_{X \in \calX_{L, N}} \det[M^{(0)}(i,j)]_{i,j=1}^N = \det M^{(N)} - \det \widetilde{M}^{(N)},
  \end{equation}
  where $M^{(N)}(i,j) = (1- \lambda_j^{-1})^{-i-1}$ is a Vandermonde matrix and 
  \begin{equation}
    \widetilde{M}^{(N)} (i,j) = \begin{cases}
      (1 - \lambda_j^{-1})^{-i-1}, & i \leq N-1, \\
      (1 - \lambda_j^{-1})^{-1} z^{-L}, & i =N.
    \end{cases}
  \end{equation}
  Moreover, one can factorize $z^{-L}$ from the last row of $\widetilde{M}^{N}$ and $(1-\lambda_j)^{-1}$ from each column, then rearrange the rows of the matrix to obtain
  \begin{equation*}
    \det \widetilde{M}^{(N)} = (-1)^{N-1} z^{-L} \prod_{i=1}^N (1 - \lambda_j^{-1}) \cdot \det M^{(N)}.
  \end{equation*}
  This, along with \eqref{e:sum_det}, establish the identity in the Lemma.
\end{proof}

We conclude this subsection with a technical lemma needed for the proof of \Cref{p:current_cdf}, see ~\eqref{e:Q7}.

\begin{lemma}\label{l:stat_current}
  For any integer $Q \in \bbZ$, we have
  \begin{equation}
    \oint_{C_{\epsilon'}} \frac{\dd z}{z^{1 + QL}}
    \sum_{\substack{\lambda_i \in \calQ_1(z) \\ 1 \leq i \leq N}}
    \det \left( \frac{(1 - \lambda_j)^{j-i-1} \lambda_j^{y_j+1} e^{t E(\lambda_j)}}{L \lambda_j - (L-N)} \right)_{i, j=1}^N
    \prod_{i = 1}^N (1- \lambda_i^{-1})^{Q+1} = (-1)^{(N+1) Q}.
  \end{equation}
\end{lemma}

\begin{proof}
  We compute the integral by residue computation. Note that the only residue inside the contour is located at $z=0$.
  So, we take a series expansion of the integrand at $z=0$ to obtain the result.
  
  For the Bethe roots $\lambda \in \calQ_1(z)$, we have the following expansion
  \begin{equation}
    1- \lambda^{-1} = \eta z^{d}(1 + o(1)), 
  \end{equation}
  where $\eta$ is an $N^{\mathrm{th}}$ root of unity; see \eqref{e:q_small_sol}.
  Additionally, note that in the summation, we may assume that the Bethe roots and, also, the roots of unity, are pairwise distinct. It is due to the fact that the determinant in the integrand will be identically zero if two columns are equal to each other. Then, we have
  \begin{equation*}
    z^{-QL} \prod_{i=1}^N(1- \lambda_i^{-1})^{Q+1}
    = z^{-QL} \prod_{i=1}^N(\eta_{\sigma(i)} z^d)^{Q+1} (1 + o(1))
    = (-1)^{(N+1)(Q+1)} z^L (1 + o(1)),
  \end{equation*}
  where $\eta_j = \exp(2 \icomp \pi j/ N )$ and $\sigma \in S_N$ is a permutation.
  Now, let $\bbU_N$ be the set of $N^{\mathrm{th}}$ roots of unity.
  Then, we have the following series expansion for the sum of the determinants
  \begin{equation}
    \begin{split}
      & \sum_{\substack{\lambda_i \in \calQ_1(z) \\ 1 \leq i \leq N}}
      \det \left( \frac{(1 - \lambda_j^{-1})^{j-i-1} \lambda_j^{y_j+1} e^{t E(\lambda_j)}}{L \lambda_j - (L-N)} \right)_{i, j=1}^N
      = \sum_{\substack{\eta_i \in \bbU_N \\ 1 \leq i \leq N}}
      \det \left(\frac{( \eta_j z^d)^{j - i-1}}{N} \right)_{i, j=1}^N (1 + o(1)) \\
      & = (-1)^{N+1} z^{-L} \det \left( \sum_{\eta \in \bbU_N} \eta^{j - i}/N \right)_{i, j=1}^N(1 + o(1))
      = (-1)^{N+1} z^{-L} \det(\mathds{1}(i=j))_{i, j=1}^N(1 + o(1))\\
      & = (-1)^{N+1} z^{-L} (1 + o(1)).
    \end{split}
  \end{equation}
  Thus, we obtain the result by taking the product of the series expansions above,
  \begin{align*}
    & \oint_{C_{\epsilon'}} \frac{\dd z}{z^{1 + QL}}
      \sum_{\substack{\lambda_i \in \calQ_1(z) \\ 1 \leq i \leq N}}
    \det \left( \frac{(1 - \lambda_j)^{j-i-1} \lambda_j^{y_j+1} e^{t E(\lambda_j)}}{L \lambda_j - (L-N)} \right)_{i, j=1}^N \prod_{i = 1}^N (1- \lambda_i^{-1})^{Q+1} \\
    & = (-1)^{(N+1)Q} \oint_{C_{\epsilon'}} \frac{\dd z}{z} (1 + o(1))
      = (-1)^{(N+1)Q}.
  \end{align*}
\end{proof}

\subsection{Flat initial condition}

In the next two sections, we turn to some specific classes of initial conditions: flat and step initial conditions.
In this section, we consider the flat case and, in the next section, we consider the step initial conditions. We give an expression for cumulative distribution function of the current as a contour integral of a Fredholm determinant; see \Cref{p:flat_current}.
We then use this result to give a proof of \Cref{t:asymptotic_flat_current}.
We begin by introducing some necessary notation. 

Fix an integer $d \geq 2$.
Assume that the length of the ring $L$ and the number of particles $N$ satisfy the following identity $L = d N $.
Recall that some of the results are slightly different in the case $2N = L$ but still true; see \Cref{r:2nl}.
Then, the flat initial conditions are given as follows:
\begin{equation}
  \label{e:flat_IC}
  (y_1, \dots, y_N) = (\delta, d + \delta, \dots, (N-1)d + \delta) \in \calX_{L, N} \quad \text{for some} \, 0 \leq \delta < d.
\end{equation}

Recall the sets of deformed Bethe roots $\calQ_0(z)$ and $\calQ_1(z)$, given by \eqref{e:q_small_sol} for $z \in \bbC$ satisfying $0 < |z| < r_0$. The set $\calQ_1(z)$ contains $N$ points and the set $\calQ_0(z)$ contains $L-N = (d-1)N$ points of the solutions to
\[
  q_z(w) = 1 - z^L w^{-L} (1-w^{-1})^{-N} = 1 - (z^d w^{-d} (1-w^{-1}))^N=0.
\]
Let us introduce the following sets of roots,
\[
  D_k(z) := \{ w \in \bbC \mid w^d (1-w^{-1}) = z^d e^{2 \pi \icomp k/N} \}, \qquad 0 \leq k \leq N-1.
\]
Each set $D_k(z)$ contains $d$ roots of the function $q_z(w)$ so that exactly 1 root is in the set $\calQ_1(z)$ and $N-1$ roots are in the set $\calQ_0(z)$.
For $u \in D_k(z) \cap \calQ_0(z)$, write $V(u) := D_k(z) \cap \calQ_1(z)$ which only contains only one element.
For $v \in D_k(z) \cap \calQ_1(z)$, write $U(v) := D_k(z) \cap \calQ_0(z)$ which contains $N-1$ elements.
See \Cref{fig:Bethe_UV} for an illustration.

\begin{figure}[htb!]
  \centering
  \includegraphics[scale=0.8]{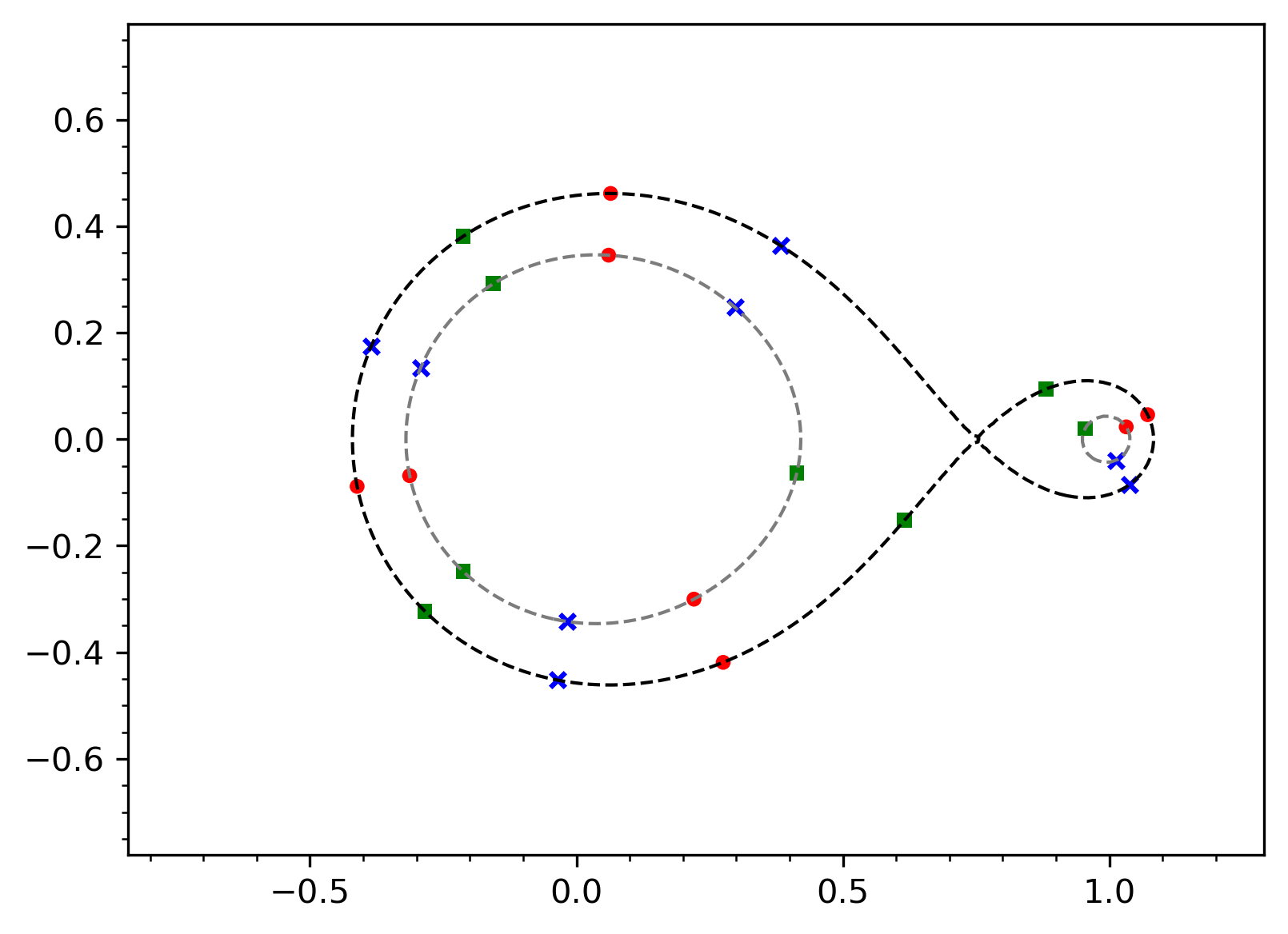}
  \caption{The two dotted curves correspond to $|w^{L-N}(w-1)^N| = |z|^L$ with $|z| = r_0 := \rho^{\rho} (1-\rho)^{1-\rho}$ (black color, outer curve) and $|z| = 0.8 r_0$ (grey color, inner curve with two components).
    Note that the self-intersection point of the black curve corresponds to the critical solution $w = 1-\rho$.
    The roots on the right (resp. left) of the critical solution are put in $\calQ_1(z)$ (resp. $\calQ_0(z)$).
    Additionally, the red (circle), green (squre) and blue (cross) dots on the curves correspond to the root sets $D_0(z)$, $D_1(z)$ and $D_2(z)$ with $z = r_0 e^{i\pi/6}$ and $z = 0.8 r_0 e^{i\pi/6}$.
    Therefore, for any given root $u$ on the left (resp. $v$ on the right), the set $V(u)$ (resp. $U(v)$) is given by the the unique element on the right (resp. three elements on the left), represented by the dot of the same color (shape) on the same dotted curve.
  }
  \label{fig:Bethe_UV}
\end{figure}

Additionally, we define the following factors of the function $q_z(w)$
\begin{equation}
  \label{e:qz01}
  q_{z, 0}(w) = \prod_{v \in \calQ_0(z)} (w-v)
  \qquad \text{and} \qquad
  q_{z, 1}(w) = \prod_{v \in \calQ_1(z)} (w-v).
\end{equation}
In particular, we note the following identities
\begin{equation}
  \label{e:product_qz01}
  q_{z, 0}(w) q_{z, 1}(w) = w^L (1-w^{-1})^N q_z(w) = w^{L-N} (w-1)^N - z^L
\end{equation}
and
\begin{equation}
  \label{e:derivative_qz01}
  q_{z, 1}'(v) q_{z, 0}(v) = L v^{L-N-1} (v-1)^{N-1} [v - (1-\rho)], \qquad v \in \calQ_1(z).
\end{equation}
These identities will appear later in the proof of \Cref{p:flat_current}.

For $z \in \bbC$ satisfying $0 < |z| < r_0$, define the function
\begin{equation}
  \label{e:C_N_flat}
  \calC_N^\icflat(z)
  := \frac{ \prod_{u \in \calQ_0(z)} (1-u)^{-Q} \prod_{v \in \calQ_1(z)} v^{L-N-Q-\tfrac{d}{2}+1}  \sqrt{d^{-1} (v - (1-\rho))^{-1} } e^{t E(v)} }{ \prod_{v \in \calQ_1(z)} \prod_{u \in \calQ_0(z)} \sqrt{v-u} },
\end{equation}
where $\sqrt{\cdot}$ is the square root function with branch cut $\bbR_{<0}$.
We also define the operator $K_z^\icflat$ acting on $\ell^2(\calQ_0(z))$ by the kernel
\begin{equation}
  \label{e:K_flat}
  K_z^\icflat(u, u') = \frac{ f^\icflat(u) }{ (u-v') f^\icflat(v') }, \qquad u, u' \in \calQ_0(z),
\end{equation}
where $v'$ is the unique element in $V(u')$ and the function $f^\icflat : \calQ(z) \longrightarrow \bbC$ is defined by
\begin{equation}
  \label{e:f_flat}
  f^\icflat(w) := \frac{ w^{-N+2} (1-w^{-1})^{Q-N+1} e^{t E(w)} }{w - (1-\rho)}
  \cdot
  \left\{
    \begin{array}{ll}
      q_{z, 1}(w), & w \in \calQ_0(z), \\
      q_{z, 1}'(w), & w \in \calQ_1(z).
    \end{array}
  \right.
\end{equation}

\begin{proposition}
  \label{p:flat_current}
  Take the PushASEP on a ring with flat initial conditions given by \eqref{e:flat_IC} and let $Q$ be an integer. Then, we have
  \begin{equation}
    \label{e:flat_current}
    \bbP^\icflat_{L, d, \delta}(Q_{L-1}(t) \geq Q) = \oint_{C_r} \calC_N^\icflat (z) \cdot \det \big( I + K_z^\icflat \big) \frac{\dd z}{z},
  \end{equation}
  where $0 < r < r_0 := \rho^{\rho} (1-\rho)^{1-\rho}$.
\end{proposition}

\begin{proof}
  We mimic the proof from \cite[Theorem~7.1]{baikliu2018}.
  We start with \eqref{e:current_cdf1} where we make the identification \eqref{e:flat_IC} for the flat initial condition. Then, we have
  \begin{align*}
    & \bbP^\icflat_{L, d, \delta} (Q_{L-1}(t) \geq Q)
      = (-1)^{(N+1)Q} \oint_{C_{\epsilon'}} \frac{\dd z}{z^{1+QL}}
      \sum_{\substack{\lambda_i \in \calQ(z) \\ 1 \leq i \leq N}}
    \det \Bigg[ \frac{1}{L} \frac{ (1- \lambda_j^{-1})^{Q+j-i}  \lambda_j^{(j-1)d + \delta +1} e^{t E(\lambda_j)}}{ \lambda_j - (1-\rho) } \Bigg]_{i,j=1}^N \\
    & = (-1)^{(N+1)Q + {N \choose 2}} \oint_{C_{\epsilon'}} \frac{\dd z}{z^{1+QL}}
      \sum_{\substack{\lambda_i \in \calQ(z) \\ 1 \leq i \leq N}}
    \det \Bigg[ \frac{1}{L} \frac{ (1- \lambda_j^{-1})^{Q+j-(N-i+1)}  \lambda_j^{(j-1)d + \delta+1} e^{t E(\lambda_j)}}{ \lambda_j - (1-\rho) } \Bigg]_{i,j=1}^N \\
    & = C_1 \oint_{C_{\epsilon'}} \frac{\dd z}{z^{1+QL}}
      \det \Big[ \sum_{\lambda \in \calQ(z)} (1-\lambda^{-1})^{i-1} [\lambda^d (1-\lambda^{-1})]^{j-1} f_1(\lambda) \Big]_{i,j=1}^N
  \end{align*}
  where, in the second line, we revert the row indices $i \leftarrow N-i+1$ and, in the third line, we use the linearity of the determinant and set
  \begin{equation}
    C_1  = (-1)^{(N+1)Q + {N \choose 2}} L^{-N}, \quad f_1(w)  =  \frac{w^{1 + \delta} (1-w^{-1})^{Q-N+1} e^{t E(w)}}{w-(1-\rho)}.
  \end{equation}
  Applying the Cauchy-Binet/Andreief formula, we obtain
  \begin{align*}
    \frac{C_1}{N!} \oint_{C_{\epsilon'}} \frac{\dd z}{z^{1+QL}}
    \sum_{\substack{\lambda_i \in \calQ(z) \\ 1 \leq i \leq N}}
    \prod_{i < j} (\lambda_j^{-1} - \lambda_i^{-1}) [ \lambda_i^d (1-\lambda_i^{-1}) - \lambda_j^d (1-\lambda_j^{-1}) ]
    \prod_{i=1}^N f_1(\lambda_i) \\
    = \frac{C_1}{N!} \oint_{C_{\epsilon'}} \frac{\dd z}{z^{1+QL}}
    \sum_{\substack{\lambda_i \in \calQ(z) \\ 1 \leq i \leq N}}
    \prod_{i < j} (\lambda_i - \lambda_j) [ \lambda_i^d (1-\lambda_i^{-1}) - \lambda_j^d (1-\lambda_j^{-1}) ]
    \prod_{i=1}^N f_2(\lambda_i)
  \end{align*}
  where $f_2(w) := w^{-(N-1)} f_1(w)$ for any $w \in \bbC$.
  The above formula takes exactly the same form as \cite[(7.13)]{baikliu2018} and may be simplified using exactly the same method as in the proof of \cite[Theorem~7.1]{baikliu2018}.
  In particular, the above formula can be further simplified to
  \[
    C_1 \oint_{C_{\epsilon'}} \frac{\dd z}{z^{1+QL}}
    \prod_{j=1}^N f_2(v_j) \prod_{1 \leq i < j \leq N} (v_i - v_j) [ v_i^d (1-v_i^{-1}) - v_j^d (1-v_j^{-1}) ]
    \det(I + K_z)
  \]
  where we enumerate the elements in $\calQ_1(z)$ by $\calQ_1(z) := \{ v_1, \dots, v_N \}$ and define the kernel $K_z^{\icflat}$ to be
  \[
    K_z^{\icflat}(u, u') := \frac{f^{\icflat}(u)}{f^{\icflat}(u') (u-v')}, \quad \forall u, u' \in \calQ_0(z),
  \]
  with
  \[
    f^{\icflat}(w) := f_2(w)
    \cdot
    \left\{
      \begin{array}{ll}
        q_{z, 1}(w), & w \in \calQ_0(z), \\
        q_{z, 1}'(w), & w \in \calQ_1(z).
      \end{array}
    \right.
  \]
  Then, we consider the following computation
  \begin{align}
    \calC_N^\icflat (z) & := \frac{C_1}{z^{QL}} \prod_{j=1}^N f_2(v_j) \prod_{1 \leq i < j \leq N} (v_i - v_j) [ v_i^d (1-v_i^{-1}) - v_j^d (1-v_j^{-1}) ] \nonumber \\
                        & = \frac{(-1)^{(N+1)Q + {N \choose 2}}}{ z^{QL} L^{N} } \prod_{j=1}^N f_2(v_j)
                          \prod_{1 \leq i < j \leq N} (v_i - v_j)^2
                          \prod_{1 \leq i < j \leq N} \frac{ v_i^d (1-v_i^{-1}) - v_j^d (1-v_j^{-1}) }{ v_i - v_j } \nonumber \\
                        & = \frac{(-1)^{(N+1)Q}}{z^{QL}} \prod_{j=1}^N \frac{q_{z, 1}'(v_j) v_j^{-N+\delta+2} (1-v_j^{-1})^{Q-N+1} e^{t E(v_j)} }{ L (v_j - (1-\rho) ) }
                          \prod_{1 \leq i < j \leq N} \frac{ v_i^d (1-v_i^{-1}) - v_j^d (1-v_j^{-1}) }{ v_i - v_j }, \label{e:C_N_flat_intermediate}
  \end{align}
  where the last equality is due to the following identity,
  \[
    \prod_{j=1}^N q_{z, 1}'(v_j) = (-1)^{{N \choose 2}} \prod_{1 \leq i < j \leq N} (v_i - v_j)^2.
  \]
  We rewrite the terms in \eqref{e:C_N_flat_intermediate}. First, note that the $L$ roots of the polynomial $w^{L-N} (w-1)^N - z^L$ (with respect to $w$) are given by $\lambda \in \calQ(z) = \calQ_0(z) \bigcup \calQ_1(z)$. So, the product of these roots satisfies the relation
  \begin{equation}
    \label{e:roots_relation1}
    \prod_{u \in \calQ_0(z)} u \prod_{v \in \calQ_1(z)} v = (-1)^{L+1} z^{L}.
  \end{equation}
  Similarly, the $L$ roots of the polynomial $(\tilde{w} - 1)^L + (-1)^{L+1} z^{-L} \tilde{w}^{N}$ (with respect to $\tilde{w}$) are given by $1 - \lambda^{-1}$ with $\lambda \in \calQ(z)$, giving another relation
  \begin{equation}
    \label{e:roots_relation2}
    \prod_{u \in \calQ_0(z)} (1-u^{-1}) \prod_{v \in \calQ_1(z)} (1-v^{-1}) = 1.
  \end{equation}
  Then, we rewrite the first product in \eqref{e:C_N_flat_intermediate}, along with its prefactor, as follows
  \begin{align*}
    &\frac{(-1)^{(N+1)Q}}{z^{QL}} \prod_{j=1}^N \frac{q_{z, 1}'(v_j) v_j^{-N+\delta+2} (1-v_j^{-1})^{Q-N+1} e^{t E(v_j)} }{ L (v_j - (1-\rho) ) } = \frac{(-1)^{(N+1)Q}}{z^{QL}} \prod_{v \in \calQ_1(z)} \frac{ v^{L-N+\delta} (1-v^{-1})^{Q} e^{t E(v)} }{ q_{z, 0}(v) } \\
    &= \frac{(-1)^{(L-N)Q}}{\prod_{u \in \calQ_0(z)} u^Q \prod_{v \in \calQ_1(z)} v^Q}
      \prod_{v \in \calQ_1(z)} \frac{ v^{L-N+\delta} (1-v^{-1})^{Q} e^{t E(v)} }{ q_{z, 0}(v) }
      \prod_{u \in \calQ_0(z)} (1-u^{-1})^{-Q} \prod_{v \in \calQ_1(z)} (1-v^{-1})^{-Q} \\
    &= \prod_{v \in \calQ_1(z)} \frac{ v^{L-N-Q+\delta} e^{t E(v)} }{ q_{z, 0}(v) }
      \prod_{u \in \calQ_0(z)} (1-u)^{-Q},
  \end{align*}
  where the first equality follows from \eqref{e:derivative_qz01}, the second equality follows from the relations \eqref{e:roots_relation1} and \eqref{e:roots_relation2}, and the last equality follows from simplification of the previous expression.

  We write the second product in \eqref{e:C_N_flat_intermediate}, by \cite[(7.30)-(7.32)]{baikliu2018}, as follows
  \[
    \prod_{1 \leq i < j \leq N} \frac{ v_i^d (1-v_i^{-1}) - v_j^d (1-v_j^{-1}) }{ v_i - v_j }
    = \frac{ \prod_{v \in \calQ_1(z)} \prod_{u \in \calQ_0(z)} \sqrt{v-u} }{ \prod_{v \in \calQ_1(z)} \sqrt{d(v - (1-\rho))} \sqrt{v}^{d-2} }.
  \]
  We obtain the result and the expression of $\calC_N^\icflat$ in \eqref{e:C_N_flat} by putting the previous identities together.
\end{proof}

\begin{proof}[Proof of \Cref{t:asymptotic_flat_current}]
  The computations go as in \cite[Section~8.1]{baikliu2018} except that, in our setting, the expressions for $\calC_N^\icflat(z)$ and $K_z^\icflat$ are slightly different and the variables $u$, $v$ and $w$ correspond to our $u-1$, $v-1$ and $w-1$. In particular, set
  \begin{align*}
    t & = \frac{\tau}{ \sqrt{\rho(1-\rho)} } L^{3/2}, \\
    Q & = vt - x \sqrt{1-\rho} \tau^{1/3} N^{1/2} 
        = vt - x \rho^{2/3} (1-\rho)^{2/3} t^{1/3},
  \end{align*}
  and, additionally, replace the variable $z$ by $\bfz$ and introduce the change of variables $\bfz^L = (-1)^N r_0^L z$.
  Then, we consider the asymptotic expansion of \eqref{e:flat_current} as $N \rightarrow \infty$.
  
  Let us start with the asymptotic expansion of $\calC_N^\icflat(\bfz)$. We express $\calC_N^\icflat(\bfz)$ as the product of the three following terms
  \begin{equation}
    \begin{split}
      \calC_{N, 1}^\icflat(\bfz) & = \frac{ \prod_{u \in \calQ_0(\bfz)} (\sqrt{1-u})^N \prod_{v \in \calQ_1(\bfz)} (\sqrt{v})^{L-N} }{ \prod_{v \in \calQ_1(\bfz)} \prod_{u \in \calQ_0(\bfz)} \sqrt{v-u} }, \\
      \calC_{N, 2}^\icflat(\bfz) & = \frac{1}{ \prod_{v \in \calQ_1(\bfz)} \sqrt{d(v-(1-\rho))} (\sqrt{v})^{d-2} }, \\
      \calC_{N, 3}^\icflat(\bfz) & = \prod_{u \in \calQ_0(\bfz)} (\sqrt{1-u})^{-2Q-N} \prod_{v \in \calQ_1(\bfz)} (\sqrt{v})^{L-N-2Q+2\delta} e^{t E(v)}.
    \end{split}
  \end{equation}
  Using the results summarized in \cite[Section~8.1.1]{baikliu2018}, we find that
  \begin{align*}
    \calC_{N, 1}^\icflat(\bfz) & = e^{B(z)} (1 + O(N^{\epsilon-1/2})) \\
    \calC_{N, 2}^\icflat(\bfz) & = e^{A_3(z)} (1 + O(N^{\epsilon-1/2}))
  \end{align*}
  where $B(z)$ and $A_3(z)$ are defined as in \eqref{e:A_i(z)} and \eqref{e:B(z)}.
  For $\calC_{N, 3}^\icflat(\bfz)$, we follow the proof given in \cite[Section~9.3]{baikliu2018} and find
  \begin{align}
    \ln \calC_{N, 3}^\icflat(\bfz)
    & = - \big( Q + \frac{N}{2} \big) \sum_{u \in \calQ_0(z)} \ln(1-u)
      + \sum_{v \in \calQ_1(z)} \big( \big( \frac{L-N}{2} -Q +\delta \big) \ln v + t E(v) \big) \nonumber \\
    & = L \bfz^L \int_{1-\rho - \icomp \infty}^{1-\rho + \icomp \infty} (G^\icflat(w) - G^\icflat(1-\rho)) \frac{w - (1-\rho)}{w(w-1) \tilde{q}_\bfz(w)} \frac{\dd w}{2 \pi \icomp}
      \label{e:Gflat_integral}
  \end{align}
  with $\tilde{q}_\bfz(w) = w^L (1-w^{-1})^N - \bfz^L = w^L (1-w^{-1})^N q_\bfz(w)$ and
  \begin{align*}
    G^\icflat(w) = - \big( Q + \frac{N}{2} \big) \ln(1-w) - \big( \frac{L-N}{2} -Q +\delta \big) \ln w - t E(w).
  \end{align*}

  We make the change of variables $w = 1-\rho + \rho \sqrt{1-\rho} \xi N^{-1/2}$ and split the above integral into two parts $|\xi| \leq N^{\epsilon/4}$ and $|\xi| > N^{\epsilon/4}$ for some fixed $\epsilon \in (0, \tfrac12)$.
  The second part gives the estimate $\calO(e^{-C N^{\epsilon/2}})$ while the first part requires more care.
  Taylor expansion (using the computer software \texttt{Sage}) shows that for $w = 1-\rho + \alpha$ with $\alpha = \rho \sqrt{1-\rho} \xi N^{-1/2} = \calO(N^{\epsilon/4-1/2})$,
  \begin{align}
    &G^\icflat(w) - G^\icflat(1-\rho)= \nonumber \\
    &\quad   C_1 \tau N^{1/2} \xi^2 
      + \bigg( - \tau^{1/3} x \xi + \frac14 \xi^2 + C_2 \tau \xi^3 \bigg) + E_1 N^{-1/2} \xi^2 + E_2 N^{-1/2} \xi^3 + E_3 N^{-1/2} \xi^4 + O(N^{5\epsilon/4-1}) \label{e:G_flat_Taylor}
  \end{align}
  where 
  \begin{align*}
    C_1 & = \frac{1}{2} \frac{1}{\rho \sqrt{1-\rho}} \Big[ (1-2\rho) p - \frac{q}{(1-\rho)^2} \Big] = \frac{\vshock}{2 \rho \sqrt{1-\rho}}, \\
    C_2 & = \frac{1}{3 \rho (1-\rho)} \Big[ (3 \rho^2 - 3 \rho + 1) p + \frac{3\rho -1}{(1-\rho)^2} q \Big],
  \end{align*}
  and the constants $E_i$ depend only on $\rho$, $\tau$, $p$ and $q$.
  Later, we will see that $E_2=\frac{ 1-2\rho }{ 6 \sqrt{1-\rho} }$ is the only non-zero contribution after the integration in \eqref{e:Gflat_integral}.

  Rewriting \cite[(9.36) and (9.37)]{baikliu2018}, for the same $w = 1-\rho + \alpha$ as above, we find
  \begin{align}
    \frac{\tilde{q}_\bfz(w)}{\bfz^L}
    & = \frac{e^{-\xi^2/2} - z}{z}
      \bigg( 1 + \frac{2\rho - 1}{3 \sqrt{1-\rho}} \frac{e^{-\xi^2/2}}{e^{-\xi^2/2} - z} \xi^3 N^{-1/2} + E_4 \xi^4 N^{-1} + \calO(N^{5\epsilon/4 - 3/2})  \bigg), \label{e:BL_CI_term1} \\
    \frac{L (w-(1-\rho))}{w(w-1)}
    & = - \frac{ \xi N^{1/2} }{\rho \sqrt{1-\rho}}
      \bigg( 1 + \frac{1-2\rho}{\sqrt{1-\rho}} \xi N^{-1/2} + E_5 \xi^2 N^{-1} + \calO(N^{3\epsilon/4 - 3/2}) \bigg). \label{e:BL_CI_term2}
  \end{align}
  Thus,
  \begin{equation}
    \frac{ L(w-(1-\rho)) \bfz^L }{ w(w-1) \tilde{q}_\bfz(w) }  \frac{\dd w}{\dd \xi}
    = - \frac{z}{e^{-\xi^2/2} - z} \bigg[ \xi + \bigg( \frac{ 1-2\rho }{ \sqrt{1-\rho} } \xi^2
    + \frac{ (1-2\rho) }{ 3 \sqrt{1-\rho} } \frac{ e^{-\xi^2/2} }{ e^{-\xi^2/2} - z } \xi^4 \bigg) N^{-1/2} \bigg]
    + \calO(N^{-1})
  \end{equation}
  Now, we can put the above three expansions together and evaluate the integral in \eqref{e:Gflat_integral} in the domain $|\xi| \leq N^{\epsilon/4}$.
  Using the fact that the integral of an odd function (in $\xi$) is zero on our integration domain, we may simplify
  \begin{align*}
    & \int_{1-\rho-\icomp \rho \sqrt{1-\rho} N^{\epsilon/4}}^{1-\rho+\icomp \rho \sqrt{1-\rho} N^{\epsilon/4}}
      ( G^\icflat(w) - G^\icflat(1-\rho) ) \frac{ L(w-(1-\rho)) \bfz^L }{ w(w-1) \tilde{q}_\bfz(w) } \frac{\dd w}{2\pi \icomp}= \\
    & \int_{-\icomp N^{\epsilon/4}}^{\icomp N^{\epsilon/4}}
      \frac{z}{e^{-\xi^2/2} - z} \bigg( \tau^{1/3} x \xi^2 -  C_2 \tau \xi^4 - \frac{1-2\rho}{\sqrt{1-\rho}} C_1 \tau \xi^4 \bigg) \frac{\dd \xi}{2 \pi \icomp} - \frac{1-2\rho}{3\sqrt{1-\rho}} C_1 \tau
      \int_{-\icomp N^{\epsilon/4}}^{\icomp N^{\epsilon/4}}
      \frac{z e^{-\xi^2/2}}{(e^{-\xi^2/2} - z)^2} \xi^6 \frac{\dd \xi}{2 \pi \icomp}\\
    & + \calO(N^{2 \epsilon -1}).
  \end{align*}
  In the limit $N \to \infty$, we find that integral on the right side is equal to 
  \begin{equation}
    \tau^{1/3} x A_1(z) + \bigg( 3 C_2 + \frac{3(1-2\rho)}{\sqrt{1-\rho}} C_1 - \frac{5(1-2\rho)}{\sqrt{1-\rho}} C_1 \bigg) \tau A_2(z)
    = \tau^{1/3} x A_1(z) + \bigg( p + \frac{q}{(1-\rho)^3} \bigg) \tau A_2(z)
  \end{equation}
  due to the following identities
  \begin{align*}
    I_2 & = \int_{\myre \xi = 0} \xi^2 \frac{z}{e^{-\xi^2/2} -z} \frac{\dd \xi}{2 \pi \icomp}
          = -\frac{1}{\sqrt{2\pi}} \Li_{3/2}(z) = A_1(z), \\
    I_4 & = \int_{\myre \xi = 0} \xi^4 \frac{z}{e^{-\xi^2/2} -z} \frac{\dd \xi}{2 \pi \icomp}
          = \frac{3}{\sqrt{2\pi}} \Li_{5/2}(z) = -3 A_2(z), \\
    J_6 & = \int_{\myre \xi = 0} \xi^6 \frac{z e^{-\xi^2/2}}{ (e^{-\xi^2/2} -z)^2 } \frac{\dd \xi}{2 \pi \icomp}
          = -5 I_4 = 15 A_2(z).
  \end{align*}

  We now give an alternate expression of the Fredholm determinant $\det(I + K_\bfz^\icflat)$ that is more suitable for asymptotic expansion. Recall the expression of the kernel $K_\bfz^\icflat$ from \eqref{e:K_flat} and \eqref{e:f_flat},
  \begin{align*}
    K_z^\icflat(u, u') & = \frac{ f^\icflat(u) }{ (u-v') f^\icflat(v') }, \qquad u, u' \in \calQ_0(z), \\
    f^\icflat(w) & := \frac{ w^{-Q+1} (w-1)^{Q-N+1} e^{t E(w)} }{w - (1-\rho)}
                   \cdot
                   \left\{
                   \begin{array}{ll}
                     q_{z, 1}(w), & w \in \calQ_0(z), \\
                     q_{z, 1}'(w), & w \in \calQ_1(z).
                   \end{array}
                                     \right.,
  \end{align*}
  where $v' \in \calQ_1(z)$ is the unique element satisfying
  \begin{equation}
    \label{e:u'_v'}
    (u')^d (1 - (u')^{-1}) = (v')^d (1 - (v')^{-1}).
  \end{equation}
  Additionally, we introduce a conjugated kernel $\tilde{K}_\bfz^\icflat$ defined as follows
  \begin{equation*}
    \tilde{K}_z^\icflat(u, u') = \frac{ h^\icflat(u) }{ (u-v') h^\icflat(v') }, \qquad u, u' \in \calQ_0(z),
  \end{equation*}
  with the function $h^\icflat$ is given by
  \begin{align*}
    h^\icflat(w) & := \frac{ g^\icflat(w) }{(w - (1-\rho)) (w-1)^N}
                   \cdot
                   \left\{
                   \begin{array}{ll}
                     q_{z, 1}(w), & w \in \calQ_0(z), \\
                     q_{z, 1}'(w), & w \in \calQ_1(z),
                   \end{array}
                                     \right. \\
    g^\icflat(w) & = \frac{ \tilde{g}^\icflat(w) [ w^{d-1} (w-1) ]^{k} }{ \tilde{g}^\icflat(1-\rho) [ (1-\rho)^{d-1} (-\rho) ]^{k} }\\
    \tilde{g}^\icflat(w) &= w^{-Q+1} (w-1)^{Q+1} e^{t E(w)}
  \end{align*}
  and the exponent $k$ is chosen to be
  \begin{equation*}
    k = - \Bigg\lfloor \sqrt{ \frac{\rho}{1-\rho} } \Big( (1 - 2\rho) p - \frac{q}{(1-\rho)^2} \Big) \tau L^{3/2} \Bigg\rfloor
    = - \rho \vshock t + O(1).
  \end{equation*}
  This gives us
  \[
    \ln g^\icflat(w) = \tau^{1/3} x \xi - \frac{\tau}{3} \bigg( p + \frac{q}{(1-\rho)^3} \bigg) \xi^3 + \calO(N^{-1/2}).
  \]
  Then, by using \eqref{e:u'_v'}, one may show that the kernel $\tilde{K}_\bfz^\icflat$ is equal to the kernel $K_\bfz^\icflat$ after conjugation, 
  \begin{equation*}
    K_\bfz^\icflat(u, u') = \left(\frac{(u')^d(1- (u')^{-1})}{u^d(1- u^{-1})} \right)^k \tilde{K}_\bfz^\icflat(u, u').
  \end{equation*}
  Moreover, it follows that
  \begin{equation}
    \label{e:Kz_flat_conjugation}
    \det(I + K_\bfz^\icflat) = \det(I + \tilde{K}_\bfz^\icflat).
  \end{equation}
  Below we obtain the asymptotic expansion by considering the right side of the identity above.

  We now consider the asymptotic expansion of the Fredholm determinant. In particular, after careful analysis similar to \cite[Section~8.1.2]{baikliu2018}, we may show that
  \begin{equation*}
    \lim_{N \to \infty} \det(I + \tilde{K}_\bfz^\icflat) = \det(I - \calK_z^\icflat).
  \end{equation*}
  The operator $\calK_z^\icflat$ is defined on $\calS_{-}(z)$ with kernel
  \begin{equation*}
    \calK_z^\icflat(\xi, \eta) = \frac{ \exp \big( \frh_1(\xi, z) + \frh_1(\eta, z) - \frac13 r \tau \xi^3 + \tau^{1/3} x \xi - \frac13 r \tau \eta^3 + \tau^{1/3} x \eta \big) }{ \xi (\xi + \eta) }, 
  \end{equation*}
  Recall the definition of $r = r(p, q) := p + \frac{q}{(1-\rho)^3}$
  and define $\frh_1$ by
  \begin{equation*}
    \frh_1(\xi, z) := -\frac{1}{\sqrt{2 \pi}} \int_{-\infty}^\xi \Li_{1/2}(z e^{(\xi^2 - y^2) /2}) \dd y, \quad \myre(\xi) \leq 0,
  \end{equation*}
  where the integral contour is taken to be $(-\infty, \myre(\xi)] \cup [\myre(\xi), \xi]$ so that the integral is well defined.
  Note that we may simplify the expression given above, for $\xi \in \calS_{-}(z)$, as follows 
  \begin{equation*}
    \frh_1(\xi, z) := -\frac{1}{\sqrt{2 \pi}} \int_{-\infty}^\xi \Li_{1/2}(e^{- y^2/2}) \dd y, \quad \xi \in \calS_{-}(z),
  \end{equation*}
  since $z = e^{-\xi^2/2}$. Thus, by putting all the asymptotic expansions together, we obtain the cumulative probability distribution function $F_1(\tau^{1/3}x; r\tau)$.
\end{proof}

\subsection{Step initial condition}

The statistics of the local current depends on the observed location for the step initial conditions.
This is in contrast to the flat initial conditions, where the initial conditions are translation invariant in the relaxation time scale.
Note the current $Q_{L-1}(t)$ on the edge $(L-1)$ has a well suited formula for asymptotic analysis, see \Cref{p:current_cdf}.
Whereas, the current for any other edge does not apper to be suitable for asymptotic analysis.
Thus, to study the asymptotic behavior of the current on different relative locations, we simply consider shifted step initial conditions.
More precisely, let us consider the PushASEP model with the step initial condition defined by
\begin{equation}
  \label{e:step_IC1}
  (y_1, \dots, y_N) = (0, 1, \dots, N-1) + (m, \dots, m) \in \calX_{L, N},
\end{equation}
for an integer $0 \leq m \leq L-N$, or,
\begin{equation}
  \label{e:step_IC2}
  (y_1, \dots, y_N) = (0, 1, \dots, k-1, L-N+k, \dots, L-1)\in \calX_{L, N},
\end{equation}
for and integer $0 \leq k \leq N$.
When we consider the second initial condition above, we define additionally $m = (N-k)(d-1)$.
Note that the difference between the two cases above is that, in the second case \eqref{e:step_IC2}, the position of the current considered is in the middle of the step initial configuration.
This leads to slightly different computations that we will see below.
Our goal is to rewrite and simplify the distribution function of the current observable given in \eqref{p:current_cdf} in this specific setting.

For $z \in \bbC$ satisfying $0 < |z| < r_0$, we define the function $\calC_{i, N}^\icstep(z)$, for $i=1,2$, depending on whether the initial step condition is in the case \eqref{e:step_IC1} or \eqref{e:step_IC1}.
In the case \eqref{e:step_IC1}, we define
\begin{equation}
  \label{e:C_N_step1}
  \calC_{1,N}^\icstep(z)
  := \frac{ \prod_{v \in \calQ_1(z)} v^{L-N+m-Q} e^{t E(v)} \prod_{u \in \calQ_0(z)} (1-u)^{-Q} }{ \prod_{v \in \calQ_1(z)} \prod_{u \in \calQ_0(z)} (v-u) }.
\end{equation}
In the case \eqref{e:step_IC2}, we define
\begin{equation}
  \label{e:C_N_step2}
  \calC_{2, N}^\icstep(z)
  := \frac{ \prod_{v \in \calQ_1(z)} v^{L-N-Q} e^{t E(v)} \prod_{u \in \calQ_0(z)} (1-u)^{-Q+N-k} }{ \prod_{v \in \calQ_1(z)} \prod_{u \in \calQ_0(z)} (v-u) }.
\end{equation}

We also define the operator $K_{i, z}^\icstep$ acting on $\ell^2(\calQ_0(z))$ by the kernel
\begin{equation}
  \label{e:K_z_step}
  K_{i, z}^\icstep(u, u') = \sum_{v \in \calQ_1(z)} \frac{ f_i^\icstep(u) f_i^\icstep(v)^{-1} }{ (u-v) (u'-v) }, \qquad u, u' \in \calQ_0(z),
\end{equation}
where the function $f_i^\icstep : \calQ(z) \longrightarrow \bbC$, for $i=1,2$, is defined differently according to whether the initial step condition is in the case \eqref{e:step_IC1} or \eqref{e:step_IC2}, respectively. In the case \eqref{e:step_IC1}, we define
\begin{equation}
  \label{e:f_step1}
  f_1^\icstep(w) := \frac{ w^{m-N+2} (1-w^{-1})^{Q-N+1} e^{t E(w)} }{w - (1-\rho)}
  \cdot
  \left\{
    \begin{array}{ll}
      (q_{z, 1}(w))^2, & w \in \calQ_0(z), \\
      (q_{z, 1}'(w))^2, & w \in \calQ_1(z).
    \end{array}
  \right.
\end{equation}
In the case \eqref{e:step_IC2}, we define
\begin{equation}
  \label{e:f_step2}
  f_2^\icstep(w) := \frac{ w^{-2N+k+2} (1-w^{-1})^{Q-2N+k+1} e^{t E(w)} }{w - (1-\rho)}
  \cdot
  \left\{
    \begin{array}{ll}
      (q_{z, 1}(w))^2, & w \in \calQ_0(z), \\
      (q_{z, 1}'(w))^2, & w \in \calQ_1(z).
    \end{array}
  \right.
\end{equation}

\begin{proposition}
  \label{p:step_current}
  Let $\bbP_i^\icstep$, for $i=1, 2$, be the probability distribution for the PushASEP on a ring with step initial condition given by \eqref{e:step_IC1} or \eqref{e:step_IC1}, respectively, and let $Q$ be an integer. Then, we have
  \begin{equation}
    \label{e:step_current}
    \bbP_i^\icstep (Q_{L-1}(t) \geq Q) = \oint_{C_r} \calC_{i, N}^\icstep (z) \cdot \det \big( I + K_{i, z}^\icstep \big) \frac{\dd z}{z},
  \end{equation}
  where $0 < r < r_0 := \rho^{\rho} (1-\rho)^{1-\rho}$.
\end{proposition}

\begin{proof}
  We consider the two types of step initial conditions, \eqref{e:step_IC1} and \eqref{e:step_IC2}, separately. Below, we begin by considering the step initial conditions given by \eqref{e:step_IC1}.
  
  We start with formula for the cumulative probability distribution for the current, given by \eqref{e:current_cdf1}, and specialize to  the step initial conditions \eqref{e:step_IC1}. Then, we obtain
  \begin{align*}
    \bbP^\icstep_1 (Q_{L-1}(t) \geq Q)
    & = (-1)^{(N+1)Q} \oint_{C_{\epsilon'}} \frac{\dd z}{z^{1+QL}}
      \sum_{\substack{\lambda_i \in \calQ(z) \\ 1 \leq i \leq N}}
    \det \Bigg[ \frac{1}{L} \frac{ (1- \lambda_j^{-1})^{Q+j-i}  \lambda_j^{j+m} e^{t E(\lambda_j)}}{ \lambda_j - (1-\rho) } \Bigg]_{i,j=1}^N \\
    & = (-1)^{(N+1)Q + {N \choose 2}} \oint_{C_{\epsilon'}} \frac{\dd z}{z^{1+QL}}
      \sum_{\substack{\lambda_i \in \calQ(z) \\ 1 \leq i \leq N}}
    \det \Bigg[ \frac{1}{L} \frac{ (1- \lambda_j^{-1})^{Q+j-(N-i+1)}  \lambda_j^{j+m} e^{t E(\lambda_j)}}{ \lambda_j - (1-\rho) } \Bigg]_{i,j=1}^N \\
    & = C \oint_{C_{\epsilon'}} \frac{\dd z}{z^{1+QL}}
      \det \Big[ \sum_{\lambda \in \calQ(z)} (1-\lambda^{-1})^{i+j-2} \lambda^{j-1} f_{1,1}(\lambda) \Big]_{i,j=1}^N
  \end{align*}
  where, in the second line, we revert the line numbers $i \leftarrow N-i+1$ and, in the third line, we use the linearity of the determinant by setting
  \begin{equation*}
    C  = (-1)^{(N+1)Q + {N \choose 2}} L^{-N}, \quad  f_{1,1}(w)  =  \frac{w^{m+1} (1-w^{-1})^{Q-N+1} e^{t E(w)}}{w-(1-\rho)}.
  \end{equation*}
  Applying the Cauchy-Binet/Andreief formula, we obtain
  \begin{align*}
    \frac{C}{N!} \oint_{C_{\epsilon'}} \frac{\dd z}{z^{1+QL}}
    \sum_{\substack{\lambda_i \in \calQ(z) \\ 1 \leq i \leq N}}
    \prod_{i < j} (\lambda_i - \lambda_j)^2
    \prod_{i=1}^N \lambda_i^{-(N-1)} f_{1,1}(\lambda_i) \\
    = \frac{C}{N!} \oint_{C_{\epsilon'}} \frac{\dd z}{z^{1+QL}}
    \sum_{\substack{\lambda_i \in \calQ(z) \\ 1 \leq i \leq N}}
    \prod_{i < j} (\lambda_i - \lambda_j)^2
    \prod_{i=1}^N f_{1,2}(\lambda_i)
  \end{align*}
  where $f_{1,2}(\lambda_i) = \lambda_i^{-(N-1)} f_{1,1}(\lambda_i)$.
  The above formula takes exactly the same form as \cite[(7.45)]{baikliu2018} and may be simplified using exactly the same method as in the proof of \cite[Theorem~7.2]{baikliu2018}. 
  In particular, the above formula can be further simplified to
  \[
    (-1)^{{N \choose 2}} C_1
    \oint_{C_{\epsilon'}} \frac{\dd z}{z^{1+QL}}
    \prod_{v \in \calQ_1(z)} f_{1,2}(v) q_{z, 1}'(v) \det(I + K_{1,z}^{\icstep}),
  \]
  where the kernel $K_{1,z}^{\icstep}$ is defined to be
  \begin{equation}
    \label{e:K_z_step_intermediate}
    K_{1, z}^{\icstep}(u, u')
    := \sum_{v \in \calQ_1(z)} \frac{ f_1^{\icstep}(u) f_1^{\icstep}(v)^{-1} }{ (u-v) (u'-v) }, \qquad u, u' \in \calQ_0(z),
  \end{equation}
  with
  \begin{equation}
    \label{e:f_f2_step_intermediate}
    f_1^{\icstep}(w) := f_{1,2}(w)
    \cdot
    \left\{
      \begin{array}{ll}
        (q_{z, 1}(w))^2, & w \in \calQ_0(z), \\
        (q_{z, 1}'(w))^2, & w \in \calQ_1(z).
      \end{array}
    \right.
  \end{equation}
  Then, we  compute
  \begin{align}
    \calC_{1,N}^\icstep (z)
    & = (-1)^{{N \choose 2}} \frac{C}{z^{QL}} \prod_{v \in \calQ_1(z)} f_{1,2}(v) q_{z, 1}'(v) \nonumber \\
    & = \frac{(-1)^{(N+1)Q}}{z^{QL}} \prod_{v \in \calQ_1(z)} \frac{ v^{m-N+2} (1-v^{-1})^{Q-N+1} e^{t E(v)} }{ L( v - (1-\rho) ) } \cdot
      \frac{ L v^{L-N-1} (v-1)^{N-1} [v - (1-\rho)]} { q_{z, 0}(v) } \nonumber \\
    & =  \frac{(-1)^{(N+1)Q}}{z^{QL}}
      \prod_{v \in \calQ_1(z)} \frac{ v^{L-N+m} (v-1)^Q e^{t E(v)} }{ q_{z, 0}(v) } \nonumber \\
    & = \prod_{v \in \calQ_1(z)} \frac{ v^{L-N+m-Q} e^{t E(v)} }{ q_{z, 0}(v) }
      \prod_{u \in \calQ_0(z)} (1 - u)^{-Q},
  \end{align}
  where in the second line, we use \eqref{e:derivative_qz01}
  and in the last line, we use the relations \eqref{e:roots_relation1} and \eqref{e:roots_relation2} to simplify the expression.
  
  Now, let us consider the second type of step initial conditions given as in \eqref{e:step_IC2},
  \[
    y_j = j-1 + (L-N) \mathds{1}_{k+1 \leq j \leq N}, \quad \forall j = 1, \dots, N.
  \]
  The computation is very similar to what we just did above, so here we only write down the important intermediate steps and the computations that differ.
  First, by specializing \eqref{e:current_cdf1} for the step initial conditions \eqref{e:step_IC2}, we find
  \begin{align*}
    \bbP^\icstep_2 (Q_{L-1}(t) \geq Q)
    & =  C \oint_{C_{\epsilon'}} \frac{\dd z}{z^{1+QL}}
      \det \Big[ \sum_{\lambda \in \calQ(z)} (1-\lambda^{-1})^{i+j-2} \lambda^{j-1} f_{2,1}(\lambda) g_j(\lambda) \Big]_{i,j=1}^N,
  \end{align*}
  by setting
  \begin{equation*}
    C  = (-1)^{(N+1)Q + {N \choose 2}} L^{-N}, \quad f_{2,1}(w)  =  \frac{w (1-w^{-1})^{Q-N+1} e^{t E(w)}}{w-(1-\rho)}, \quad g_j(w)  =\begin{cases}
      w^{L-N}, & \text{if}\ k+1 \leq j \leq N, \\
      1, & \text{otherwise}.
    \end{cases}
  \end{equation*}
  Applying the Cauchy-Binet/Andreief formula, we get
  \begin{align}
    \label{e:step_exp1}
    \frac{C}{N!} \oint_{C_{\epsilon'}} \frac{\dd z}{z^{1+QL}}
    \sum_{\substack{\lambda_i \in \calQ(z) \\ 1 \leq i \leq N}}
    \prod_{i < j} (\lambda_j - \lambda_i)
    \prod_{i=1}^N \lambda_i^{-(N-1)} f_{2,1}(\lambda_i)
    \det \big[ (\lambda_j - 1)^{i-1} g_i(\lambda_j) \big]_{i, j = 1}^N.
  \end{align}
  Note that for $w \in \calQ(z)$, we may rewrite 
  \[
    g_i(w) = w^{L-N} = z^L (w - 1)^{-N},
  \]
  for $k+1 \leq i \leq N$. This allows us to compute the above determinant by first factorizing the powers $z^L$ and $(\lambda_i-1)^{-(N-k)}$ out, and then by permuting the rows,
  \begin{align*}
    \det \big[ (\lambda_j - 1)^{i-1} g_i(\lambda_j) \big]_{i, j = 1}^N
    & = (-1)^{(N-1)k} z^{(N-k)L} \prod_{i=1}^N (\lambda_i - 1)^{-(N-k)} \det \big[ (\lambda_j - 1)^{i-1} \big]_{i, j = 1}^N \\
    & = (-1)^{(N-1)k} z^{(N-k)L} \prod_{i=1}^N (\lambda_i - 1)^{-(N-k)} \prod_{i < j} (\lambda_j - \lambda_i)
  \end{align*}
  Inserting the above formula into \eqref{e:step_exp1}, we find
  \begin{align*}
    & (-1)^{(N-1)k}
      \frac{C}{N!} \oint_{C_{\epsilon'}} \frac{\dd z}{z^{1+(Q-N+k)L}}
      \sum_{\substack{\lambda_i \in \calQ(z) \\ 1 \leq i \leq N}} 
    \prod_{i < j} (\lambda_j - \lambda_i)^2
    \prod_{i=1}^N \lambda_i^{-(N-1)} (\lambda_i - 1)^{-(N-k)} f_{2,1}(\lambda_i) \\
    & = (-1)^{(N-1)k}
      \frac{C}{N!} \oint_{C_{\epsilon'}} \frac{\dd z}{z^{1+(Q-N+k)L}}
      \sum_{\substack{\lambda_i \in \calQ(z) \\ 1 \leq i \leq N}}
    \prod_{i < j} (\lambda_j - \lambda_i)^2
    \prod_{i=1}^N f_{2,2}(\lambda_i),
  \end{align*}
  where
  \begin{align*}
    f_{2,2}(w) & = w^{-(N-1)-(N-k)} (1 - w^{-1})^{-(N-k)}  f_{2,1}(w) \\
               & = \frac{ w^{-Q+1} (w-1)^{Q-2N+k+1} e^{t E(w)} }{w - (1-\rho)}.
  \end{align*}
  The above formula takes exactly the same form as \cite[(7.45)]{baikliu2018} and can be further simplified to
  \[
    (-1)^{(N-1)k + {N \choose 2}} C
    \oint_{C_{\epsilon'}} \frac{\dd z}{z^{1+(Q-N+k)L}}
    \prod_{v \in \calQ_1(z)} f_{2,2}(v) q_{z, 1}'(v) \det(I + K^{\icstep}_{2,z}),
  \]
  where the kernel $K^{\icstep}_{2, z}$ is defined to be as in \eqref{e:K_z_step_intermediate}
  \begin{equation}
    \label{e:K_z_step2_intermediate}
    K_{2, z}^{\icstep}(u, u')
    := \sum_{v \in \calQ_1(z)} \frac{ f_2^{\icstep}(u) f_2^{\icstep}(v)^{-1} }{ (u-v) (u'-v) }, \qquad u, u' \in \calQ_0(z),
  \end{equation}
  with
  \begin{equation}
    \label{e:f_f2_step2_intermediate}
    f_2^{\icstep}(w) := f_{2,2}(w)
    \cdot
    \left\{
      \begin{array}{ll}
        (q_{z, 1}(w))^2, & w \in \calQ_0(z), \\
        (q_{z, 1}'(w))^2, & w \in \calQ_1(z).
      \end{array}
    \right.
  \end{equation}
  Finally, the following computation allows us to conclude the proof,
  \begin{align}
    \calC_{2,N}^\icstep (z)
    & = (-1)^{(N-1)k + {N \choose 2}} C \prod_{v \in \calQ_1(z)} f_{2,2}(v) q_{z, 1}'(v) \nonumber \\
    & = \frac{(-1)^{(N+1)(Q+k)}}{z^{(Q-N+k)L}} \prod_{v \in \calQ_1(z)} \frac{ v^{-Q+1} (v-1)^{Q-2N+k+1} e^{t E(v)} }{ L(v - (1-\rho)) } \cdot
      \frac{ L v^{L-N-1} (v-1)^{N-1} [v - (1-\rho)]} { q_{z, 0}(v) } \nonumber \\
    & = \prod_{v \in \calQ_1(z)} \frac{ v^{L-N-Q} (v-1)^{Q-N+k} e^{t E(v)} }{ q_{z, 0}(v) } \nonumber \\
    & = \prod_{v \in \calQ_1(z)} \frac{ v^{L-N-Q} e^{t E(v)} }{ q_{z, 0}(v) }
      \prod_{u \in \calQ_0(z)} (1-u)^{-Q+N-k},
  \end{align}
  where in the last line, we use again the relations \eqref{e:roots_relation1} and \eqref{e:roots_relation2} to simplify the expression.
\end{proof}

Below, we perform the asymptotic analysis in the relaxation time scale.

\begin{proof}[Proof of \Cref{t:asymptotic_step_current}]
  Let us start with the step initial condition as in \eqref{e:step_IC1}, assume that the parameters $p$, $q$, and $\rho$ are chosen such that $\vshock \neq 0$, and set the parameters as follow,
  \begin{align*}
    t & = \frac{L}{|\vshock|} \bigg\lfloor \frac{|\vshock| \tau}{ \sqrt{\rho(1-\rho)} } L^{1/2} \bigg\rfloor - \frac{(\rho + \gamma) L + m}{\vshock}, \\
      & =  \frac{\tau L^{3/2}}{ \sqrt{\rho(1-\rho)} }
        - \frac{L}{|\vshock|} \bigg\{ \frac{|\vshock| \tau}{ \sqrt{\rho(1-\rho)} } L^{1/2} \bigg\}
        - \frac{(\rho + \gamma) L + m}{\vshock}, \\
    Q & = vt - x \rho^{2/3} (1-\rho)^{2/3} t^{1/3} - (1-\rho) N + \rho m.
  \end{align*}
  The case with $\vshock = 0$ is similar, and we omit the details in the computations below.
  Additionally, introduce the change of variables $\bfz^L = (-1)^N r_0^L z$. Then, we consider the asymptotic expansion of \eqref{e:step_current} as $N \rightarrow \infty$.
  
  First, we compute the asymptotic expansion of $\calC_{1, N}^\icstep(\bfz)$ as given in \eqref{e:C_N_step1}. In the following, we set $\calC_{N}^\icstep(\bfz) = \calC_{1, N}^\icstep(\bfz)$ for ease of notation. We express $\calC_N^\icstep(\bfz)$ as the product of the two following terms
  \begin{equation}
    \label{e:C_N_step_decomposition}
    \begin{split}
      \calC_{N, 1}^\icstep(\bfz) & = \frac{ \prod_{u \in \calQ_0(\bfz)} (1-u)^N \prod_{v \in \calQ_1(\bfz)} v^{L-N} }{ \prod_{v \in \calQ_1(\bfz)} \prod_{u \in \calQ_0(\bfz)} (v-u) }, \\
      \calC_{N, 2}^\icstep(\bfz) & = \prod_{u \in \calQ_0(\bfz)} (1-u)^{-Q-N} \prod_{v \in \calQ_1(\bfz)} v^{m-Q} e^{t E(v)}.
    \end{split}
  \end{equation}
  Using the results summarized in \cite[Section~8.1.1]{baikliu2018}, we find that
  \begin{align*}
    \calC_{N, 1}^\icstep(\bfz) & = e^{2 B(z)} (1 + O(N^{\epsilon-1/2}))
  \end{align*}
  where $B(z)$ is defined as in \eqref{e:B(z)}.
  For $\calC_{N, 2}^\icstep(\bfz)$, we follow the proof given in \cite[Section~9.3]{baikliu2018} and find
  \begin{align}
    \ln \calC_{N, 2}^\icstep(\bfz)
    & = - (Q + N) \sum_{u \in \calQ_0(z)} \ln(1-u)
      + \sum_{v \in \calQ_1(z)} \big( (m -Q ) \ln v + t E(v) \big) \nonumber \\
    & = L \bfz^L \int_{1-\rho - \icomp \infty}^{1-\rho + \icomp \infty} (G^\icstep(w) - G^\icstep(1-\rho)) \frac{w - (1-\rho)}{w(w-1) \tilde{q}_\bfz(w)} \frac{\dd w}{2 \pi \icomp}
      \label{e:Gstep_integral}
  \end{align}
  with $\tilde{q}_\bfz(w) = w^L (1-w^{-1})^N - \bfz^L = w^L (1-w^{-1})^N q_\bfz(w)$ and
  \begin{align*}
    G^\icstep(w) = - (Q + N) \ln(1-w) + (Q-m) \ln w - t E(w).
  \end{align*}

  Following the same steps as in the proof of \Cref{t:asymptotic_flat_current} , we start with the Taylor expansion (using the computer software \texttt{Sage}) with $w = 1-\rho + \alpha$ and $\alpha = \rho \sqrt{1-\rho} \xi N^{-1/2} = \calO(N^{\epsilon/4-1/2})$,
  \begin{align}
    G^\icstep(w) - G^\icstep(1-\rho)
    & =  C_1 N^{1/2} \xi^2
      + \bigg( - \tau^{1/3} x \xi  - \frac12 \gamma \xi^2 +  C_2 \tau \xi^3 \bigg) \nonumber \\
    & \qquad + E_1 N^{-1/2} \xi^2 + E_2 N^{-1/2} \xi^3 + E_3 N^{-1/2} \xi^4 + O(N^{5\epsilon/4-1}) \label{e:G_step_Taylor}
  \end{align}
  where 
  \begin{align*}
    C_1 & = \frac{1}{2} \frac{1}{\rho \sqrt{1-\rho}} \Big[ (1-2\rho) p - \frac{q}{(1-\rho)^2} \Big] = \frac{\vshock}{2 \rho \sqrt{1-\rho}}, \\
    C_2 & = \frac{1}{3 \rho (1-\rho)} \Big[ (3 \rho^2 - 3 \rho + 1) p + \frac{3\rho -1}{(1-\rho)^2} q \Big],
  \end{align*}
  as in \eqref{e:G_flat_Taylor} with the $E_i$ constants depending only on $\rho$, $\tau$, $p$ and $q$.

  Note that in the expansions of \eqref{e:G_flat_Taylor} and \eqref{e:G_step_Taylor}, the coefficients that contribute in the integrals \eqref{e:Gflat_integral} and \eqref{e:Gstep_integral} are exactly the same.
  Therefore, we find the same behavior for both $\calC_{N, 3}^\icflat(\bfz)$ and $\calC_{N, 2}^\icstep(\bfz)$.
  
  For the asymptotic behavior of the kernel $K_{1, \bfz}^\icstep$, as in \eqref{e:Kz_flat_conjugation}, we have the identity
  \begin{equation*}
    \det(I + K_{1, \bfz}^\icstep) = \det(I + \tilde{K}_{ \bfz}^\icstep),
  \end{equation*}
  where 
  \begin{equation*}
    \tilde{K}_{ z}^\icstep(u, u') = \sum_{v \in \calQ_1(z)} \frac{ h^\icstep(u) h^\icstep(v)^{-1} }{ (u-v) (u'-v) }, \qquad u, u' \in \calQ_0(z),
  \end{equation*}
  with the function $h^\icstep$ given by
  \begin{align*}
    h^\icstep(w) & := \frac{ g^\icstep(w) }{(w - (1-\rho)) (w-1)^{2N}}
                   \cdot
                   \left\{
                   \begin{array}{ll}
                     (q_{z, 1}(w))^2, & w \in \calQ_0(z), \\
                     (q_{z, 1}'(w))^2, & w \in \calQ_1(z),
                   \end{array}
                                         \right. \\
    g_1^\icstep(w) & = \frac{ \tilde{g}^\icstep(w) [ w^{L-N} (w-1)^N ]^{k}  }{ \tilde{g}^\icstep(1-\rho) [ (1-\rho)^{L-N} (-\rho)^N ]^{k}  }, \\
    \tilde{g}^\icstep(w) &= w^{m-Q+1} (w-1)^{Q+N+1} e^{t E(w)},
  \end{align*}
  and the exponent $k$ is chosen to be
  \begin{equation}
    \label{e:g_step_k}
    k = - \sgn(\vshock) \cdot \bigg\lfloor \frac{|\vshock| \tau}{ \sqrt{\rho(1-\rho)} } L^{1/2} \bigg\rfloor
    =  - \frac{\vshock \tau}{ \sqrt{\rho(1-\rho)} } L^{1/2}
    + \sgn(\vshock) \cdot \bigg\{ \frac{|\vshock| \tau}{ \sqrt{\rho(1-\rho)} } L^{1/2} \bigg\}.
  \end{equation}
  This gives us
  \[
    \ln g^\icstep(w) = \tau^{1/3} x \xi + \frac12 \gamma \xi^2 - \frac{\tau}{3} \bigg( p + \frac{q}{(1-\rho)^3} \bigg) \xi^3  + \calO(N^{-1/2}).
  \]

  We now consider the asymptotic expansion of the Fredholm determinant. In particular, after careful analysis similar to \cite[Section~8.1.2]{baikliu2018}, we may show that
  \begin{equation*}
    \lim_{N \to \infty} \det(I + \tilde{K}_{ \bfz}^\icstep) = \det(I - \calK_{z}^\icstep).
  \end{equation*}
  The operator $\calK_z^\icstep$ is defined on $\calS_{-}(z)$ with kernel
  \begin{equation*}
    \calK_{z}^\icstep(\xi_1, \xi_2) = \sum_{\eta \in \calS_{-}(z)} \frac{ \exp \big( \Phi_z(\xi_1; \tau^{1/3} x, r\tau) + \Phi_z(\eta; \tau^{1/3} x, r\tau) + \frac{\gamma}{2} (\xi_1^2 - \eta^2) \big) }{ \xi_1 \eta (\xi_1 + \eta) (\xi_2 + \eta) }.
  \end{equation*}
  The details of the proof is similar to and may be found in \cite[Section~8.1.2]{baikliu2018}; we omit the details here.

  For the step initial condition as in \eqref{e:step_IC2}, we only describe the main differences below. If the parameters $p$, $q$, and $\rho$ are chosen such that $\vshock \neq 0$, we set the parameters as follow,
  \begin{align*}
    t & = \frac{L}{|\vshock|} \bigg\lfloor \frac{|\vshock| \tau}{ \sqrt{\rho(1-\rho)} } L^{1/2} \bigg\rfloor - \frac{k + \gamma L}{\vshock}, \\
    Q & = vt - x \rho^{2/3} (1-\rho)^{2/3} t^{1/3} - (1-\rho) k.
  \end{align*}
  Then, in the asymptotic computations for $\calC_{N}^\icstep(\bfz) = \calC_{2, N}^\icstep(\bfz)$, we have the same decomposition as in \eqref{e:C_N_step_decomposition} with $\calC_{N, 2}^\icstep(\bfz)$ replaced by
  \begin{align*}
    \calC_{N, 2}^\icstep(\bfz) & = \prod_{u \in \calQ_0(\bfz)} (1-u)^{-Q-k} \prod_{v \in \calQ_1(\bfz)} v^{-Q} e^{t E(v)}.
  \end{align*}
  As a result, the function $G^\icstep(w)$ changes as follows 
  \begin{align*}
    G^\icstep(w) = - (Q + k) \ln(1-w) + Q \ln w - t E(w).
  \end{align*}

  Following the same steps as before, the Taylor expansion (using the computer software \texttt{Sage}) gives
  \begin{align*}
    G^\icstep(w) - G^\icstep(1-\rho)
    & = C_1 N^{1/2} \xi^2
      + \bigg( - \tau^{1/3} x \xi - \frac12 \gamma \xi^2 + C_2 \tau \xi^3 \bigg) \\
    & \qquad + E_1 N^{-1/2} \xi^2 + E_2 N^{-1/2} \xi^3 + E_3 N^{-1/2} \xi^4 + O(N^{5\epsilon/4-1})
  \end{align*}
  for $w = 1-\rho + \alpha$ and $\alpha = \rho \sqrt{1-\rho} \xi N^{-1/2} = \calO(N^{\epsilon/4-1/2})$, which is exactly the same expression as in \eqref{e:G_step_Taylor}.
  Similarly, for the kernel asymptotics, we introduce the same quantities with the only difference that
  \begin{align*}
    g^\icstep(w) & = \frac{ \tilde{g}^\icstep(w) [ w^{L-N} (w-1)^N ]^{k} }{ \tilde{g}^\icstep(1-\rho) [ (1-\rho)^{L-N} (-\rho)^N ]^{k}}, \\
    \tilde{g}^\icstep(w) &= w^{-Q+1} (w-1)^{Q+k+1} e^{t E(w)},
  \end{align*}
  where the exponent $k$ is chosen to be the same as in \eqref{e:g_step_k},
  \begin{equation*}
    k = - \sgn(\vshock) \cdot \bigg\lfloor \frac{|\vshock| \tau}{ \sqrt{\rho(1-\rho)} } L^{1/2} \bigg\rfloor
    =  - \frac{\vshock \tau}{ \sqrt{\rho(1-\rho)} } L^{1/2}
    + \sgn(\vshock) \cdot \bigg\{ \frac{|\vshock| \tau}{ \sqrt{\rho(1-\rho)} } L^{1/2} \bigg\},
  \end{equation*}
  leading to the following asymptotics
  \[
    \ln g^\icstep(w) = \tau^{1/3} x \xi + \frac12 \gamma \xi^2 - \frac{\tau}{3} \bigg( p + \frac{q}{(1-\rho)^3} \bigg) \xi^3  + \calO(N^{-1/2}).
  \]

\end{proof}

\section{Generating series expansion}
\label{s:prob_fun_expansion}

We write the the generating series for the joint distribution $g(X;\zeta; t)$ as an infinite sum of integrals by taking a series expansion of the coupling term in the formula \eqref{e:main_1} given in \Cref{t:main}. The key in this argument is to write the series expansion of the coupling term in terms of the amplitude coefficients $A_{\sigma}$ and to determine the action of the product of two amplitude coefficients. These results are given below; see \Cref{l:coupling_expansion} and \Cref{l:amp_product}. Then, given those results, we prove \Cref{p:method_images}.

\begin{lemma}\label{l:coupling_expansion}
  Take $z, \zeta \in \bbC$ with $|\zeta|=1$ and $|z| = R',\epsilon' $ so that \eqref{e:conditions} are satisfied. Let $\lambda_i = \lambda_i(z)$, for $i=1, \cdots, N$, be solutions of the decoupled Bethe equation \eqref{e:dBE1} and $s = (N\, N-1 \, \cdots \, 1) \in S_N$ denote the shift permutation. Then, we have
  \begin{equation}
    \frac{1}{1 - (-1)^{N-1}\zeta^{L}z^{-L}\prod_{i=1}^N(1- \lambda_i^{-1})} = \sum_{k=0}^{\infty}\sum_{m=0}^{N-1} A_{s^m}^{-1} \prod_{i=1}^N\left(\frac{\lambda_i}{\zeta}\right)^{-L k}   \prod_{j=0}^{m-1} \left(\frac{\lambda_{N-j}}{\zeta}\right)^{-L}
  \end{equation}
  for $|z| = R'$ and
  \begin{equation}
    \frac{1}{1 - (-1)^{N-1}\zeta^Lz^{-L}\prod_{i=1}^N(1- \lambda_i^{-1})} = -\sum_{k=-1}^{-\infty}\sum_{m=0}^{N-1}A_{s^m}^{-1} \prod_{i=1}^N\left(\frac{\lambda_i}{\zeta}\right)^{-L k}   \prod_{j=0}^{m-1} \left(\frac{\lambda_{N-j}}{\zeta}\right)^{-L}
  \end{equation}
  for $|z| = \epsilon'$ with at least one root so that $\lambda_i \in \mathcal{Q}_0(z)$, where the set $\mathcal{Q}_0(z)$ is given by \eqref{e:q_small_sol}.
\end{lemma}

\begin{proof}
  In both cases, we take a series expansion of the terms on the left side of the identities. Then, after applying certain identities due the decoupled Bethe equations \eqref{e:dBE1}, we obtain the result.
  
  For $|z|=R'$, we have
  \begin{equation}
    \left|\zeta^{L}z^{-L} \prod(1 - \lambda_i^{-1}) \right|  \ll 1
  \end{equation}
  On the other hand, for $|z|=\epsilon'$ and at least one root so that $\lambda_i \in \mathcal{Q}_0(z)$, we have
  \begin{equation}
    \left|\zeta^{L}z^{-L} \prod(1 - \lambda_i^{-1}) \right| \gg 1.
  \end{equation}
  The case with $|z| = \epsilon'$ and none of the roots in the set $\mathcal{Q}_0(z)$ is excluded since the norm of the quantity above is not necessarily less than or greater than one, which depends on the argument of $z$.

  Given the inequalities above, we have the following series expansions. For $|z| = R'$, we have
  \begin{equation}
    \frac{1}{1 - (-1)^{N-1}\zeta^{L}z^{-L}\prod_{i=1}^N(1- \lambda_i^{-1})} = \sum_{r=0}^{\infty} (-1)^{(N-1) r}\zeta^{L r}z^{-L r}\prod_{i=1}^N(1- \lambda_i^{-1})^r.
  \end{equation}
  Similarly, for for $|z|=\epsilon'$ and at least one root so that $\lambda_i \in \mathcal{Q}_0(z)$, we have
  \begin{equation}
    \frac{1}{1 - (-1)^{N-1}\zeta^Lz^{-L}\prod_{i=1}^N(1- \lambda_i^{-1})} = -\sum_{r=-1}^{-\infty} (-1)^{(N-1) r}\zeta^{L r}z^{-L r}\prod_{i=1}^N(1- \lambda_i^{-1})^r.
  \end{equation}
  Since both summations have the same terms, up to an overall sign, and the only difference is the range of the exponent $r$, we proceed the rest of the proof treating both cases simultaneously assuming that $r \in \mathbb{Z}$.

  We rewrite the terms in the sum, using the decoupled Bethe equation \eqref{e:dBE1} and the amplitude coefficients $A_{\sigma}$. First, as a straightforward consequence of the decoupled Bethe equation \eqref{e:dBE1}, we have
  \begin{equation}
    \prod_{i=1}^N(1- \lambda_i^{-1})^{N k} = z^{ L N k} \prod_{i=1}^{N} \lambda_i^{- Lk},
  \end{equation}
  for any $k \in \bbZ$. Next, we write
  \begin{equation}
    \prod_{i=1}^N (1- \lambda_i^{-1})^m = \prod_{j=0}^{m-1}(1 - \lambda_{N-j}^{-1})^{N} \frac{\prod_{i=1}^N (1- \lambda_i^{-1})^m}{\prod_{j=0}^{m-1}(1 - \lambda_{N-j}^{-1})^{N}},
  \end{equation}
  for $m=0,1,  \cdots, N-1$. We simplify the first term on the right side of the previous identity using \eqref{e:dBE1} and note that the second term on the right side is equal to $A_{s^m}^{-1}$, up to a sign, since
  \begin{equation}
    s^{m}(i)-i = \begin{cases}
      -m, \quad i > m \\
      N- m, \quad i \leq m,
    \end{cases}
  \end{equation}
  for $m =0, 1, \cdots, N-1$. That is, we have
  \begin{equation}
    \prod_{i=1}^N (1- \lambda_i^{-1})^m = (-1)^{(N-1)m}z^{L m} A_{s^m}^{-1} \prod_{i=0}^{m-1} \lambda_{N-i}^{-L},
  \end{equation}
  for $m = 0, 1, \cdots,N$. Then, for any $r \in \mathbb{Z}$, we write $r = N k + m$ for some $k \in \mathbb{Z}$ and $0 \leq m \leq N-1$, and we have
  \begin{equation}
    (-1)^{(N-1)r}\zeta^{L r}z^{-L r}\prod_{i=1}^N(1- \lambda_i^{-1})^r =  A_{s^m}^{-1} \prod_{i=1}^N\left(\frac{\lambda_i}{\zeta}\right)^{-L k}   \prod_{j=0}^{m-1} \left( \frac{\lambda_{N-j}}{\zeta}\right)^{-L}.
  \end{equation}
  The result follows by rewriting each term in the series expansion above using the previous identity. Additionally, note that the overall exponent on the $-1$ terms, after rewriting the terms, equals $(N-1)(N k + 2m)$, which is even, making the term equal to one. Lastly, we separate the two cases by noting the following set decompositions
  \begin{equation}
    \{r \leq -1\} = \{N k + m \mid k \leq -1,\, m=0, \dots, N-1 \}, \quad \{r \geq 0\} = \{Nk +m \mid k \geq 0,\, m=0, \dots, N-1 \}.
  \end{equation}
  Thus, we obtain the result.
\end{proof}

Below, we use the use the following notation for the action of the symmetric group on the variables of a function:
\begin{equation}
  \sigma \cdot f(w_1, \dots, w_N) := f(w_{\sigma(1)}, \dots, w_{\sigma(N)})
\end{equation}
for any function $f$ with $N \in \bbN$ variables and any permutation $\sigma \in S_N$.

\begin{lemma}\label{l:amp_product}
  Take the amplitude coefficients $A_{\sigma}$ given by \eqref{e:bethe_coeffs}. Then,
  \begin{equation}
    A_{\tau}^{-1}(\vec{w}) A_{\sigma}(\vec{w}) = \tau \cdot A_{\tau^{-1}\cdot \sigma}(\vec{w})
  \end{equation}
  for any $\tau , \sigma \in S_N$.
\end{lemma}

\begin{proof}
  This is just a computation that follows from the definition of the amplitude coefficients:
  \begin{equation}
    \begin{split}
      A_{\tau}^{-1}(\vec{w}) A_{\sigma}(\vec{w}) &= (-1)^{\tau + \sigma} \prod_{i=1}^N(1 - w_{\tau(i)}^{-1})^{i - \tau(i)} \prod_{i=1}^N(1-w_{\sigma(i)}^{-1})^{\sigma(i)-i}\\
      &=(-1)^{\tau + \sigma} \prod_{i=1}^N(1 - w_{i}^{-1})^{\tau^{-1}(i) - i} \prod_{i=1}^N(1-w_{i}^{-1})^{i-\sigma^{-1}(i)}\\
      &= (-1)^{\tau + \sigma} \prod_{i=1}^N(1 - w_{i}^{-1})^{\tau^{-1}(i) - \sigma^{-1}(i)}\\
      & =(-1)^{\tau^{-1}\cdot \sigma} \prod_{i=1}^N(1- w_{\sigma(i)}^{-1})^{\tau^{-1}\cdot \sigma (i)- i} = \tau \cdot A_{\tau^{-1} \cdot \sigma}( \vec{w}).
    \end{split}
  \end{equation}
\end{proof}

The expansion of the probability function is given as a sum functions that resemble the probability function without the coupling terms. Recall the function
\begin{equation}
  \begin{split}
    u^{\rm BL(p,q)}(Y, X; t) := 
    & \CIpi{1} \oint_{C_{R'}} \frac{\dd z}{z}  \sum_{\SumLambdas}
    \frac{ \sum_{\sigma \in S_N} A_{\sigma} \prod_{i =1}^N \left( ( \lambda_{\sigma(i)}/ \zeta)^{y_{\sigma(i)} - x_i} e^{tE(\lambda_i)} \right)}
    { \prod_{i =1}^N \left( \frac{L \lambda_i - (L-N)}{\lambda_i - 1} \right)}\\
    = & \CIpi{1} \oint_{C_{R'}} \frac{\dd z}{z} \sum_{\SumLambdas}
    \det \left[ \frac{1}{L} \frac{ (1- \lambda_j^{-1})^{j-i+1}  \lambda_j^{y_j - x_i +1} e^{tE(\lambda_i)}}
      { \lambda_j - (1-\rho) } \right]_{i,j=1}^N \prod_{i=1}^{N} \zeta^{x_i -y_i},
  \end{split}
\end{equation}
for any $X, Y \in \bbZ^N$, which was also given earlier in \eqref{e:prob_fun_winding}. Now, we are set to prove \Cref{p:method_images}, which give an alternate series expansion of the generating series for the joint probability distribution $g(Y,X; \zeta;t)$.

\begin{proof}[Proof of \Cref{p:method_images}]
  We take the formula given by \eqref{e:main_1} for the generating series of the joint probability function. Also, we express the $u_{0}(\zeta)$ function by the identity given in \eqref{e:u0_alt}.
  Then, we have
  \begin{equation}
    g(X; \zeta; t)= \oint^{R'}_{\epsilon'} \frac{\dd z}{z}  \sum_{\substack{\lambda_{i} \in \mathcal{Q}(z)\\ 1 \leq i \leq N}}^{\prime}
    \frac{ \sum_{\sigma \in S_N} A_{\sigma} \prod_{i =1}^N \left( ( \lambda_{\sigma(i)}/ \zeta)^{y_{\sigma(i)} - x_i} e^{tE(\lambda_i)} \right)}
    { p_z(\vec{\lambda}; \zeta)\prod_{i =1}^N \left( \frac{L \lambda_i - (L-N)}{\lambda_i - 1} \right)}
  \end{equation}
  so that the summation excludes the case where $|z|= \epsilon'$ and $\lambda_i \in \calQ_1(z)$ for all $i=1, \dots, N$ -- we write the prime on the summation sign to denote the sum excluding this case. Note that the case we are excluding follows from a cancellation due to the expression \eqref{e:u0_alt} for $u_0(\zeta)$.

  Recall that 
  \begin{equation}
    p_z(\vec{\lambda}; \zeta) = 1 - (-1)^{N-1}\zeta^Lz^{-L}\prod_{i=1}^N(1- \lambda_i^{-1}).
  \end{equation}
  Then, by \Cref{l:coupling_expansion}, we have
  \begin{equation}
    \begin{split}
      &g(X; \zeta; t)=\\
      &\sum_{k=0}^{\infty}\sum_{m=0}^{N-1}\oint_{C_{R'}} \frac{\dd z}{z} \sum_{i=1}^{N} \sum_{\SumLambdas}^{\prime}
      \frac{ \sum_{\sigma \in S_N} A_{s^m}^{-1} A_{\sigma} \prod_{i =1}^N  \left( \frac{\lambda_{\sigma(i)}}{ \zeta}\right)^{y_{\sigma(i)} -k L - x_i}  \prod_{j=0}^{m-1} \left(\frac{\lambda_{N-j}}{\zeta}\right)^{-L} e^{tE(\vec{\lambda})}}
      { \prod_{i =1}^N \left( \frac{L \lambda_i - (L-N)}{\lambda_i - 1} \right)}\\
      &+\sum_{k=-1}^{-\infty}\sum_{m=0}^{N-1}\oint_{C_{\epsilon'}} \frac{\dd z}{z} \sum_{\SumLambdas}^{\prime}
      \frac{ \sum_{\sigma \in S_N} A_{s^m}^{-1}A_{\sigma} \prod_{i =1}^N  \left( \frac{\lambda_{\sigma(i)}}{ \zeta}\right)^{y_{\sigma(i)}-kL - x_i} \prod_{j=0}^{m-1} \left(\frac{\lambda_{N-j}}{\zeta}\right)^{-L}e^{tE(\vec{\lambda})}}
      { \prod_{i =1}^N \left( \frac{L \lambda_i - (L-N)}{\lambda_i - 1} \right)}.
    \end{split}
  \end{equation}
  Now, after taking the series expansion, we can extend the summation and include the case where $|z|= \epsilon'$ and $\lambda_i \in \calQ_1(z)$ for all $i=1, \dots, N$. More precisely, we have that 
  \begin{equation}
    \oint_{C_{\epsilon'}} \frac{\dd z}{z} \sum_{\SumLambdasR}
    \frac{ \sum_{\sigma \in S_N} (s^{m}\cdot A_{s^{-m} \cdot \sigma})\prod_{i =1}^N  \left( \frac{\lambda_{\sigma(i)}}{ \zeta}\right)^{y_{\sigma(i)} -k L - x_i}  \prod_{j=0}^{m-1} \left(\frac{\lambda_{N-j}}{\zeta}\right)^{-L} e^{tE(\vec{\lambda})}}
    { \prod_{i =1}^N \left( \frac{L \lambda_i - (L-N)}{\lambda_i - 1} \right)} = 0.
  \end{equation}
  The argument for this identity follows the same computations as in the proof of \Cref{p:current_cdf} where we showed a similar result; see the argument between \eqref{e:Q_pos2} and \eqref{e:Q_pos3}. Then, we have
  \begin{equation}
    \begin{split}
      &g(X; \zeta; t)=\\
      &\sum_{k=0}^{\infty}\sum_{m=0}^{N-1}\oint_{C_{R'}} \frac{\dd z}{z} \sum_{\SumLambdas}
      \frac{ \sum_{\sigma \in S_N} A_{s^m}^{-1} A_{\sigma} \prod_{i =1}^N  \left( \frac{\lambda_{\sigma(i)}}{ \zeta}\right)^{y_{\sigma(i)} -k L - x_i}  \prod_{j=0}^{m-1} \left(\frac{\lambda_{N-j}}{\zeta}\right)^{-L} e^{tE(\vec{\lambda})}}
      { \prod_{i =1}^N \left( \frac{L \lambda_i - (L-N)}{\lambda_i - 1} \right)}\\
      &+\sum_{k=-1}^{-\infty}\sum_{m=0}^{N-1}\oint_{C_{\epsilon'}} \frac{\dd z}{z} \sum_{\SumLambdas}
      \frac{ \sum_{\sigma \in S_N} A_{s^m}^{-1}A_{\sigma} \prod_{i =1}^N  \left( \frac{\lambda_{\sigma(i)}}{ \zeta}\right)^{y_{\sigma(i)}-kL - x_i} \prod_{j=0}^{m-1} \left(\frac{\lambda_{N-j}}{\zeta}\right)^{-L}e^{tE(\vec{\lambda})}}
      { \prod_{i =1}^N \left( \frac{L \lambda_i - (L-N)}{\lambda_i - 1} \right)}.
    \end{split}
  \end{equation}
  Next, we deform the contour $C_{\epsilon'}$ to the contour $C_{R'}$. This deformation does not introduce any residue terms due to \Cref{l:well_def_hol}. Then, we have
  \begin{equation}
    \begin{split}
      &g(X; \zeta; t)=\\
      &\sum_{k=-\infty}^{\infty}\sum_{m=0}^{N-1}\oint_{C_{R'}} \frac{\dd z}{z} \sum_{\SumLambdas}
      \frac{ \sum_{\sigma \in S_N} A_{s^m}^{-1} A_{\sigma} \prod_{i =1}^N  \left( \frac{\lambda_{\sigma(i)}}{ \zeta}\right)^{y_{\sigma(i)} -k L - x_i}  \prod_{j=0}^{m-1} \left(\frac{\lambda_{N-j}}{\zeta}\right)^{-L} e^{tE(\vec{\lambda})}}
      { \prod_{i =1}^N \left( \frac{L \lambda_i - (L-N)}{\lambda_i - 1} \right)}.
    \end{split}
  \end{equation}
  
  We now conclude the proof by simplifying the integrand by using \Cref{l:amp_product}. We have
  \begin{equation}
    \begin{split}
      &g(X; \zeta; t)\\
      &=\sum_{k=-\infty}^{\infty}\sum_{m=0}^{N-1}\oint_{C_{R'}} \frac{\dd z}{z} \sum_{\SumLambdas}
      \frac{ \sum_{\sigma \in S_N} (s^{m}\cdot A_{s^{-m} \cdot \sigma})\prod_{i =1}^N  \left( \frac{\lambda_{\sigma(i)}}{ \zeta}\right)^{y_{\sigma(i)} -k L - x_i}  \prod_{j=0}^{m-1} \left(\frac{\lambda_{N-j}}{\zeta}\right)^{-L} e^{tE(\vec{\lambda})}}
      { \prod_{i =1}^N \left( \frac{L \lambda_i - (L-N)}{\lambda_i - 1} \right)}\\
      &=\sum_{k=\infty}^{\infty}\sum_{m=0}^{N-1}\oint_{C_{R'}} \frac{\dd z}{z} \sum_{\SumLambdas}
      s^m \cdot \left(\frac{ \sum_{\sigma \in S_N}  A_{s^{-m} \cdot \sigma} \prod_{i =1}^N  \left( \frac{\lambda_{s^{-m}\cdot\sigma(i)}}{ \zeta}\right)^{\tilde{y}_{s^{-m} \cdot \sigma(i)} -k L - x_i}  \prod_{j=1}^{m} \left(\frac{\lambda_{j}}{\zeta}\right)^{-L} e^{tE(\vec{\lambda})}}
        { \prod_{i =1}^N \left( \frac{L \lambda_i - (L-N)}{\lambda_i - 1} \right)} \right)\\
      &= \sum_{k=\infty}^{\infty}\sum_{m=0}^{N-1}\oint_{C_{R'}} \frac{\dd z}{z} \sum_{\SumLambdas}
      \frac{ \sum_{\sigma \in S_N}  A_{ \sigma} \prod_{i =1}^N  \left( \frac{\lambda_{\sigma(i)}}{ \zeta}\right)^{\tilde{y}_{\sigma(i)} -k L - x_i}  \prod_{j=1}^{m} \left(\frac{\lambda_{j}}{\zeta}\right)^{-L} e^{tE(\vec{\lambda})}}
      { \prod_{i =1}^N \left( \frac{L \lambda_i - (L-N)}{\lambda_i - 1} \right)}\\
      &= \sum_{k=\infty}^{\infty}\sum_{m=0}^{N-1}  u^{\rm BL(p,q)}(\widetilde{Y}^m - (kL, \dots, kL), X; t) \zeta^{m L} \prod_{i=1}^N \zeta^{x_i - y_i + k L} ,\\
      &= \sum_{m=-\infty}^{\infty} u^{\rm BL(p,q)}(Y^m, X; t) \prod_{i=1}^N \zeta^{x_i - y^m_i}
    \end{split}
  \end{equation}
  where $\widetilde{Y}^m = (\tilde{y}^m_i)_{i=1}^N := y_{s^m(i)}$ and $Y^m = (y^m_i)_{i=1}^N$ is the $m^{th}$ periodic shift of $Y$ given by \eqref{e:Y_m_shift}. Note that the third to last equality is due to relabeling the elements of the symmetric group, i.e.~$s^{-m} \cdot \sigma \mapsto \sigma$, and relabeling the decoupled Bethe roots, i.e.~$\lambda_{s^m(i)} \rightarrow \lambda_{i}$. In particular, for each fixed $k$ and $m$, the terms are equal after relabeling, both the symmetric group elements and the Bethe roots, since we are summing over all permutations and all decoupled Bethe roots. Thus, we have the desired result.
\end{proof}

\section*{Acknowledgements}

The work of J.-H.L.~and A.S.~was partially supported by the EPSRC grant EP/R024456/1. Both authors would like to give a warm thanks to Nikos Zygouras for his support that led to this collaboration.
J.-H.L. also acknowledges support from National Science and Technology Council of Taiwan through grant 110-2115-M-002-013-MY3.

\appendix

\section{Delta basis, proof}
\label{a:delta_basis}

In this appendix, we give the proof of \Cref{p:delta_contour_formula}. We will decompose the nested contour integrals into a linear combination of $2 \cdot 3^N$ nested contour integrals. We do this by writing each contour integral as a linear combination of contour integrals so that the contour of each integral is connected. Then, we show that one of the $2 \cdot 3^N$ nested contour integrals gives us the delta function, while the rest of the integrals are identically zero. We evaluate each of the $2 \cdot 3^N$ nested contour integrals by taking a series expansion of the integrand and, consecutively, taking the integrals with respect to the $z$-variable and the $w$-variables. We give the details below.

First, we introduce the following notation for contour integrals that will appear in the arguments below. Take a partition $I \sqcup J \sqcup K = [N] := \{1, \cdots, N \}$ and let $r >0$ be radius of the contour for the $z$-variable. Then, we define
\begin{equation}
  \label{e:CI_R_epsilon_v2}
  \CI(r, I, J, K)
  = (-1)^{ |I| + |J|} \oint_{C_{r}} \frac{\dd z}{z}
  \prod_{i \in I} \oint_{C_R} \frac{\dd w_i}{w_i}
  \prod_{i \in J} \oint_{C_{\epsilon_1}} \frac{\dd w_i}{w_i}
  \prod_{i \in K} \oint_{1+C_{\epsilon_2}} \frac{\dd w_i}{w_i}
  \frac{ h_{\vec{w}} (Y) u_{\vec{w}} (X) h_{\zeta}(X-Y)}{ p_z( \vec{w}; \zeta ) \prod_{i=1}^N q_z(w_i) }.
\end{equation}
Note that $I$ (resp. $J$ and $K$) denotes the set of indices $1 \leq i \leq N$ such that $w_i \in C_R$ (resp.~$w_i \in C_{\epsilon_1}$ and $w_i \in 1 + C_{\epsilon_2}$). Below, we begin with some preliminary result, \Cref{l:CI_delta_function} and \Cref{l:CI_zero}, giving the value of these terms for different cases of the partition $I \sqcup J \sqcup K = [N]$ and the radius $r>0$.

\begin{lemma}
  \label{l:CI_delta_function}
  Using the notation \eqref{e:CI_R_epsilon_v2}, we have
  \begin{align*}
    \CI(R', [N], \emptyset, \emptyset) = \mathds{1} (X = Y).
  \end{align*}
  so that $R'>0$ satisfies \eqref{e:conditions}.
\end{lemma}

\begin{proof}
  This identity follows by doing a series expansion of the denominator, which can also be seen, equivalently, as the formula for TASEP on the interger line from \cite{Schutz97} and the generalization from \cite{TW08a}.
  
  We first consider the series expansion of the denominators of the integrand.
  Given $z \in C_{R'}$ and $w_i \in C_{R}$ for $i = 1, \cdots, N$, we have the following series expansions,
  \begin{align*}
    \frac{1}{p_z(\vec{w}; \zeta)} = \frac{1}{1 + (-1)^N \zeta^L z^{-L} \prod_{i=1}^N (1- w_i^{-1})}
    & = \sum_{k =0}^{\infty} \left( (-1)^{N-1} \zeta^L z^{-L} \prod_{i=1}^N (1- w_i^{-1}) \right)^k, \\
    \frac{1}{q_z(w_i)} = \frac{1}{1 - z^{L} w_i^{-L}(1-w_i^{-1})^{-N}}
    & = \sum_{k=0}^{\infty} \left(z^{L} w_i^{-L}(1-w_i^{-1})^{-N} \right)^k,
  \end{align*}
  for all $i=1, \cdots, N$, where we use \eqref{e:z_contour}, \eqref{e:powers} and \eqref{e:beta} to guarantee convergence of the series expansion.
  
  We compute the integral with respect to $z$, term by term, after taking the series expansion of the denominator. Note that the integral vanishes unless the exponent of $z$ is exactly $-1$.
  Therefore, we find
  \begin{align*}
    & \oint_{C_R} \frac{\dd z}{z} \oint_{C_{R}} \frac{\dd w_1}{w_1} \cdots \oint_{C_{R}} \frac{\dd w_N}{w_N}
      \frac{ h_{\vec{w}} (Y) u_{\vec{w}} (X) }{ p_z( \vec{w} ) \prod_{i=1}^N q_z(w_i) } \\
    & = \CIpi{N} \oint_{C_{R}} \frac{\dd w_1}{w_1} \cdots \oint_{C_{R}} \frac{\dd w_N}{w_N} \sum_{\sigma \in S_N} A_{\sigma} \prod_{i =1}^N w_{\sigma(i)}^{y_{\sigma(i)} - x_i} \prod_{i =1}^N \frac{1}{1 + (-1)^N \zeta^L (1-w_i^{-1})^{1-N} w_i^{-L}}.
  \end{align*}
  We continue by writing the term in the second product as follows,
  \begin{equation}
    \label{e:denom_decomp}
    \frac{1}{1 + (-1)^N \zeta^L (1-w_i^{-1})^{1-N} w_i^{-L}}
    = 1 - \frac{ (-1)^N \zeta^L (1-w_i^{-1})^{1-N} w_i^{-L}}{1 + (-1)^N \zeta^L (1-w_i^{-1})^{1-N} w_i^{-L}}.
  \end{equation}
  Note that
  \begin{equation}
    |A_{\sigma}(\vec{w})| = \calO(1), \quad
    \left| \frac{ (-1)^N \zeta^L (1-w_i^{-1})^{1-N}w_i^{-L}}{1 + (-1)^N \zeta^L (1-w_i^{-1})^{1-N} w_i^{-L}} \right| = \calO(R^{-L}), \quad
    \left| w_{\sigma(i)}^{y_{\sigma(i)}-x_{i}-1} \right| = \calO(R^{L-2})
  \end{equation}
  as $|w_i|= R \rightarrow \infty$.
  Thus, if we expand the product in the integrand using \eqref{e:denom_decomp}, the only contributing term is the leading term given by one. That is,
  \begin{align*}
    & \CIpi{N} \oint_{C_{R}} \frac{\dd w_1}{w_1} \cdots \oint_{C_{R}} \frac{\dd w_N}{w_N}
      \sum_{\sigma \in S_N} A_{\sigma} \prod_{i =1}^N w_{\sigma(i)}^{y_{\sigma(i)} - x_i} \prod_{i =1}^N \frac{1}{1 + (-1)^N \zeta^L (1-w_i^{-1})^{1-N} w_i^{-L}} \\
    = & \CIpi{N} \oint_{C_{R}} \frac{\dd w_1}{w_1} \cdots \oint_{C_{R}} \frac{\dd w_N}{w_N}
        \sum_{\sigma \in S_N} A_{\sigma} \prod_{i =1}^N w_{\sigma(i)}^{y_{\sigma(i)} - x_i},
  \end{align*}
  since any term that is not equal to $1$ from the expansion of the second product in the first line will vanish by deforming the contour of some $w_i$ to infinity. The remaining contour integral is the same as the contour integral for the TASEP on the line, see \cite{TW08a}. Then, our result follows from Theorem 2.1 in \cite{TW08a}.
\end{proof}

\begin{rem}
  The formula we use for TASEP on the line in the proof of \Cref{l:CI_delta_function} uses large contours instead of the typical small contours, but one may easily check that this does not affect the result by performing a change of variables $w \mapsto 1/w$. 
\end{rem}

\begin{lemma}
  \label{l:CI_zero}
  Using the notation \eqref{e:CI_R_epsilon_v2}, we have
  \begin{enumerate}[label=(\alph*)]
    \item $\CI(R', I, J, K) = 0$ if $J \sqcup K \neq \emptyset$,
    \item $\CI(\epsilon', I, J, K) = 0$ if $|K| \leq N-1$,
    \item $\CI(\epsilon', I, J, K) = 0$ if $|K| = N$.
  \end{enumerate}
\end{lemma}

\begin{proof}
  We follow the same ideas as in the proof of \Cref{l:CI_delta_function}. We use series expansions to integrate out the $z$ variable, then deform one of the $w$ contours in a specific way to establish the result.

  First, we note that the series expansion of $1/p_z(\vec{w} ; \zeta)$ depends on whether $| \zeta^L z^{-L} \textprod_{i=1}^N (1- w_i^{-1}) |$ is greater or smaller than 1.
  Recall that $|\zeta|=1$ and consider the different cases in the statement of the Lemma.
  
  In case (a), since $z \in C_{R'}$, using $\beta_1 < d$ in \eqref{e:beta1_bounds}, we have
  \begin{align*}
    | \zeta^L z^{-L} \textprod_{i=1}^N (1- w_i^{-1}) | = \calO \big( (\epsilon')^L \epsilon_1^{-N} \big)
    = \calO \big( (\epsilon')^{L - \beta_1 N}  \big) \ll 1.
  \end{align*}
  In case (b), since $z \in C_{\epsilon'}$, using $\beta_2 < \tfrac{L}{N-1}$ in \eqref{e:beta2_bounds}, we have
  \begin{align*}
    | \zeta^L z^{-L} \textprod_{i=1}^N (1- w_i^{-1}) |
    = \Omega \big( (\epsilon')^{-L} \epsilon_2^{N-1} \big)
    = \Omega \big( (\epsilon')^{-L + \beta_2 (N-1)}  \big) \gg 1.
  \end{align*}
  In case (c), since $z \in C_{\epsilon'}$, using $\beta_2 > \tfrac{L}{N}$ in \eqref{e:beta2_bounds}, we have
  \begin{align*}
    | \zeta^L z^{-L} \textprod_{i=1}^N (1- w_i^{-1}) | = \calO \big( (\epsilon')^{-L} \epsilon_2^{N} \big)
    = \calO \big( (\epsilon')^{-L + \beta_2 N}  \big) \ll 1.
  \end{align*}
  Therefore, for the different cases we have
  \begin{equation}
    \begin{split}
      \label{e:p_z_expansion}
      \frac{1}{ p_z(\vec{w} ;\zeta )} &= \frac{1}{1 + (-1)^N \zeta^L z^{-L} \prod_{i=1}^N (1- w_i^{-1})}\\
      &=
      \begin{cases}
        \sum_{k =0}^{\infty} \left( (-1)^{N-1} \zeta^L z^{-L} \prod_{i=1}^N (1- w_i^{-1}) \right)^k, &  \mbox{ in (a) and (c)}, \\
        - \sum_{k=1}^{\infty} \left( (-1)^{N-1} \zeta^L z^{L} \prod_{i=1}^N (1 - w_i^{-1})^{-1} \right)^k, & \mbox{ in (b)}.
      \end{cases}
    \end{split}
  \end{equation}
  On the other hand, due to the conditions \eqref{e:beta1_bounds}, $\beta_1 > \tfrac{1}{1-\rho}$ and $\beta_2 > d$ in \eqref{e:beta2_bounds},
  the series expansion for $1/q_z(w_i)$ depends only on whether $w_i \in C_R$, $w_i \in C_{\epsilon_1}$ or  $w_i \in 1 + C_{\epsilon_2}$, independently of $z \in C_{R'}$ or $z \in C_{\epsilon'}$.
  We have
  \begin{equation}\label{e:q_z_expansion}
    \begin{split}
      \frac{1}{q_z(w_i)} &= \frac{1}{1 - z^{L} w_i^{-L}(1-w_i^{-1})^{-N}}\\
      &=
      \begin{cases}
        \sum_{k=0}^{\infty} \left( z^{L} w_i^{-L}(1-w_i^{-1})^{-N} \right)^k, & w_i \in C_{R}, \\
        - \sum_{k=1}^{\infty} \left( z^{-L} w_i^{L}(1-w_i^{-1})^{N} \right)^k, & w_i \in C_{\epsilon_1} \cup (1+ C_{\epsilon_2} ),
      \end{cases}
    \end{split}
  \end{equation}
  for all $i = 1, \cdots, N$.

  First, consider case (c). The series expansion of the integrand $\CI(\epsilon', \emptyset, \emptyset, [N])$ only has negative powers of $z$, due to \eqref{e:p_z_expansion} and \eqref{e:q_z_expansion}. Thus, the corresponding contour integral vanishes after taking the contour integral with respect to the $z$-variable. This establishes case (c).

  Now, we treat the remaining cases cases (a) and (b) using a similar method. We illustrate the idea for case (a) below and omit the proof in case (b). 
  
  For case (a) with $I = \emptyset$, one may also show that the contour integral vanishes since there are only negative powers of $z$ after a series expansion, as was shown in case (c). Then, in the argument below, we consider case (a) with $I \neq \emptyset, [N]$, where the second case is omitted by the assumption of the lemma.

  Using \eqref{e:p_z_expansion} and \eqref{e:q_z_expansion}, we expand the denominator in the integrand of $\CI(R', I, J, K)$ as follows
  \begin{align*}
    & \frac{1}{ p_z(\vec{w} ; \zeta) \prod_{i=1}^N q_z(w_i) } \\
    & = \sum_{k \geq 0} \left( (-1)^{N-1} \zeta^L z^{-L} \prod_{i=1}^N (1 - w_i^{-1}) \right)^k
      \prod_{i \in I} \sum_{n_i \geq 0}  \left( z^{L} w_i^{-L}(1-w_i^{-1})^{-N} \right)^{n_i}
      \prod_{j \in J \sqcup K} \sum_{m_j \geq 1}  \left( z^{-L} w_j^{L}(1 - w_j^{-1})^{N} \right)^{m_j}, \\
    & = \sum_{\substack{k \geq 0 \\ n_i \geq 0, i \in I \\ m_j \geq 1, j \in J \sqcup K}}
    \left( (-1)^{N-1} \zeta^{k L} \prod_{i=1}^N (1 - w_i^{-1}) \right)^k
    \left( w_i^{-L}(1 - w_i^{-1})^{-N} \right)^{n_i}
    \left( w_j^{L}(1 - w_j^{-1})^{N} \right)^{m_j}  z^{L(\sum n_i - \sum m_j - k)}.
  \end{align*}
  Then, the coefficient $[z^0]$ reads,
  \begin{align}
    \label{e:TiS_TjS_expansion}
    \sum_{\substack{n_i \geq 0, i \in I \\ m_j \geq 1, j \in J \sqcup K}}
    \prod_{i \in I} (T_i S)^{n_i} \prod_{j \in J \sqcup K} (T_j S)^{-m_j} \mathds{1} \big( \sum n_i \geq \sum m_j \big),
  \end{align}
  where $T_i =  w_i^{-L}(1 - w_i^{-1})^{-N}$ and $S = (-1)^{N-1} \zeta^{k L} \prod_{i=1}^N (1 - w_i^{-1})$.
  Note that $|T_i S| < 1$ for $i \in I$ and $|T_j S| > 1$ for $j \in J \sqcup K$, where we obtain the inequalities by the assumption $I \neq \emptyset, [N]$ and by the bounds \eqref{e:beta1_bounds} and \eqref{e:beta1_bounds}. More precisely, the exponent in $\epsilon'$ of $T_j S$ for $j \in J \sqcup K$ writes
  \begin{equation}
    -(L-N) \beta_1 - |J| \beta_1 + |K| \beta_2
    \leq -(L-N) \beta_1 + (N-1) \beta_2
    < - \frac{L-N}{1-\rho} + \frac{L}{N-1} (N-1) = 0.
  \end{equation}
  In particular, this means that the series in\eqref{e:TiS_TjS_expansion} is absolutely converging.

  Now, we can bound the contour integral $CI(R', I, J, K)$ as follows, using the triangle inequality,
  \begin{align*}
    &|\CI(R', I, J, K)|\\
    & = \Bigg|
      \prod_{i \in I} \oint_{C_R} \frac{\dd w_i}{w_i}
      \prod_{j \in J} \oint_{C_{\epsilon_1}} \frac{\dd w_j}{w_j}
      \prod_{k \in K} \oint_{C_{1 + \epsilon_2}} \frac{\dd w_k}{w_k}
    \\
    & \hspace{20mm}\sum_{\substack{n_i \geq 0, i \in I \\ m_j \geq 1, j \in J \sqcup K}}
    \prod_{i \in I} (T_i S)^{n_i} \prod_{j \in J \sqcup K} (T_j S)^{-m_j} \mathds{1} \big( \sum n_i \geq \sum m_j \big)
    \cdot h_{\vec{w}} (Y) u_{\vec{w}} (X)
    \Bigg| \\
    & \leq
      \prod_{i \in I} \oint_{C_R} \Big| \frac{\dd w_i}{w_i} \Big|
      \prod_{j \in J} \oint_{C_{\epsilon_1}} \Big| \frac{\dd w_j}{w_j} \Big|
      \prod_{k \in K} \oint_{C_{1 + \epsilon_2}} \Big| \frac{\dd w_k}{w_k} \Big|
      \sum_{\substack{n_i \geq 0, i \in I \\ m_j \geq 1, j \in J \sqcup K}}
    \prod_{i \in I} |T_i S|^{n_i} \prod_{j \in J \sqcup K} |T_j S|^{-m_j} 
    \cdot h_{\vec{w}} (Y) u_{\vec{w}} (X) \\
    & =
      \prod_{i \in I} \oint_{C_R} \Big| \frac{\dd w_i}{w_i} \Big|
      \prod_{j \in J} \oint_{C_{\epsilon_1}} \Big| \frac{\dd w_j}{w_j} \Big|
      \prod_{k \in K} \oint_{C_{1 + \epsilon_2}} \Big| \frac{\dd w_k}{w_k} \Big|
      \prod_{i \in I} \frac{1}{1 - |T_i S|}
      \prod_{j \in J \sqcup K} \frac{1}{|T_jS| - 1}
      \cdot h_{\vec{w}} (Y) u_{\vec{w}} (X).
  \end{align*}

  We conclude by showing that the integral on the right side above is arbitrarily small. We choose any $j \in J \sqcup K \neq \emptyset$ and deform the contour in $w_j$.
  If $j \in J$ we deform the corresponding contour $C_{\epsilon_1}$ to $0$.
  Otherwise, if $j \in K$, we deform the corresponding contour $1+C_{\epsilon_2}$ to $1$.
  Note that, for any $\sigma \in S_N$ and $i = \sigma^{-1}(j)$, we have
  \begin{equation}
    \label{e:CI_wj_bigO}
    \begin{split}
      & |A_{\sigma}(\vec{w}) w_j^{y_j - x_i}| = \calO( w_j^{i-j + y_j -x_i} ) = \calO(w_j^{-(L-N)}), \quad\;
      | T_j S |^{-1}  = \calO(w_j^{L-N+1}), \quad \mbox{ if } |w_j| \to 0, \\
      & |A_{\sigma}(\vec{w}) w_j^{y_j - x_i}| = \calO( (w_j-1)^{j-i} ) = \calO(w_j^{-(N-1)}), \quad
      | T_j S |^{-1}  = \calO( (w_j-1)^{N-1}), \quad \mbox{ if } |w_j| \to 1,    
    \end{split}
  \end{equation}
  In the first part of \eqref{e:CI_wj_bigO}, we use $i-1 \leq x_i, y_i \leq L-N+i-1$ for any $i = 1, \cdots, N$.
  This shows that by shrinking the contour corresponding to $w_j$ to 0 or to 1, depending on whether $j \in J$ or $j \in K$, the integral $\CI(R', I, J, K)$ vanishes.
\end{proof}

\begin{rem}
  In this proof for \Cref{l:CI_zero}, we are also restricted to the case $2N<L$ due to the inequality in case (a) and the series expansion in \eqref{e:q_z_expansion}. As in \Cref{l:contour_deformation_induction}, we may also extend the result of \Cref{l:CI_zero} to the case $2N = L$ by taking \eqref{e:conditions_p} for the radius of the contours. See \Cref{r:conditions_p}.
\end{rem}

We now prove \Cref{p:delta_contour_formula}.

\begin{proof}[Proof of \Cref{p:delta_contour_formula}]
  Let $\mathrm{CI}$ denote the $(N+1)$-fold contour integral on the right side of \eqref{e:initial_condition}. Then, we write
  the nested contour integral $\mathrm{CI}$ as a linear combination of $2 \cdot 3^N$ nested contour integrals,
  \begin{equation}
    \label{e:CI_initial_decomposition_v2}
    \CI = \sum_{I \sqcup J \sqcup K = [N] } \big[ \CI(R', I, J, K) - \CI(\epsilon', I, J, K) \big],
  \end{equation}
  with the terms in the sum given by \eqref{e:CI_R_epsilon_v2}. Note that $I$ (resp. $J$ and $K$) denotes the set of indices $1 \leq i \leq N$ such that $w_i \in C_R$ (resp.~$w_i \in C_{\epsilon_1}$ and $w_i \in 1 + C_{\epsilon_2}$). 
  
  We now evaluate each of the nested contour integrals $\CI(r,I, J, K )$. By \Cref{l:CI_zero}, we have
  \begin{equation}
    \CI(r, I, J , K) = 0
  \end{equation}
  if $(r, I, J, K) \neq (R, [N] , \emptyset, \emptyset)$.
  Additionally, by \Cref{l:CI_delta_function}, we have
  \begin{equation}
    \CI (R, [N] , \emptyset, \emptyset) = \mathds{1}(X=Y). 
  \end{equation}
  Then, it follows that $\CI = \mathds{1}(X=Y)$ as claimed.
\end{proof}

\section{Properties of Bethe eigenfunctions}

\label{app:Bethe_properties}

In this appendix, we give the technical lemmas needed for the proof of \Cref{p:Bethe_Ansatz}. We show that the Bethe function $u_{\vec{\lambda}}$, given by either \eqref{e:bethe_function_zeta} or \eqref{e:bethe_function}, satisfies the exclusion/push rule \eqref{e:exclusion} and the periodicity condition \eqref{e:periodicity} if $\vec{\lambda} \in \bbC^N$ is a solution of the $\zeta$-deformed Bethe equations \eqref{e:bethe_equations}.

\begin{lemma}
  \label{l:exclusion_check}
  Under the hypothesis of \Cref{p:Bethe_Ansatz}, the Bethe eigenfunctions $X \mapsto u_{\vec{\lambda}}(X; \zeta)$ satisfy the exclusion rule given in \eqref{e:exclusion}.
  That is,
  \begin{equation}
    \label{e:exclusion_appendix}
    u_{\vec{\lambda}}(X; \zeta)_{x_k = x_{k-1}} = \zeta^{-1} \cdot u_{\vec{\lambda}}(X; \zeta)_{x_k = x_{k-1} + 1},
  \end{equation}
  for any $2 \leq k \leq N$ and $0 \leq x_1 < \cdots < x_{k-1} < x_{k+1} < \cdots < x_N < L$.
\end{lemma}

\begin{proof}
  The difference between the left and right side of \eqref{e:exclusion_appendix} gives
  \begin{equation}\label{e:exc_diff1}
    u_{\vec{\lambda}}(X; \zeta)_{x_k = x_{k-1}} - \zeta^{-1} \cdot u_{\vec{\lambda}}(X; \zeta)_{x_k = x_{k-1} + 1}
    = \sum_{\sigma \in S_N}  \left( 1 - \zeta^{-1}\lambda_{\sigma(k)}^{-1} \right) A_\sigma( \vec{\lambda}; \zeta ) \prod_{i=1}^N  \lambda_{\sigma(i)}^{-x_i} ,
  \end{equation}
  using \eqref{e:bethe_function_zeta} for the Bethe function with $x_k = x_{k-1}$ on the right side of the equation above.
  
  We rewrite the previous equation by relabelling the indices by an action of the symmetric group. More precisely, we consider the change of labels $\sigma \mapsto \sigma \circ \tau$ so that $\tau= (k-1, k)$ is a transposition. This leads to an alternate expression for the difference between the left and right side of \eqref{e:exclusion_appendix},
  \begin{align}
    u_{\vec{\lambda}}(X; \zeta)_{x_k = x_{k-1}} - \zeta^{-1} \cdot u_{\vec{\lambda}}(X; \zeta)_{x_k = x_{k-1} + 1} &=\sum_{\sigma \in S_N}  \left( 1 - \zeta^{-1}\lambda_{\sigma\circ\tau(k)}^{-1} \right) A_{\sigma \circ \tau} ( \vec{\lambda}; \zeta ) \prod_{i=1}^N \lambda_{\sigma \circ \tau(i)}^{-x_i} \nonumber \\
                                                                                                                   & = \sum_{\sigma \in S_N}  \left( 1 - \zeta^{-1}\lambda_{\sigma(k-1)}^{-1} \right) A_{\sigma \circ \tau} ( \vec{\lambda}; \zeta ) \prod_{i=1}^N \lambda_{\sigma(i)}^{-x_i} . \label{e:exc_diff2}
  \end{align}
  Then, we take the sum of the two expressions \eqref{e:exc_diff1} and \eqref{e:exc_diff2} to compute the difference between the left and right side of \eqref{e:exclusion_appendix}. We have
  \begin{equation*}
    \begin{split}
      &2 \left( u_{\vec{\lambda}}(X; \zeta)_{x_k = x_{k-1}} - \zeta^{-1} \cdot u_{\vec{\lambda}}(X; \zeta)_{x_k = x_{k-1} + 1}\right)\\
      &=\sum_{\sigma \in S_N} \left[ \left( 1 - \lambda_{\sigma(k), \zeta}^{-1} \right) A_\sigma( \vec{\lambda}; \zeta )   +
        \left( 1 - \lambda_{\sigma (k-1), \zeta}^{-1} \right) A_{\sigma \circ \tau} ( \vec{\lambda}; \zeta )  \right]\prod_{i=1}^N \lambda_{\sigma(i)}^{-x_i}\\
      &= 0,
    \end{split}
  \end{equation*}
  where the last equality is due to the identity
  \begin{equation}
    \label{e:Bethe_coeff_ratio}
    \frac{ A_{\sigma} ( \vec{\lambda}; \zeta ) }{ A_{\sigma \circ \tau} ( \vec{\lambda}; \zeta ) }
    = - \frac{ 1 - \lambda_{\sigma (k-1), \zeta}^{-1} }{ 1 - \lambda_{\sigma (k), \zeta}^{-1} }.
  \end{equation}
  that readily follows from the definition of the $A_\sigma$ coefficient given in \eqref{e:bethe_coeffs_zeta}. This establishes the result.
\end{proof}

\begin{lemma}
  \label{l:periodicity_check}
  Under the hypothesis of \Cref{p:Bethe_Ansatz}, the Bethe eigenfunctions $X \mapsto u_{\vec{\lambda}}(X; \zeta)$ satisfy the periodicity rule.
  In other words,
  \begin{equation}
    \label{e:periodicity_appendix}
    u_{\vec{\lambda}}(x_1, \cdots, x_{N}; \zeta) = u_{\vec{\lambda}}(x_2, \cdots, x_{N}, x_{1}+L; \zeta)
  \end{equation}
  for all $ x_1 < \cdots < x_{N-1} < x_N <x_1+ L$.
\end{lemma}

\begin{proof}
  Let $X = (x_1, \cdots, x_N)$ and write $s \cdot X = (x_2, \cdots, x_{N} , x_1 +L)$. Then, using \eqref{e:bethe_function_zeta}, we have
  \begin{equation}
    \begin{split}
      u_{\vec{\lambda}} (s \cdot X;\zeta) &=  \sum_{\sigma \in S_N} A_{\sigma}(\vec{\lambda}; \zeta) \lambda_{\sigma(N)}^{- x_1 -L} \prod_{i=1}^{N-1} \lambda_{\sigma(i)}^{- x_{i+1}} =  \sum_{\sigma \in S_N} A_{\sigma \circ \nu}(\vec{\lambda};\zeta) \lambda_{\sigma\circ \nu (N)}^{- x_1 -L} \prod_{i=1}^{N-1} \lambda_{\sigma\circ \nu(i)}^{- x_{i+1}}\\
      &= \sum_{\sigma \in S_N} A_{\sigma \circ \nu}(\vec{\lambda};\zeta) \lambda_{\sigma (1)}^{-L} \prod_{i=1}^{N-1} \lambda_{\sigma(i)}^{- x_{i}}\\
      u_{\vec{\lambda}} (X;\zeta) &= \sum_{\sigma \in S_N} A_{\sigma }(\vec{\lambda};\zeta) \prod_{i=1}^{N-1} \lambda_{\sigma(i)}^{- x_{i}}
    \end{split}
  \end{equation}
  where in the first formula, we make the change of indices $\sigma \mapsto \sigma \circ \nu$ with $\nu = (1, 2, \dots, N)$. Note that
  \begin{equation}
    \frac{A_{\sigma \circ \nu}(\vec{\lambda};\zeta)}{A_{\sigma }(\vec{\lambda};\zeta)} = (-1)^{N-1} \prod_{i=1} \frac{1 - \left(\zeta \lambda_{\sigma(i)}\right)^{-1}}{1 - \left(\zeta \lambda_{\sigma(1)}\right)^{-1}},
  \end{equation}
  which follows directly from the definition of the $A_\sigma$ coefficient given in \eqref{e:bethe_coeffs_zeta}. Then, we have
  \begin{equation}
    \begin{split}
      u_{\vec{\lambda}} (X;\zeta) - u_{\vec{\lambda}} (s \cdot X;\zeta) = \sum_{\sigma \in S_N} \left(1 + (-1)^N \lambda_{\sigma(1)}^{-L} \prod_{i=1} \frac{1 - \left(\zeta \lambda_{\sigma(i)}\right)^{-1}}{1 - \left(\zeta \lambda_{\sigma(1)}\right)^{-1}} \right)A_{\sigma }(\vec{\lambda};\zeta) \prod_{i=1}^{N-1} \lambda_{\sigma(i)}^{- x_{i}}=0
    \end{split}
  \end{equation}
  where the last equality is due to the $\zeta$-deformed Bethe equation \eqref{e:bethe_equations_zeta}. This establishes the result.
\end{proof}

\section{Bethe roots and Fuss--Catalan numbers}
\label{s:Bethe_FC}

In this part of Appendix, we introduce and show some properties of Fuss--Catalan numbers that will be useful later for residue computations in \Cref{a:residue_computations}.
The generating function of these numbers appear naturally as the series expansion for the Bethe roots on the right, i.e.~the roots with behavior $\lambda_i(z) \rightarrow 1$ when $z \to 0$; see \Cref{l:lambda_expansion}.

Let us introduce the $m$-th $(p, r)$-Fuss--Catalan number given by $A_0(p,r) =1$ and 
\begin{equation}\label{e:Afc}
  A_m(p, r) := \frac{r}{m!} \prod_{i=1}^{m-1} (m p + r -i)
\end{equation}
for $m \in \bbZ_{\geq 1}$ where $p, r \in \bbQ$.
Note that we may write 
$A_m(p, r) = \frac{r}{mp+r} {mp+r \choose m}$ if $p, r \in \bbZ_{\geq 0}$ and $m p +r \neq 0$.
Additionally, let us define the generating function for the Fuss--Catalan numbers.
\begin{equation}\label{e:fc}
  B_{p, r} (z) = \sum_{m \geq 0} A_m(p, r) z^m.
\end{equation}

\begin{lemma} \label{l:Bgf_relations}
  The generating functions $B_{p, r}(z)$ satisfy the following relations,
  \begin{align}
    B_{p, 1} (z) & = 1 + z B_{p, p}(z), \label{e:Bgf1} \\
    B_{p, 1} (z)^r & = B_{p, r} (z), \label{e:Bgf2} \\
    B_{p, r}(z) & = [1 + z B_{p, r}(z)^{p/r}]^r. \label{e:Bgf3}
  \end{align}
  where $p, r \in \bbQ$.
  In particular, one has $B_{p, p}(z) = [1 + z B_{p, p}(z) ]^p$.
\end{lemma}

\begin{proof}
  The first two identities, \eqref{e:Bgf1} and \eqref{e:Bgf2}, follow from the following identities:
  \begin{equation}
    A_{m}(p,p) = A_{m+1}(p,1), \quad \sum_{k=0}^m A_{k}(p,r) A_{m-k}(p, s) = A_m(p,r+s).
  \end{equation}
  The first identity above is a straight forward computation.
  The second identity is proved in \cite[Section 4.5]{Riordan1968} and in \cite[Proposition 2.1]{Mlotkowski2010} in a more general setting. The third identity, \eqref{e:Bgf3}, follows from the other two identities, \eqref{e:Bgf1} and \eqref{e:Bgf2}.
\end{proof}

Now, in \Cref{l:lambda_expansion}, we give a series expansion for the deformed Bethe roots, using the generating series for the Fuss--Catalan numbers, when $|z|$ is small and the root is near one.

\begin{lemma}
  \label{l:lambda_expansion}
  Fix $z \in \bbC$ such that $|z| < r_0 = \rho^{\rho}(1-\rho)^{1-\rho}$, with $\rho= N/L$, and
  let $\lambda \in \calQ_1(z)$ be a solution of \eqref{e:dBE1} such that $\lambda$ tends to 1 when $z$ tends to 0, c.f.~\eqref{e:q_small_sol}.
  Let $Z = - \eta z^d$, where $\eta$ is an $N$-th root of unity and $d=1/\rho$.
  Then,
  \begin{equation}
    \label{e:lambda_expansion}
    1 - \lambda^{-1} = - Z \cdot \varphi(Z),
  \end{equation}
  where $\varphi = B_{d, d}$ is the generating function of the $(d, d)$-Fuss--Catalan numbers, given by \eqref{e:fc}. That is,
  \begin{equation}
    \label{e:Fuss--Catalan-gf}
    \varphi(Z) = \sum_{m \geq 0} A_m(d, d) Z^m, \quad \mbox{ where } A_m(d, d) = \frac{1}{m+1} {(m+1) d \choose m}.
  \end{equation}
  Moreover, the radius of convergence of $\varphi(Z)$ is given by $r_0^d = \frac{(d-1)^{d-1}}{d^d} = \rho (1-\rho)^{d-1}$ and $\varphi(Z)$ converges on \emph{all points} on the boundary of its disk of convergence.
\end{lemma}

\begin{proof}
  We want to write a series expansion for $1 - \lambda^{-1}$ in terms of $z$, where $\lambda$ is one of the roots given in the statement.
  It can be easily shown that the series expansion takes the form
  \begin{equation}
    1 - \frac1\lambda = \sum_{k \geq 1} a_k z^{dk},
  \end{equation}
  where $a_1$ is an $N$-th root of unity. Consider the following change of variables: $Z = -a_1 z^d$ and $\varphi = -Z^{-1} (1 - \lambda^{-1})$.
  Then, the deformed Bethe equation \eqref{e:dBE1} can be rewritten as
  \begin{equation*}
    \varphi = -Z^{-1} (1 - \lambda^{-1}) = [1 - (1 - \lambda^{-1}) ]^d = (1 + Z \varphi)^d,
  \end{equation*}
  from which we recognize that $\varphi$ is the generating function of the $(d, d)$-Fuss--Catalan numbers, by \Cref{l:Bgf_relations}.

  To compute the radius of convergence of $\varphi$, we use the ratio test and get
  \begin{align*}
    \frac{A_{m+1}(d, d)}{A_m(d, d)} \longrightarrow \frac{d^d}{(d-1)^{d-1}} = r_0^{-d}.
  \end{align*}
  To show its convergence at all points of the boundary of the disk of convergence, we need to get a finer analysis than the ratio test.
  First, we rewrite $\varphi$ in the following way,
  \begin{equation}
    \label{e:mu_normalized}
    \varphi(Z) = \sum_{m \geq 0} A_m(d, d) r_0^{md} \cdot (\tfrac{Z}{r_0^d})^m = \sum_{m \geq 0} a_m b_m,
  \end{equation}
  where we denote $a_m = A_m(d, d) r_0^{md}$ and $b_m = (\tfrac{Z}{r_0^d})^m$.
  The ratio test at the next order gives,
  \begin{align*}
    \frac{a_{m+1}}{a_m} = \frac{A_{m+1}(d, d) r_0^{(m+1)d}}{A_m(d, d) r_0^{md}} \sim 1 - \frac{3}{2m}.
  \end{align*}
  The above asymptotics implies that $a_m \sim m^{-3/2}$ so the series $\sum a_m$ converges (absolutely).
  We can apply Abel's test to \eqref{e:mu_normalized} since $(a_m)_{m \geq M}$ is decreasing for $M$ large enough and converges to 0, which shows that the series converges if $|Z| = r_0^d$ and $Z \neq r_0^d$.
  At the point $Z = r_0^d$, the absolute summability of $\sum a_m$ gives the convergence of $\varphi$.
\end{proof}

\begin{lemma}
  \label{l:sum_ln}
  Let $z \in \bbC$ such that $|z| \leq r_0^d= \rho (1-\rho)^{d-1}$ with $\rho = 1/d = L/N$ . Define the symmetric function $\Psi(z) = \prod_{k=1}^N \varphi(Z_k)$, where $\varphi$ the generating function defined in \eqref{e:lambda_expansion}. Additionally, set $Z_k = - \eta_k z^d$ so that $\eta_k$'s all the distinct $N$-th roots of unity. Then, we have
  \begin{align}
    \ln \Psi(z) & = \sum_{i=1}^N \sum_{n \geq 1} \frac{1}{n} {dn \choose n} Z_i^n
                  = \sum_{k \geq 1} \frac{(-1)^{kN}}{k} {dkN \choose kN} z^{dkN}, \label{e:ln_phi} \\
    \Psi(z) & = 1 + (-1)^N {L \choose N} z^L + \calO(z^{2L}). \label{e:phi}
  \end{align}

\end{lemma}

\begin{proof}
  Using \Cref{l:log_Bgf} and \eqref{e:Bgf2}, we have
  \begin{align*}
    \ln \Psi(z) = \sum_{i=1}^N \ln \varphi(Z_i)
    = \sum_{i=1}^N \sum_{n \geq 1} \frac{1}{n} {dn \choose n} Z_i^n
    = \sum_{k \geq 1} \frac{(-1)^{kN}}{k} {dkN \choose kN} z^{dkN}.
  \end{align*}
  where in the last equality we use the following identity for the $N$-the roots of unity,
  \begin{align*}
    \sum_{i=1}^N Z_i^n = \sum_{i=1}^N \eta_i^n z^{d n} = \left\{
    \begin{array}{ll}
      0 & \mbox{ if } n \nmid N, \\
      (-1)^n z^{dn} N & \mbox{ otherwise.}
    \end{array}
                        \right.
  \end{align*}

\end{proof}

\begin{lemma} \label{l:log_Bgf}
  Let $d$ be a positive integer.
  The logarithm of the generating function $B_{d, 1}(z)$ has the following series expansion,
  \begin{equation}
    \label{e:log_Bgf}
    \ln B_{d, 1}(z) = \sum_{n \geq m \geq 1} (-1)^{m-1} \frac{1}{n} {dn \choose n-m} z^n
    = \sum_{n \geq 1} \frac{1}{dn} {dn \choose n} z^n.
  \end{equation}
\end{lemma}

The proof given here is essentially the proof from \cite{Prodinger}.
In the original proof, there was a mistake in the binomial coefficients that we correct here.

\begin{proof}
  Let $\alpha \in \bbR$.
  Write $B_{d, 1}(z) = 1 + u(z)$.
  Consider the exponential generating function of $\ln B_{d, 1} (z)$,
  \begin{align}
    F(z, \alpha) = \sum_{p \geq 0} \frac{\alpha^p}{p!} ( \ln B_{d, 1}(z) )^p
    = ( B_{d, 1} (z) )^\alpha
    = \sum_{m \geq 0} {\alpha \choose m} u(z)^m.
  \end{align}
  Then, the result follows from computing the coefficient $[\alpha^1] F(z, \alpha)$.
  
  First, note that the binomial coefficient gives
  \begin{align}
    [\alpha^1] {\alpha \choose m} = \frac{(-1)^{m-1}}{m}.
  \end{align}
  Next, we want to expand $u(z)$ into a series.
  Equations \eqref{e:Bgf1} and \eqref{e:Bgf2} give $z = \frac{u}{(1+u)^d}$.
  We can apply the Lagrange-Bürmann inversion formula \eqref{e:LB2} to $\phi(w) = (1 + w)^d$ and $H(w) = w^m$.
  For $n \geq m$,
  \begin{align*}
    [z^n] u(z)^m = [z^n] H(u(z)) = \frac{1}{n} [w^{n-1}] ( m w^{m-1} \phi(w)^n )
    = \frac{m}{n} {dn \choose n-m}.
  \end{align*}
  The second equality of \eqref{e:log_Bgf} is due to the following equality for any positive integer $A$,
  \begin{align*}
    \sum_{m=0}^n (-1)^m {A \choose m} = (-1)^n {A-1 \choose m}.
  \end{align*}
\end{proof}

\begin{lemma}[Lagrange-Bürmann inversion formula]
  \label{l:LB_formula}
  Let $\phi$ and $H$ be arbitrary analytic functions such that $\phi(0) \neq 0$.
  Let $f(w) = w / \phi(w)$ and $g$ be the inverse of $f$, that is $f(g(z))=z$.
  Then the coefficients of $g(z)$ and $H(g(z))$ are given as follows
  \begin{align}
    [z^n] g(z) & = \frac{1}{n} [w^{n-1}] \phi(w)^n, \label{e:LB1} \\
    [z^n] H(g(z)) & = \frac{1}{n} [w^{n-1}] H'(w) \phi(w)^n, \label{e:LB2}
  \end{align}
\end{lemma}
This is a classical result, and its proof is based on contour integral; see this modern reference \cite[Section 5.4]{Stanley2023} for instance.

\section{Location of Bethe roots}

First, let us state some results about solutions to the deformed Bethe equations \eqref{e:dBE1} and \eqref{e:dBE2}.
\begin{lemma}\label{l:bethe_roots_location}
  Fix positive integers $L \geq N \geq 1$ and denote the ratio by $\rho = N / L$.
  Let $r_0 := \rho^\rho (1-\rho)^{1-\rho}$ and $\zeta \in \bbU$.
  Then, any solution $z$, $w_1, \dots, w_N$ to the decoupled Bethe equations \eqref{e:dBE1} and \eqref{e:dBE2}, with $w_i$'s pairwise distinct, needs to satisfy the following properties.
  \begin{enumerate}
    \item The product of the roots, $w_1 \cdots w_N$, is a $L$-th root of unity.
    \item The following inequality holds $|z| \leq |w_i| + \rho$ for all $1 \leq i \leq N$. In particular, $|z| \leq 1 + \rho$.
    \item There exists a small enough $\epsilon > 0$ (depending on $L$, $N$ and $\zeta$) such that there is no solution when $0 < |z| < \epsilon$.
  \end{enumerate}
\end{lemma}

\begin{proof}
  \begin{enumerate}
    \item It is a direct consequence of \eqref{e:dBE1} and \eqref{e:dBE2} by taking the product of \eqref{e:dBE1} over all $i=1, \cdots, N$ and then dividing by the $N^{th}$ power of \eqref{e:dBE2}.
    \item This can be obtained by convexity,
    \begin{align*}
      |z|^d = |w_i|^d |1 - w_i^{-1}| \leq |w|^d (1 + |w_i|^{-1})
      \leq |w_i|^d \left( 1 + \frac{1}{d|w_i|} \right)^d = (|w_i| + \rho)^d.
    \end{align*}
    \item Let us proceed in two steps. First, given solution to the deformed Bethe equations  with $|z| < r_0$, we show that all the $w_i$'s should belong to the set $\calQ_1(z)$ for small enough $|z|$.
    Then, we may use \Cref{l:lambda_expansion} to deduce more information on the solutions.

    Let $z$, $w_1, \dots, w_N$ be a solution to the decoupled Bethe equations.
    Assume that $w_1, \dots, w_\ell \in \calQ_0(z)$ and $w_{\ell+1}, \dots, w_N \in \calQ_1(z)$ for some $1 \leq \ell \leq N$.
    Then, we have
    \begin{align*}
      |z|^L = \prod_{i=1}^N |1 - w_i^{-1}| = \prod_{i=1}^N |w_i - 1|
      \geq \rho^\ell \prod_{i=\ell+1}^N \frac{ |z|^d }{ |w_i|^{d-1} },
    \end{align*}
    which is equivalent to
    \begin{align*}
      \Big( \frac{|z|^d}{\rho} \Big)^\ell \geq \prod_{i=\ell+1}^N \frac{1}{ |w_i|^{d-1} }.
    \end{align*}
    The left side of the above inequality tends to 0 whereas the right side tends to 1 when $z \to 0$.
    This shows that we need to have $\ell = 0$, meaning that when $|z|$ is small enough, all the $w_i$'s need to be in $\calQ_1(z)$.

    Now, we may choose $\epsilon > 0$, depending on $N$ and $L$, so that when $|z| < \epsilon$, the only solutions to the decoupled Bethe equations have all the $w_i$'s in $\calQ_1(z)$.
    Since $\calQ_1(z)$ contains exactly $N$ elements, a solution with all the $w_i$'s distinct means that $w_i$'s are a permutation of elements in $\calQ_1(z)$.
    By \Cref{l:lambda_expansion}, we may write
    \[
      1 - w_i^{-1} = -Z_i \varphi(Z_i), \quad Z_i = -\eta_i z^d,
    \]
    where $\eta_i$'s are distinct $N$-th roots of unity and $\varphi$ is the $(d, d)$-Fuss--Catalan generating function as given by \eqref{e:Fuss--Catalan-gf} in \Cref{s:Bethe_FC}.
    Using the fact that $p_z(\vec{w}; \zeta) = 0$, we find
    \begin{align*}
      1 & = (-1)^{N-1} \zeta^L z^{-L} \prod_{i=1}^N \big( -Z_i \varphi(Z_i) \big)
          = \zeta^L \prod_{i=1}^N \varphi(Z_i)
          = \zeta^L \Psi(z),
    \end{align*}
    where $\Psi$ is analytic for $|z| < r_0$ with expansions given in \Cref{e:ln_phi} and \Cref{e:phi}.

    When $\zeta^L = 1$, we clearly see that $z=0$ is a solution to the above equation.
    Due to the analyticity of $\Psi$, we know that the solutions need to be isolated, so we can find a small enough $\epsilon > 0$ so that there is no solution with $0 < |z| < \epsilon$.
    When $\zeta^L \neq 1$, we clearly see that $z=0$ is not a solution to the above equation.
    By continuity of $\Psi$, we can also find a small enough $\epsilon > 0$ so that there is no solution with $0 \leq |z| < \epsilon$.
  \end{enumerate}
\end{proof}

\section{Residue computations}
\label{a:residue_computations}

\subsection{Completion of the proof of Lemma \ref{l:u0_residue}}

Recall the definition from \Cref{e:disjoint_Q1} of the following set for small enough $z \in \bbC^*$,
\begin{align*}
  \calQ_1^\neq (z) & := \{ (\lambda_i)_{1 \leq i \leq N-1} \in \bbC^{N-1} \mbox{ such that } \lambda_i \in \calQ_1(z) \mbox{ are pairwise distinct} \}.
\end{align*}

In this part of Appendix, we complete the proof of \Cref{l:u0_residue} where we need to estimate, under the condition $\zeta^L = 1$, a limit for any $\vec{\lambda} = (\lambda_i)_{1 \leq i \leq N-1} \in \calQ_1^\neq (z)$, see \Cref{l:limit_ratio}, and compute a summation over the set $\calQ_1^\neq (z)$, see \Cref{l:sum_permutations}.

For any $\vec{\lambda} = (\lambda_i)_{1 \leq i \leq N-1} \in \calQ_1^\neq (z)$, we recall the definition of $\mu$,
\begin{equation}
  \label{e:mu_invs_def}
  \mu^{-1} = 1 + (-1)^N z^L \prod_{i=1}^{N-1} (1 - \lambda_i^{-1}),
\end{equation}
originally given in \Cref{e:P2_pole_mu}. Since there are exactly $N$ Bethe roots that tend to 1 when $z$ tends to 0, we can complete the set of Bethe roots by writing $\lambda_N$ for the missing one.
Now, we can write $1 - \lambda_i^{-1} = - Z_i \varphi(Z_i)$ for all $1 \leq i \leq N$.
This allows us to rewrite the following product,
\begin{align*}
  \prod_{i=1}^{N-1} (1 - \lambda_i^{-1})
  & = (1-\lambda_N^{-1})^{-1} \prod_{i=1}^N (1-\lambda_i^{-1})
    = (1-\lambda_N^{-1})^{-1} \prod_{i=1}^N (- Z_i) \mu (Z_i) \\
  & = (-1)^{N-1} (1-\lambda_N^{-1})^{-1}  \Psi(z) z^L,
\end{align*}
where we use the identity $\prod_{i=1}^N Z_i = - z^{L}$ and \eqref{e:phi}.
Then, $\mu^{-1}$ can be rewritten as
\begin{equation}
  \label{e:mu_invs}
  \mu^{-1}
  = 1 + (-1)^N z^L \prod_{i=1}^{N-1} (1 - \lambda_i^{-1})^{-1} 
  = 1 - (1-\lambda_N^{-1}) \Psi(z)^{-1}
  = 1 + Z_N \varphi(Z_N) \Psi(z)^{-1}.
\end{equation}

\begin{lemma}
  \label{l:sum_permutations}
  Then, the following summation can be computed,
  \begin{align*}
    \lim_{z \to 0} \sum_{\vec{\lambda} \in \calQ_1 (z)} u_{ (\vec{\lambda}, w_N) } (X)
    = N^N.
  \end{align*}
\end{lemma}

\begin{proof}
  Due to the determinantal structure of $u_{ (\vec{\lambda}, w_N) }(X)$, we may take the summation is over the set $\calQ_1^\neq (z)$.
  Moreover, we have the following series expansion,
  \begin{align*}
    u_{ (\vec{\lambda}, w_N) }(X)
    & = \Bigg( \sum_{\sigma \in S_N} A_\sigma (\vec{w}) \prod_{i=1}^N w_\sigma(i)^{y_{\sigma(i)} - x_i} \Bigg)_{(w_1, \cdots, w_{N-1}) = (\lambda_1, \cdots, \lambda_{N-1}) } \\
    & = \sum_{\sigma \in S_N} (-1)^{\sigma} \prod_{i=1}^N \eta_{\sigma(i)}^{\sigma(i) - i} + \calO(z^d)
      = \det \big[ \eta_j^{j-i} \big] + \calO(z^d).
  \end{align*}
  The first equality is due to the following Taylor series
  \begin{align*}
    & 1 - \lambda_i^{-1} = - Z_i \varphi(Z_i) = \eta_i z^d + \calO(z^{2d}), \quad i = 1, \dots, N-1, \\
    & 1 - w_N^{-1} = -Z_N \varphi(Z_N) \Psi(z)^{-1} = \eta_N z^d + \calO(z^{2d})
  \end{align*}
  and  $A_\sigma(\vec{w}) = (-1)^{\sigma} \prod_{i=1}^N (1 - w_{\sigma(i)}^{-1})^{\sigma(i) -i}$.
  Then, we may conclude the computation
  \begin{align*}
    \sum_{\vec{\lambda} \in \calQ_1^\neq (z)} u_{ (\vec{\lambda}, w_N) }(X)
    & = \sum_{\vec{\lambda} \in \calQ_1^\neq (z)}  \det \big[ \eta_j^{j-i} \big] + \calO(z^d)
      = \sum_{(\eta_j) \mbox{ \scriptsize pairwise distinct}}  \det \big[ \eta_j^{j-i} \big] + \calO(z^d) \\
    & = \sum_{ \substack{\eta_j \in \bbU_N \\ 1 \leq j \leq N} }  \det \big[ \eta_j^{j-i} \big] + \calO(z^d)
    = \det \big[ \sum_{\eta \in \bbU_N} \eta^{j-i}  \big] + \calO(z^d) \\
    & = \det[ N \cdot \mathds{1} (i=j) ] + \calO(z^d) = N^N + \calO(z^d).
  \end{align*}

\end{proof}

\begin{lemma}
  \label{l:limit_ratio}
  Given $\vec{\lambda} = (\lambda_i)_{1 \leq i \leq N-1} \in \calQ_1^\neq (z)$ with distinct Bethe roots that tend to 1 when $z$ tends to zero.
  Define $w_N \in \bbC$ as in \eqref{e:mu_invs_def}
  Then, the following ratio is independent of the choice of $\vec{\lambda}$ in the limit,
  \begin{equation}
    \label{e:limit_ratio}
    \lim_{z \to 0} \Bigg( \frac{\prod_{i=1}^N (w_i -1) }{q_z(w_N)} \Bigg)_{\subalign{ & (w_1, \cdots, w_k) = (\lambda_1, \cdots, \lambda_{N-1}) \\  & \phantom{(} w_N = \mu }}
    = N^{-1} {L \choose N}^{-1}.
  \end{equation}
\end{lemma}

\begin{proof}
  We write the series expansion of the numerator and the denominator of the ratio.
  First, the numerator is the product of the two following terms,
  
  \begin{subequations}
    \label{e:ratio_numerator}
    \begin{align}
      (1 - \mu^{-1}) \prod_{i=1}^{N-1} (1 - \lambda_i^{-1})
      & = (-1)^{N-1} z^L, \\
      \mu \prod_{i=1}^{N-1} \lambda_i
      & = 1 + \calO(z^d).
    \end{align}
  \end{subequations}
  Then, we use \eqref{e:mu_invs} and the first decoupled Bethe equation \eqref{e:dBE1} to write
  \begin{align*}
    (1 - \mu^{-1})^{-N}
    = (1-\lambda_N^{-1})^{-N} \Psi(z)^{N} = z^{-L} \lambda_N^L \Psi(z)^N.
  \end{align*}
  Hence, the denominator has the following expansion
  \begin{equation}
    \label{e:ratio_denominator}
    q_z(\mu)
    = 1 - z^L \mu^{-L} (1 - \mu^{-1})^{-N}
    = 1 - \Big( \frac{\lambda_N}{\mu} \Big)^L \Psi(z)^N
    = (-1)^{N-1} N  {L \choose N} z^L + \calO(z^{L+d}),
  \end{equation}
  where we use \eqref{e:phi} and
  \begin{align*}
    \frac{\lambda_N}{\mu}
    & = \frac{1 + Z_N \varphi(Z_N) \Psi(z)^{-1} }{1 + Z_N \varphi(Z_N)}
      = 1 - \frac{Z_N \varphi(Z_N) (1 - \Psi(z)^{-1}}{1 + Z_N \varphi(Z_N)}
      = 1 + \calO(z^{L+d}).
  \end{align*}
  Finally, we obtain the desired result using the expansions \eqref{e:ratio_numerator} and \eqref{e:ratio_denominator}.
\end{proof}

\section{The double root does not solve the Bethe equation}
\label{a:last_appendix}

In this part of Appendix, we aim to show that the decoupled Bethe equations \eqref{e:dBE} do not have any solution if one of the $\lambda_i$'s equals $1-\rho$.
This completes the proof of \Cref{l:residue_computations}.

\begin{lemma}
  \label{l:lambdas_right_critical}
  Let $z \in \bbC^*$ and $\lambda_i = \lambda_i(z) \in \calQ(z)$, for $1 \leq i \leq N$.
  Additionally, if we assume 
  \begin{itemize}
    \item the $\lambda_i$'s are chosen to be pairwise distinct from $\calQ_1(z)$;
    \item there exists an index $i$ such that $\lambda_i = 1-\rho$,
  \end{itemize}
  then $|z|^L < \prod_{i=1}^N |1 - \lambda_i^{-1}|$.
\end{lemma}

\begin{proof}
  If one of the Bethe roots is known to be $1-\rho$, then $z$ needs to satisfy $z^L = (-1)^N r_0^L$, where we recall that $r_0 = \rho^\rho (1-\rho)^{1-\rho}$.
  Since the roots $\lambda_k$'s are chosen from $\calQ_1(z)$, which contains exactly $N$ elements, and need to be pairwise distinct, we may choose $\lambda_k$ such that $1 - \lambda_k^{-1} = - Z_k \varphi(Z_k)$ by \Cref{l:lambda_expansion}.
  Then,
  \begin{align*}
    \prod_{k=1}^N (1-\lambda_k^{-1}) = \prod_{k=1}^N \eta_k z^d \varphi(Z_k)
    = (-1)^{N-1} z^L \prod_{k=1}^N \varphi(Z_k).
  \end{align*}
  
  Using \Cref{l:log_Bgf} and \eqref{e:Bgf2}, we find
  \begin{align*}
    \prod_{i=1}^N \varphi(Z_i)
    & = \exp \left( \sum_{i=1}^N \ln \varphi(Z_i) \right)
      = \exp \left( \sum_{i=1}^N \sum_{n \geq 1} \frac{1}{n} {dn \choose n} Z_i^n \right) \\
    & = \exp \left( \sum_{k \geq 1} \frac{(-1)^{kN}}{k} {dkN \choose kN} z^{dkN} \right) 
      = \exp \left( \sum_{k \geq 1} \frac{1}{k} {dkN \choose kN} r_0^{kL} \right) > 1.
  \end{align*}
\end{proof}

\begin{proposition}
  \label{p:double_root_no_solution}
  Let $z \in \bbC^*$ and $\zeta \in \bbU$.
  Consider $\lambda_i = \lambda_i(z) \in \calQ(z)$, for $1 \leq i \leq N$.
  Additionally, if we assume
  \begin{itemize}
    \item the $\lambda_i$'s are chosen to be pairwise distinct;
    \item there exists an index $i$ such that $\lambda_i = 1-\rho$,
  \end{itemize}
  then $p_z(\vec{\lambda}; \zeta) \neq 0$.
\end{proposition}

\begin{proof}
  If all the $\lambda_i$'s are chosen from $\calQ_1(z)$, then we can conclude using \Cref{l:lambdas_right_critical}.

  Assume there are $n_l$ roots, denoted $\lambda_1, \cdots, \lambda_{n_l}$, chosen from $\calQ_0(z)$ and the remaining roots from $\calQ_1(z)$.
  We want to show by contradiction.
  We assume that $p_z(\vec{\lambda} ; \zeta) = 0$.
  Then, writing $\tilde{\lambda} = (\tilde{\lambda}_k)_{1 \leq k \leq N}$ where $\tilde{\lambda}_k = \lambda_k$ for $n_l +1 \leq k \leq N$ and $(\tilde{\lambda}_k)_{1 \leq k \leq n_l}$ are chosen from the other remaining roots of $\calQ_1(z)$, that is to say $\calQ_1(z) = \{ \tilde{\lambda}_k , 1 \leq k \leq N \}$.
  Since $\zeta$ is on the unit circle, we have
  \begin{align*}
    |z^L|
    & = \prod_{k=1}^N | 1-\lambda_k^{-1} | =  \prod_{k=1}^N | \lambda_k-1|
      > \rho^{n_l} \prod_{k=n_l+1}^N | \lambda_k-1|
      = \rho^{n_l} \prod_{k=n_l+1}^N \frac{|z|^d}{|\lambda_k|^{d-1}} \\
    & = \rho^{n_l} \prod_{k=1}^N \frac{|z|^d}{|\tilde{\lambda}_k|^{d-1}} \prod_{k=1}^{n_l} \frac{|\tilde{\lambda}_k|^{d-1}}{|z|^d}
      > \rho^{n_l} (1-\rho)^{n_l(d-1)} |z|^{-d n_l}  \prod_{k=1}^N \frac{|z|^d}{|\tilde{\lambda}_k|^{d-1}}.
  \end{align*}
  Since $|z|^d = r_0^d = \rho (1-\rho)^{d-1}$, the above is equivalent to
  \begin{align*}
    1 >  \prod_{k=1}^N \frac{1}{|\tilde{\lambda}_k|^{d-1}} 
    \Leftrightarrow 1 > \prod_{k=1}^N \frac{1}{|\tilde{\lambda}_k|^{d}}
    = \prod_{k=1}^N \frac{|1 - \tilde{\lambda}_k^{-1}| }{|z|^{d}}
    \Leftrightarrow |z|^L > \prod_{k=1}^N |1 - \tilde{\lambda}_k^{-1}|,
  \end{align*}
  which contradicts with \Cref{l:lambdas_right_critical}.
\end{proof}

\bibliographystyle{alpha}
\bibliography{PushASEP}

\end{document}